\documentclass[12pt,reqno,letterpaper]{amsart}

\usepackage[T1]{fontenc}
\usepackage{amsmath,amssymb,amsfonts,amsthm,mathtools}
\usepackage{enumitem}
\usepackage{xcolor}
\usepackage{geometry}
\usepackage{hyperref}

\hypersetup{
  colorlinks=true,
  linkcolor=blue,
  citecolor=blue,
  urlcolor=blue
}
\allowdisplaybreaks
\numberwithin{equation}{section}

\newtheorem{theorem}{Theorem}[section]
\newtheorem{proposition}[theorem]{Proposition}
\newtheorem{lemma}[theorem]{Lemma}
\theoremstyle{definition}
\newtheorem{definition}[theorem]{Definition}

\newcommand{\T}{\mathbb T}
\newcommand{\R}{\mathbb R}
\newcommand{\Z}{\mathbb Z}
\newcommand{\A}{\mathbb A}
\newcommand{\eps}{\varepsilon}
\newcommand{\la}{\lambda}
\newcommand{\supp}{\operatorname{supp}}
\newcommand{\ip}[2]{\left\langle #1,#2\right\rangle}
\newcommand{\norm}[1]{\left\lVert #1\right\rVert}
\newcommand{\abs}[1]{\left\lvert #1\right\rvert}

\title[Singular solutions for KP and modified KP equations]
{Non-unique singular solutions for KP and modified KP equations on
$\mathbb T^2$ and $\mathbb R^2$}

\author{Alexandru F. Radu}
\address{Simion Stoilow Institute of Mathematics, Romanian Academy,
Calea Grivitei Street, no.~21, 010702 Bucharest, Romania}
\email{alexandru.radu@imar.ro}

\begin{document}

\begin{abstract}
We construct infinitely many weak singular solutions with zero initial data
and compact time support for third- and fifth-order KP-I, KP-II, and their
modified counterparts on $\mathbb T^2$ and $\mathbb R^2$.  Their nonlinearities
are cutoff-independent, absolutely convergent Fourier convolutions.  Quadratic
solutions belong to $C_tL^p$ for $p<2$ and
$C_t(H^{-\sigma,0}\cap H^{-\sigma})$ for $\sigma>0$.  Modified solutions
belong to $C_tL^p$ for $p<3$.  One cubic family lies in $C_tH^\alpha$ for
$\alpha<1/3$; another has parabolic Fourier support and lies in $C_tH^{s,0}$
for $s<1/2$ and $C_tH^\alpha$ for $\alpha<1/4$.  The exponents $1/3$ and
$1/2$ are sharp at the $L^3$ product threshold.  For quadratic fifth-order
KP on $\mathbb R^2$, the nonuniqueness range is almost sharp.  We also
construct periodic stationary KP-I and KP-II solutions and prove that $L^2$
is the sharp threshold between singular stationary KP-I solutions and
smoothness.
\end{abstract}

\keywords{KP-I and KP-II equations, modified KP equations,
fifth-order KP equations, weak singular solutions, convex integration,
nonuniqueness, stationary solutions}
\subjclass[2020]{35Q53, 35A02, 35D30}

\maketitle

\section{Introduction}

\subsection{Equation and constraints}

For $r\in\{2,3\}$, $d\in\{3,5\}$, and
$\kappa\in\{-1,1\}$, on each of the domains
$\T^2=(\R/\Z)_x\times(\R/\Z)_y$ and $\R^2$, consider
\begin{equation}\label{eq:KP-family}
 \partial_tu+(-1)^{(d+1)/2}\partial_x^du
 +\kappa\partial_x^{-1}\partial_y^2u+\partial_x(u^r)=0.
\end{equation}
The choice $\kappa=-1$ gives KP-I, while $\kappa=1$ gives KP-II.  The
case $r=2$ is quadratic, and $r=3$ is the modified KP equation with cubic
nonlinearity.  Under $v=\sqrt3\,u$, the cubic term takes the form
$v^2\partial_xv$.  The integrable modified KP-II system also contains a
nonlocal quadratic term and is related to KP-II through the Miura transform~\cite{KenigMartel}.
We use this $\kappa$-label for both dispersion orders.

The third-order quadratic equation was introduced by Kadomtsev and
Petviashvili~\cite{KadomtsevPetviashvili}.  The standard
fifth-order quadratic convention follows by reflection: if $r=2$, $d=5$,
and $u$ solves \eqref{eq:KP-family}, then
\begin{equation}\label{eq:fifth-order-reflection}
 v(t,x,y)=-u(t,-x,y)
\end{equation}
solves
\begin{equation}\label{eq:fifth-order-divided}
 \partial_tv+\partial_x^5v
 -\kappa\partial_x^{-1}\partial_y^2v+\partial_x(v^2)=0.
\end{equation}
In the notation customary for the fifth-order equation,
$\delta=-\kappa$, with $\delta=1$ for KP-I and $\delta=-1$ for KP-II~\cite{SautTzvetkovFifth,Patterson5KP}.

On the torus we work in the zero $x$-mean class, on which
$\partial_x^{-1}$ is well defined:
\begin{equation}\label{eq:xmeanconstraint}
    \widehat u(t,0,m)=0
    \qquad\text{for every }m\in\Z.
\end{equation}
Applying $\partial_x$ formally gives the local equation
\begin{equation}\label{eq:localKP-kappa}
 u_{tx}+(-1)^{(d+1)/2}\partial_x^{d+1}u
 +\kappa u_{yy}+\partial_x^2(u^r)=0.
\end{equation}
The construction uses \eqref{eq:KP-family} together with
\eqref{eq:xmeanconstraint}, since \eqref{eq:localKP-kappa} loses an
arbitrary integration constant depending on $(t,y)$.
We write $C_tX=C([0,1];X)$ unless another time interval is displayed and set
\[
 \supp\widehat f
 :=\overline{\bigcup_{t\in[0,1]}\supp\widehat f(t)}.
\]
A fixed
$x$-frequency gap on $\R^2$ means that one constant $c_0>0$ satisfies
$\widehat u(t,\xi,\eta)=0$ for every $t$ whenever $\abs\xi<c_0$.
The notation $H^{s,0}$ places the Sobolev weight only on the
$x$-frequency, whereas $H^s$ denotes the isotropic space.

\subsection{Cauchy theory and singular solution classes}

For KP-I on $\R^2$, Ionescu--Kenig--Tataru proved global well-posedness in
the energy space~\cite{IKT}, while Guo--Molinet obtained unconditional
local well-posedness in $H^{s,0}$ for $s>3/4$ and unconditional global
well-posedness in the energy scale~\cite{GuoMolinet}; continuous local
well-posedness is known for $s>1/2$~\cite{Guo}.  The flow map is not $C^2$
at the origin~\cite{MolinetSautTzvetkovIllposed}.  On periodic domains,
Saut--Tzvetkov identified the resonant obstruction to Bourgain-space
iteration~\cite{SautTzvetkovPeriodic}, Zhang proved local well-posedness in
a Besov-energy space~\cite{Zhang}, and Kinoshita--Sanwal--Schippa showed
that the fully periodic equation is not semilinear in the sense of a $C^2$
flow map~\cite{KSS}.

The KP-II theory reaches lower regularity.  Bourgain constructed the
canonical $L^2$ flow on $\T^2$~\cite{BourgainKP}; on $\R^2$, Hadac proved
local well-posedness in $H^{s,0}$ for $s>-1/2$~\cite{Hadac}, and
Hadac--Herr--Koch reached the scaling-critical spaces $H^{-1/2,0}$ and
$\dot H^{-1/2,0}$~\cite{HHK}.  Herr--Schippa--Tzvetkov extended the
periodic canonical flow locally to $H^{s,0}$ for $s>-1/90$~\cite{HST}.
In these results, uniqueness is formulated in an auxiliary resolution
space or within a continuous extension of the canonical flow.  In the weak
singular class used here, absolute Fourier summability defines the nonlinear
term as a distribution and gives cutoff independence.

The fifth-order theory was developed by Saut--Tzvetkov~\cite{SautTzvetkovFifth}.  Robert proved global well-posedness for periodic
fifth-order KP-I in its natural energy space~\cite{RobertFifth}, and
Patterson proved unconditional uniqueness for fifth-order KP-I and KP-II
on $\R^2$~\cite{Patterson5KP}.  For the modified
equations, Kenig--Ziesler developed whole-space local theories for KP-I and
KP-II~\cite{KenigZiesler}, Gr\"unrock treated generalized KP-II
nonlinearities, including the cubic case, in nearly scaling-critical anisotropic spaces~\cite{Gruenrock}, and Bozgan studied third-order modified KP-I in periodic
geometries~\cite{BozganPeriodic} and fifth-order modified KP-I on $\R^2$
and $\R\times\T$~\cite{BozganFifthModified}.

Related definitions of nonlinear terms by convergent Fourier expansions
occur in several settings.
Lemari\'e--Rieusset uses a convergent double Fourier expansion for stationary
two-dimensional Navier--Stokes~\cite{LemarieRieusset}.
Ashkarian--Bhargava--Gismondi--Novack use an absolutely summable
Littlewood--Paley paraproduct expansion with intermittent building blocks~\cite{ABGN}, while Christ uses Fourier cutoffs for rough dispersive
solutions~\cite{Christ}.

\subsection{Main results}

On $\T^2$ we use normalized Haar measure and write
\[
 \widehat f(n,m)=\int_{\T^2}f(x,y)e^{-2\pi i(nx+my)}\,dx\,dy,
 \qquad k=(n,m),\qquad \langle k\rangle=(1+n^2+m^2)^{1/2}.
\]
Define
\[
 \mathbb P_{x=0}f=\sum_{m\in\Z}\widehat f(0,m)e^{2\pi imy},
 \qquad \mathbb P_{x\neq0}f=(I-\mathbb P_{x=0})f,
\]
and, for integers $j\geq1$,
\begin{equation}\label{eq:inversex}
 \widehat{\partial_x^{-j}f}(n,m)
 =
 \begin{cases}
  (2\pi in)^{-j}\widehat f(n,m),&n\neq0,\\
  0,&n=0.
 \end{cases}
\end{equation}
For $s\in\R$, put
\begin{equation}\label{eq:timewiener}
 \norm{f}_{\A^s_{x\neq0,t}}
 :=\sum_{\substack{n\neq0\\m\in\Z}}
 \langle(n,m)\rangle^s
 \norm{\widehat f(\cdot,n,m)}_{C_t}.
\end{equation}

\begin{definition}[Periodic nonlinear product]
\label{def:Q}
Fix $r\in\{2,3\}$ and $S>0$.  For
$u_1,\ldots,u_r\in C_t\mathcal D'(\T^2)$, set
\begin{equation}\label{eq:periodic-AF}
 \operatorname{AF}_{S,r}^x(u_1,\ldots,u_r)
 :=\sum_{(k_1+\cdots+k_r)_x\neq0}
 \langle k_1+\cdots+k_r\rangle^{-S}
 \norm{\prod_{j=1}^r\widehat u_j(\cdot,k_j)}_{C_t},
\end{equation}
where the sum is over $(k_1,\ldots,k_r)\in(\Z^2)^r$.  We abbreviate
the diagonal value by $\operatorname{AF}_{S,r}^x(u)$ and, when $r=2$,
write $\operatorname{AF}_S^x(u,v)$ and $\operatorname{AF}_S^x(u)$.
If $\operatorname{AF}_{S,r}^x(u)<\infty$, define
\begin{equation}\label{eq:periodic-product}
 \widehat{\mathbb P_{x\neq0}(u^r)}(t,k)
 =\sum_{k_1+\cdots+k_r=k}\prod_{j=1}^r\widehat u(t,k_j),
 \qquad k_x\neq0,
\end{equation}
and set the coefficients on $k_x=0$ equal to zero.  Then
\begin{equation}\label{eq:Nbound}
 \norm{\mathbb P_{x\neq0}(u^r)}_{\A^{-S}_{x\neq0,t}}
 \leq\operatorname{AF}_{S,r}^x(u).
\end{equation}
\end{definition}

\begin{definition}[Periodic weak singular solution]
\label{def:singularsolution}
Fix $r\in\{2,3\}$, $d\in\{3,5\}$, and $\kappa\in\{-1,1\}$.
Let $u_{\mathrm{in}}\in\mathcal D'(\T^2)$ satisfy
$\widehat u_{\mathrm{in}}(0,m)=0$ for every $m\in\Z$.  A function
$u\in C_t\mathcal D'(\T^2)$ is a weak singular solution of
\eqref{eq:KP-family} with initial datum $u_{\mathrm{in}}$ if
\begin{enumerate}[label=\textup{(\roman*)}]
 \item $\mathbb P_{x\neq0}u=u$;
 \item $\operatorname{AF}_{S,r}^x(u)<\infty$ for some $S>0$;
 \item for every $\phi\in C_c^\infty([0,1)\times\T^2)$,
 \begin{equation}\label{eq:weakform}
  \int_0^1\left[
   \ip{u}{\partial_t\phi+(-1)^{(d+1)/2}\partial_x^d\phi
    +\kappa\partial_x^{-1}\partial_y^2\phi}
   +\ip{\mathbb P_{x\neq0}(u^r)}{\partial_x\phi}
  \right]dt
  =-\ip{u_{\mathrm{in}}}{\phi(0)}.
 \end{equation}
\end{enumerate}
Here $\partial_x^{-1}$ is defined by \eqref{eq:inversex}.
\end{definition}

\begin{theorem}[Periodic KP and modified KP nonuniqueness]
\label{thm:periodic}
Fix $r\in\{2,3\}$, $d\in\{3,5\}$, $\kappa\in\{-1,1\}$, and
$S>d+1$.  For every nonempty open interval $J\subset(0,1)$ there are
infinitely many real periodic weak singular solutions of
\eqref{eq:KP-family} such that
\begin{equation}\label{eq:periodic-regularity}
 \begin{aligned}
 &u(0)=0,\qquad \supp_tu\Subset J,\qquad \partial_yu\not\equiv0,\\
 &u\in\bigcap_{1\leq p<r}C_tL^p(\T^2)
 \cap\bigcap_{\alpha<1-\frac2r}
 C_t\bigl(H^{\alpha,0}(\T^2)\cap H^\alpha(\T^2)\bigr),\\
 &\operatorname{AF}_{S,r}^x(u)<\infty.
 \end{aligned}
\end{equation}
When $r=3$, there are also infinitely many real periodic weak singular
solutions satisfying
\begin{equation}\label{eq:periodic-parabolic-regularity}
 \begin{aligned}
 &u(0)=0,\qquad \supp_tu\Subset J,\qquad \partial_yu\not\equiv0,\\
 &u\in\bigcap_{1\leq p<3}C_tL^p(\T^2)
 \cap\bigcap_{s<1/2}C_tH^{s,0}(\T^2)
 \cap\bigcap_{\alpha<1/4}C_tH^\alpha(\T^2),\\
 &\operatorname{AF}_{S,3}^x(u)<\infty,\\
 &\supp\widehat u\subset\{(n,m):\abs m\leq C\abs n^2\}
 \end{aligned}
\end{equation}
for some $C>0$.
\end{theorem}

On $\R^2$, write $k=(\xi,\eta)$ and use
\[
 \widehat f(\xi,\eta)
 =\int_{\R^2}f(x,y)e^{-2\pi i(x\xi+y\eta)}\,dx\,dy,
 \qquad \langle k\rangle=(1+\xi^2+\eta^2)^{1/2}.
\]
Set
\[
 \norm{G}_{\A^{-S}_{\R^2,t}}
 :=\int_{\R^2}\langle k\rangle^{-S}
 \norm{\widehat G(\cdot,k)}_{C_t}\,dk.
\]

\begin{definition}[Whole-space nonlinear product]
\label{def:Euclidean-product}
Fix $r\in\{2,3\}$ and $S>0$.  For
$u_1,\ldots,u_r\in C_tL^1(\R^2)$, set
\begin{equation}\label{eq:Euclidean-Q}
 \operatorname{AF}^{\partial_x}_{S,r}(u_1,\ldots,u_r)
 :=2\pi\int_{(\R^2)^r}\abs{\xi_1+\cdots+\xi_r}
 \langle k_1+\cdots+k_r\rangle^{-S}
 \norm{\prod_{j=1}^r\widehat u_j(\cdot,k_j)}_{C_t}
 \prod_{j=1}^r dk_j.
\end{equation}
We abbreviate the diagonal value by
$\operatorname{AF}^{\partial_x}_{S,r}(u)$ and, when $r=2$, write
$\operatorname{AF}^{\partial_x}_S(u,v)$ and
$\operatorname{AF}^{\partial_x}_S(u)$.
If this quantity is finite, define
\begin{equation}\label{eq:Euclidean-nonlinearity}
 \widehat{\partial_x[u^r]}(t,k)
 =2\pi i\xi\int_{(\R^2)^{r-1}}
 \prod_{j=1}^{r-1}\widehat u(t,k_j)
 \widehat u\left(t,k-\sum_{j=1}^{r-1}k_j\right)
 \prod_{j=1}^{r-1}dk_j,
 \qquad\text{for a.e. }k\in\R^2.
\end{equation}
Tonelli's theorem and dominated convergence give
$\partial_x[u^r]\in\A^{-S}_{\R^2,t}$.
\end{definition}

\begin{definition}[Whole-space weak singular solution]
\label{def:Euclidean-solution}
Fix $r\in\{2,3\}$, $d\in\{3,5\}$, and $\kappa\in\{-1,1\}$.  Let
$u_{\mathrm{in}}\in L^1(\R^2)$.  A real $u\in C_tL^1(\R^2)$ is a weak singular solution of
\eqref{eq:KP-family} with a fixed $x$-frequency gap if, for some
$c_0,S>0$,
\begin{enumerate}[label=\textup{(\roman*)}]
 \item $\widehat u(t,\xi,\eta)=0$ whenever $\abs\xi<c_0$;
 \item $\operatorname{AF}^{\partial_x}_{S,r}(u)<\infty$;
 \item for every $\phi\in C_c^\infty([0,1)\times\R^2)$,
 \begin{equation}\label{eq:Euclidean-weak-form}
  \int_0^1\Bigl[
   \ip{u}{\partial_t\phi+(-1)^{(d+1)/2}\partial_x^d\phi}
   -\kappa\ip{\partial_x^{-1}\partial_y^2u}{\phi}
   -\ip{\partial_x[u^r]}{\phi}
  \Bigr]dt
  =-\ip{u_{\mathrm{in}}}{\phi(0)}.
 \end{equation}
\end{enumerate}
Here
\[
 \widehat{\partial_x^{-1}\partial_y^2u}
 =2\pi i\frac{\eta^2}{\xi}\widehat u.
\]
\end{definition}

\begin{theorem}[Whole-space KP and modified KP nonuniqueness]
\label{thm:Euclidean}
Fix $r\in\{2,3\}$, $d\in\{3,5\}$, $\kappa\in\{-1,1\}$, and
$S>d+1$.  Let $J\subset(0,1)$ be a nonempty open interval and let
$c_0>0$.  There are infinitely many weak singular solutions of
\eqref{eq:KP-family} such that
\begin{equation}\label{eq:Euclidean-regularity}
 \begin{aligned}
 &u(0)=0,\qquad \supp_tu\Subset J,\qquad \partial_yu\not\equiv0,\\
 &u\in\bigcap_{1\leq p<r}C_tL^p(\R^2)
 \cap\bigcap_{\alpha<1-\frac2r}
 C_t\bigl(H^{\alpha,0}(\R^2)\cap H^\alpha(\R^2)\bigr),\\
 &\operatorname{AF}^{\partial_x}_{S,r}(u)<\infty,
 \end{aligned}
\end{equation}
and
\begin{equation}\label{eq:Euclidean-gap}
 \widehat u(t,\xi,\eta)=0\qquad\text{whenever }\abs\xi<c_0.
\end{equation}
When $r=3$, there are also infinitely many weak singular
solutions satisfying \eqref{eq:Euclidean-gap} and
\begin{equation}\label{eq:Euclidean-parabolic-regularity}
 \begin{aligned}
 &u(0)=0,\qquad \supp_tu\Subset J,\qquad \partial_yu\not\equiv0,\\
 &u\in\bigcap_{1\leq p<3}C_tL^p(\R^2)
 \cap\bigcap_{s<1/2}C_tH^{s,0}(\R^2)
 \cap\bigcap_{\alpha<1/4}C_tH^\alpha(\R^2),\\
 &\operatorname{AF}^{\partial_x}_{S,3}(u)<\infty,\\
 &\supp\widehat u\subset
 \{(\xi,\eta):\abs\eta\leq C\abs\xi^2\}
 \end{aligned}
\end{equation}
for some $C>0$.
\end{theorem}

The solutions in both theorems may be chosen arbitrarily small in any
fixed finite collection of these norms.

For $r=2$ and $d=5$, both theorems also apply to
\eqref{eq:fifth-order-divided}.

For the cubic equations, the Sobolev embedding
$H^{1/3}\hookrightarrow L^3$ holds on both domains.
If, on either domain, the Fourier support of $f$ satisfies
$\abs{k_y}\lesssim\abs{k_x}^2$ and $P_N^x$ denotes a dyadic
$x$-frequency projection, then dyadic Bernstein and Littlewood--Paley
theory give
\begin{equation}\label{eq:parabolic-product-endpoint}
 \norm f_{L^3}\lesssim
 \left(\sum_N\norm{P_N^xf}_{L^3}^2\right)^{1/2}
 \lesssim
 \left(\sum_N N\norm{P_N^xf}_2^2\right)^{1/2}
 \lesssim\norm f_{H^{1/2,0}}.
\end{equation}
Both exponents are sharp.  Indeed, the profiles in
Lemma~\ref{lem:cubic-profile} have $L^3$ norm at least one by
\eqref{eq:cubic-profile-moment}, while, after choosing $\eps>0$
sufficiently small,
\[
 \norm{\rho_{3,\la,\eps}}_{H^\alpha}\longrightarrow0
 \quad(\alpha<1/3),
 \qquad
 \norm{\widetilde\rho_{3,\la,\eps}}_{H^{s,0}}\longrightarrow0
 \quad(s<1/2)
\]
by \eqref{eq:cubic-profile-support-bounds} and
\eqref{eq:cubic-profile-endpoints}.
The corresponding isotropic and parabolic dilations give the same
optimality on $\R^2$.
For third-order modified KP
on $\R^2$, the scaling
\[
 u_\Lambda(t,x,y)=\Lambda u(\Lambda^3t,\Lambda x,\Lambda^2y),
 \qquad
 \norm{u_\Lambda(t)}_{\dot H^{s,0}}
 =\Lambda^{s-1/2}\norm{u(\Lambda^3t)}_{\dot H^{s,0}},
\]
also identifies $s=1/2$ as the critical exponent in $H^{s,0}$.  In the
whole-space fifth-order quadratic case,
Patterson's unconditional uniqueness for $s>0$~\cite{Patterson5KP} and
Theorem~\ref{thm:Euclidean} for $s<0$ leave only the endpoint $s=0$;
hence the nonuniqueness range is almost sharp.

For a time-independent distribution $u$, define
$\operatorname{AF}_S^x(u)$ and $\mathbb P_{x\neq0}(u^2)$ as in
Definition~\ref{def:Q}, with
$\norm{\widehat u(\cdot,k)}_{C_t}$ replaced by $\abs{\widehat u(k)}$.

\begin{definition}[Stationary weak singular solution]
\label{def:stationary-solution}
Fix $d\in\{3,5\}$ and $\kappa\in\{-1,1\}$.  A real distribution $u$ on
$\T^2$ is a stationary weak singular solution if
$\mathbb P_{x\neq0}u=u$, $\operatorname{AF}_S^x(u)<\infty$ for some
$S>0$, and
\begin{equation}\label{eq:stationary-weak}
 \ip{u}{(-1)^{(d+1)/2}\partial_x^d\phi
      +\kappa\partial_x^{-1}\partial_y^2\phi}
 +\ip{\mathbb P_{x\neq0}(u^2)}{\partial_x\phi}=0
 \qquad\text{for every }\phi\in C^\infty(\T^2).
\end{equation}
\end{definition}

\begin{theorem}[Periodic stationary solutions]
\label{thm:stationary-flexibility}
Fix $d\in\{3,5\}$, $\kappa\in\{-1,1\}$, and $S>d+1$.  There are
infinitely many stationary weak singular solutions depending nontrivially
on $y$ such that
\[
 u\in\bigcap_{1\leq p<2}L^p(\T^2)
    \cap\bigcap_{\sigma>0}
       \bigl(H^{-\sigma,0}(\T^2)\cap H^{-\sigma}(\T^2)\bigr),
 \qquad
 \operatorname{AF}_S^x(u)<\infty.
\]
\end{theorem}

In the $L^2$ class, the stationary KP-I symbol is coercive, and the Fourier
equation yields a regularity bootstrap.

\begin{theorem}[Regularity of stationary KP-I solutions]
\label{thm:stationary-rigidity}
Fix $d\in\{3,5\}$.  Let $u\in L^2(\T^2)$ be real,
$\mathbb P_{x\neq0}u=u$, and suppose that
\[
 (-1)^{(d+1)/2}\partial_x^du
 -\partial_x^{-1}\partial_y^2u+\partial_x(u^2)=0
 \quad\text{in }\mathcal D'(\T^2),
\]
where $u^2$ is the ordinary $L^1$ product.  Then
$u\in C^\infty(\T^2)$.
\end{theorem}

Thus $L^2$ is the sharp regularity threshold for periodic stationary KP-I
for both dispersion orders.
For $d=5$, both stationary theorems also apply to
\eqref{eq:fifth-order-divided}.

\subsection{Cubic profiles and transverse frequencies}

The cubic regularity comes from concentrating each perturbation
in both spatial variables.  The profiles are products of two Fej\'er
kernels, shifted to $x$-frequency $\la$ and placed on a Fourier
lattice of spacing $\la^\eps$.  When both kernels have order
$\la^{1-\eps}$, the unnormalized cubic moment is comparable to
$\la^{4(1-\eps)}$ and the $L^2$ norm is bounded by a constant times
$\la^{1-\eps}$.  Normalizing the cubic moment to one therefore gives
$L^2$ size $O(\la^{-(1-\eps)/3})$.  The $x$- and $y$-frequencies are
bounded by constant multiples of $\la$, so the
resulting perturbation has $H^\alpha$ size
\[
 O\bigl(\la^{\alpha-(1-\eps)/3}\bigr).
\]
This yields the isotropic range $\alpha<1/3$.

For the parabolic family, the $x$-kernel still has order
$\la^{1-\eps}$, while the $y$-kernel has order
$\la^{2(1-\eps)}$.  The unnormalized cubic moment is then comparable to
$\la^{6(1-\eps)}$ and the $L^2$ norm is bounded by a constant times
$\la^{3(1-\eps)/2}$.  Cubic normalization gives $L^2$ size
$O(\la^{-(1-\eps)/2})$.  The $x$-frequencies remain comparable
to $\la$, whereas the transverse frequencies are bounded by
$\la^{2-\eps}$.  Since $H^{s,0}$ weights only the $x$-frequency,
the perturbation has $H^{s,0}$ size
\[
 O\bigl(\la^{s-(1-\eps)/2}\bigr),
\]
which yields $s<1/2$.
The parabolic support also gives $H^\alpha$ for every $\alpha<1/4$.

To obtain nontrivial $y$-dependence on $\T^2$, fix
$(n_*,m_*)\in\Z^2$ with $n_*m_*\neq0$.  The subsequent perturbations
avoid this mode, so, for every $q$,
\[
 \widehat u_q(t,n_*,m_*)
 =\widehat u_0(t,n_*,m_*)\not\equiv0.
\]
On $\R^2$, the subsequent perturbations leave $\widehat u_0$ unchanged
on an open set disjoint from $\{\eta=0\}$.  In either domain these
frequencies have nonzero $y$-frequency, so the limit depends on $y$.

The multipliers $\partial_x^{-1}\partial_y^2$ and
$\partial_x^{-2}\partial_y^2$ constrain the $y$-frequencies.  The absolute
values of their symbols are comparable to $\eta^2/\abs\xi$ and
$\eta^2/\xi^2$, respectively.  The
isotropic profiles lie in a fixed cone $\abs\eta\lesssim\abs\xi$,
whereas the parabolic profiles lie in the region
$\abs\eta\lesssim\abs\xi^2$.  On $\R^2$, spatial
localization is followed by Fourier truncation below the spacing of the
profile lattice; modulation then produces a fixed gap from
$\{\xi=0\}$.

On the parabolic profile support, $\abs\xi\simeq\la$ and
$\abs\eta\lesssim\la^2$, so the multiplier in the differentiated
dispersion error satisfies
\[
 \left(\abs\xi^d+\frac{\eta^2}{\abs\xi}\right)
 \langle(\xi,\eta)\rangle^{-S}
 \lesssim\la^{d-S}+\la^{1-S}.
\]
The $\ell^1$ norm of the Fourier coefficients of the profile is
$O(\la^{1-\eps})$.  Thus the whole-space dispersion error is
$O(\la^{d+1-S-\eps})$, which tends to zero for $S>d+1$.

\section{Fourier estimates and intermittent profiles}

\subsection{Fourier cutoffs and weighted Wiener estimates}

Let $P_{\leq L}^{x,y}$ be a Fourier multiplier with a smooth, real, even
symbol supported in $\{\langle(n,m)\rangle\leq2L\}$ and equal to one on
$\{\langle(n,m)\rangle\leq L\}$.  We take the symbol in $[0,1]$.

For amplitudes we use the full inhomogeneous norm, which includes the
$x$-zero modes:
\begin{equation}\label{eq:fullwiener}
    \norm{a}_{\A^s_t}
    =\sum_{n,m\in\Z}\langle(n,m)\rangle^s
      \norm{\widehat a(\cdot,n,m)}_{C_t}.
\end{equation}

\begin{lemma}\label{lem:L1wiener}
Let $F\in C_tL^1(\T^2)$.  If $S>2$, then
\begin{equation}\label{eq:L1Aminus}
    \norm{\mathbb P_{x\neq0} F}_{\A^{-S}_{x\neq0,t}}
    \lesssim_S \norm{F}_{C_tL^1}.
\end{equation}
If a Fourier multiplier $T$ on $n\neq0$ has symbol satisfying
$\abs{m_T(n,m)}\lesssim\langle(n,m)\rangle^a$, then, for $S>a+2$,
\begin{equation}\label{eq:multiplierL1}
    \norm{TF}_{\A^{-S}_{x\neq0,t}}
    \lesssim_{S,a}\norm{F}_{C_tL^1}.
\end{equation}
If in addition, for some $M\geq1$,
$\widehat F(t,n,m)=0$ for every $t$ whenever $\abs n<M$, then
\begin{equation}\label{eq:gapL1}
    \norm{\mathbb P_{x\neq0} F}_{\A^{-S}_{x\neq0,t}}
    \lesssim_S M^{2-S}\norm{F}_{C_tL^1},
    \qquad S>2.
\end{equation}
\end{lemma}

\begin{proof}
\[
 \norm{TF}_{\A^{-S}_{x\neq0,t}}
 \lesssim \norm{F}_{C_tL^1}
    \sum_{n\neq0,m\in\Z}\langle(n,m)\rangle^{a-S},
\]
which converges precisely when $S>a+2$.  The proof of
\eqref{eq:L1Aminus} is the case $a=0$.  For the last estimate,
\[
 \sum_{\abs n\geq M}\sum_{m\in\Z}\langle(n,m)\rangle^{-S}
 \lesssim_S\sum_{\abs n\geq M}\langle n\rangle^{1-S}
 \lesssim_S M^{2-S}.
\]
\end{proof}

\subsection{Cutoff independence for the periodic product}

\begin{proposition}
\label{prop:cutoff}
Fix $r\in\{2,3\}$ and suppose
$\operatorname{AF}_{S,r}^x(u)<\infty$.  Let $P_N$ be Fourier
multipliers with symbols $m_N(k)$ such that each $m_N$ has finite support,
$\sup_N\norm{m_N}_{\ell^\infty}<\infty$, and $m_N(k)\to1$ for every
$k\in\Z^2$.  Then
\begin{equation}\label{eq:cutoffconv}
 \mathbb P_{x\neq0}\bigl((P_Nu)^r\bigr)
 \longrightarrow\mathbb P_{x\neq0}(u^r)
 \quad\text{in }\A^{-S}_{x\neq0,t}.
\end{equation}
Consequently,
\begin{equation}\label{eq:cutoffderivative}
 \partial_x\bigl[(P_Nu)^r\bigr]
 \longrightarrow\partial_x\mathbb P_{x\neq0}(u^r)
\end{equation}
in $\A^{-S-1}_{x\neq0,t}$ and in distributions.  If in addition
$u\in C_tL^r(\T^2)$, then the Fourier-defined product agrees with the
projection of the ordinary product $u^r\in C_tL^1(\T^2)$.
\end{proposition}

\begin{proof}
After expanding the products, the norm of the difference in
\eqref{eq:cutoffconv} is bounded by
\[
 \sum_{(k_1+\cdots+k_r)_x\neq0}
 \abs{\prod_{j=1}^rm_N(k_j)-1}
 \langle k_1+\cdots+k_r\rangle^{-S}
 \norm{\prod_{j=1}^r\widehat u(\cdot,k_j)}_{C_t}.
\]
The first factor tends pointwise to zero and is uniformly bounded, while
the remaining series is \eqref{eq:periodic-AF}.  Dominated convergence
proves \eqref{eq:cutoffconv}, and applying $\partial_x$ proves
\eqref{eq:cutoffderivative}.

For smooth approximate-identity Fourier cutoffs,
$P_Nu\to u$ in $C_tL^r$, and H\"older's inequality gives
\[
 \norm{(P_Nu)^r-u^r}_{C_tL^1}
 \leq\sum_{j=0}^{r-1}
  \norm{P_Nu}_{C_tL^r}^{r-1-j}\norm u_{C_tL^r}^j
  \norm{P_Nu-u}_{C_tL^r}\longrightarrow0.
\]
The cutoff-independent limit is therefore the projection of the ordinary
product.
\end{proof}

\subsection{The relaxed equation and perturbation update}

Fix $r\in\{2,3\}$, $d\in\{3,5\}$, and
$\kappa\in\{-1,1\}$.  On $\T^2$ let $u$ and $E$ be smooth and real,
both with zero $x$-mean, and satisfy
\begin{equation}\label{eq:relaxed-general}
 \partial_tu+(-1)^{(d+1)/2}\partial_x^du
 +\kappa\partial_x^{-1}\partial_y^2u+\partial_x(u^r)
 =\partial_xE.
\end{equation}
For $u^+=u+w$, where $w$ has zero $x$-mean, define
\begin{equation}\label{eq:periodic-update-general}
 E^+=\mathbb P_{x\neq0}\left(
 E+(u+w)^r-u^r+(-1)^{(d+1)/2}\partial_x^{d-1}w
 +\kappa\partial_x^{-2}\partial_y^2w
 +\partial_t\partial_x^{-1}w\right).
\end{equation}
Then $(u^+,E^+)$ satisfies \eqref{eq:relaxed-general}.  On $\R^2$, assume
that $\widehat w$ has a fixed gap from $\{\xi=0\}$ and use the whole-space
update
\begin{equation}\label{eq:Euclidean-update-general}
 E^+=E+(u+w)^r-u^r+(-1)^{(d+1)/2}\partial_x^{d-1}w
 +\kappa\partial_x^{-2}\partial_y^2w
 +\partial_t\partial_x^{-1}w.
\end{equation}

The Fourier support of $w$ must avoid zero $x$-frequency, but no such
restriction is imposed on $w^r$.  When the large $x$-frequencies in the
factors of $w$ sum to zero, the corresponding terms in $w^r$ return to
low $x$-frequency and cancel $E$, up to the localization and
amplitude-truncation errors on $\R^2$.  The inverse powers of
$\partial_x$ in the update act only on $w$, while the relaxed equation
contains $\partial_x(w^r)$.

\subsection{Periodic intermittent profiles}

The quadratic perturbation requires an exact second moment for cancellation,
decay in $L^p$ for $p<2$, and separated Fourier support with nonnegative
coefficients.  Trigonometric polynomial versions of the intermittent profiles
in~\cite{GMPR} provide these properties.

\begin{lemma}\label{lem:slab}
Fix $0<\eps<1$.  For every sufficiently large
$\la>1$ such that $\mu=\la^\eps\in\mathbb N$, set
\begin{equation*}
 \nu=\la^{1+\eps}=\la\mu.
\end{equation*}
We call these choices of $\la$ admissible.
Then there exists a real, even trigonometric polynomial
$\rho_{\la,\eps}=\rho_{\la,\eps}(x)$ such that
\begin{align}
 &\int_\T\rho_{\la,\eps}\,dx=0,
 \qquad \int_\T\rho_{\la,\eps}^2\,dx=1,
 \label{eq:slabmoments}\\
 &\supp\widehat\rho_{\la,\eps}
 \subset(\mu\Z\setminus\{0\})\cap[-2\nu,2\nu],
 \notag\\
 &\norm{\rho_{\la,\eps}}_{L^p(\T)}
 \leq C_{p,\eps}\la^{(1-\eps)(1/2-1/p)},
 \qquad1\leq p\leq\infty,
 \label{eq:slabLp}\\
 &0\leq\widehat\rho_{\la,\eps}(n)
 \leq C_\eps\la^{-(1-\eps)/2}.
 \label{eq:slab-coefficient-bound}
\end{align}
\end{lemma}

\begin{proof}
Choose a nonzero real, even function $\psi\in C_c^\infty((-1/8,1/8))$
with $\int_\R\psi=0$, and set
$\phi_2=(\psi*\psi)/\norm{\psi*\psi}_{L^2(\R)}$.  Then
\begin{equation*}
 \phi_2\in C_c^\infty((-1/4,1/4)),\quad
 \phi_2\text{ is real and even},\quad
 \int_\R\phi_2=0,\quad \widehat\phi_2\geq0,
 \quad \int_\R\phi_2^2=1.
\end{equation*}

For the rest of the proof write $J=\la/\mu=\la^{1-\eps}$.  Define the
untruncated periodic function
\begin{equation}\label{eq:slab-untruncated}
 \rho^\circ_{\la,\eps}(x)
 =J^{1/2}\sum_{j\in\Z}\phi_2(\la x+Jj).
\end{equation}
Since $\la=\mu J$, translating $x$ by $1/\mu$ in
\eqref{eq:slab-untruncated} changes the summation index from $j$ to
$j+1$.  Thus
$\rho^\circ_{\la,\eps}$ is $1/\mu$-periodic and, because
$\mu\in\mathbb N$, is a
well-defined function on $\T$.  The supports of the translates of $\phi_2$ are
disjoint.  Integrating over the $\mu$ fundamental intervals of length
$1/\mu$, and then making the change of variables $z=\la x$, gives, for
$1\leq p<\infty$,
\begin{equation}\label{eq:slab-untruncated-Lp}
 \norm{\rho^\circ_{\la,\eps}}_{L^p(\T)}^p
 =J^{p/2-1}\norm{\phi_2}_{L^p(\R)}^p.
\end{equation}
The same disjointness gives
\begin{equation}\label{eq:slab-untruncated-Linfty}
 \norm{\rho^\circ_{\la,\eps}}_{L^\infty(\T)}
 =J^{1/2}\norm{\phi_2}_{L^\infty(\R)}.
\end{equation}
$\rho^\circ_{\la,\eps}$ has mean zero and
$\int_\T(\rho^\circ_{\la,\eps})^2=1$.

Poisson summation with lattice spacing $J$ yields the exact Fourier series
\begin{equation*}
 \rho^\circ_{\la,\eps}(x)
 =J^{-1/2}\sum_{j\in\Z}
  \widehat\phi_2(j/J)e^{2\pi i\mu jx}.
\end{equation*}
Thus its Fourier coefficients are nonnegative, the coefficient at zero
vanishes, and every frequency lies in $\mu\Z$.

To obtain a trigonometric polynomial, fix an even
$\chi\in C_c^\infty((-2,2))$ such that
$0\leq\chi\leq1$ and $\chi=1$ on $[-1,1]$, and define
\begin{equation*}
 \widetilde\rho_{\la,\eps}(x)
 =J^{-1/2}\sum_{j\in\Z}
  \chi(j/\la)\widehat\phi_2(j/J)e^{2\pi i\mu jx}.
\end{equation*}
At $k=\mu j$, the cutoff equals $\chi(k/\nu)$; it is one on
$\abs k\leq\nu$ and supported in
$\abs k<2\nu$.  Hence the profile is real and even, its Fourier coefficients are
nonnegative, it has mean zero, and
\begin{equation}\label{eq:slab-truncated-support}
 \supp\widehat{\widetilde\rho}_{\la,\eps}
 \subset(\mu\Z\setminus\{0\})\cap[-2\nu,2\nu].
\end{equation}

For every integer $a\geq0$ and every $M\geq0$, Schwartz decay of
$\widehat\phi_2$ and comparison with an integral give
\begin{equation}\label{eq:slab-uniform-tail}
 \norm{\partial_x^a(\rho^\circ_{\la,\eps}
       -\widetilde\rho_{\la,\eps})}_{L^\infty}
 \leq C_{a,M}J^{1/2}\la^a\mu^{-M}.
\end{equation}
Indeed, the difference contains only indices $\abs j\geq\la$, and for
arbitrarily large $N$ its left side is at most
\begin{align*}
 C_{a,N}J^{-1/2}\mu^a
 \sum_{\abs j\geq\la}\abs j^a(1+\abs j/J)^{-N}
 &\leq C_{a,N}J^{a+1/2}\mu^{2a+1-N}\\
 &\leq C_{a,M}J^{1/2}\la^a\mu^{-M}
\end{align*}
once $N\geq M+a+2$, since $\la=J\mu$.
Since $J=\mu^{(1-\eps)/\eps}$, \eqref{eq:slab-uniform-tail} with a
larger decay exponent gives
$\norm{\rho^\circ_{\la,\eps}-\widetilde\rho_{\la,\eps}}_{L^2}
=O_M(\mu^{-M})$.  Hence
\begin{equation}\label{eq:slab-normalizer-limit}
 \int_\T\widetilde\rho_{\la,\eps}^2\,dx
 =1+O_M(\mu^{-M})
 \qquad\text{for every }M>0.
\end{equation}
For all sufficiently large admissible $\la$, set
\begin{equation*}
 Z_{\la,\eps}
 =\left(\int_\T\widetilde\rho_{\la,\eps}^2\,dx\right)^{1/2},
 \qquad
 \rho_{\la,\eps}
 =Z_{\la,\eps}^{-1}\widetilde\rho_{\la,\eps}.
\end{equation*}
The normalized profile is real and even, has mean zero and nonnegative
Fourier coefficients, and retains the support in
\eqref{eq:slab-truncated-support}, while giving the exact identity
$\int_\T\rho_{\la,\eps}^2=1$.  Moreover,
$Z_{\la,\eps}=1+O_M(\mu^{-M})$ and
$1/2\leq Z_{\la,\eps}\leq2$ for large $\la$.  Its Fourier coefficients are
\begin{equation*}
 \widehat\rho_{\la,\eps}(k)=
 \begin{cases}
  Z_{\la,\eps}^{-1}J^{-1/2}
  \chi(j/\la)\widehat\phi_2(j/J),&k=\mu j,\\
  0,&k\notin\mu\Z.
 \end{cases}
\end{equation*}
Since $Z_{\la,\eps}^{-1}\leq2$ and $J=\la^{1-\eps}$,
\[
 0\leq\widehat\rho_{\la,\eps}(k)
 \leq2J^{-1/2}\norm{\widehat\phi_2}_{L^\infty}
 \leq C_\eps\la^{-(1-\eps)/2},
\]
which proves \eqref{eq:slab-coefficient-bound}.

Finally fix $1\leq p\leq\infty$.  Choose $M$ so large that
$\mu^{-M}\leq J^{-1/p}$ for all sufficiently large $\la$.  Equations
\eqref{eq:slab-untruncated-Lp}, \eqref{eq:slab-untruncated-Linfty},
\eqref{eq:slab-uniform-tail}, and \eqref{eq:slab-normalizer-limit} imply
\begin{align*}
 \norm{\rho_{\la,\eps}}_{L^p(\T)}
 &\leq2\left(\norm{\rho^\circ_{\la,\eps}}_{L^p(\T)}
 +\norm{\rho^\circ_{\la,\eps}
 -\widetilde\rho_{\la,\eps}}_{L^p(\T)}\right)\\
 &\leq C_{p,\eps}J^{1/2-1/p}
 =C_{p,\eps}\la^{(1-\eps)(1/2-1/p)}.
\end{align*}
\end{proof}

\section{Periodic quadratic estimates}

\subsection{The perturbation and its error decomposition}

Fix $d\in\{3,5\}$ and $\kappa\in\{-1,1\}$.  For $r=2$, write the updated
error as
\begin{align}
    E_O&=\mathbb P_{x\neq0}(E+w^2),\label{eq:RO}\\
    E_N&=2\mathbb P_{x\neq0}(uw),\label{eq:RN}\\
    E_D&=(-1)^{(d+1)/2}\partial_x^{d-1}w
       +\kappa\partial_x^{-2}\partial_y^2w,\label{eq:RD}\\
    E_T&=\partial_t\partial_x^{-1}w.\label{eq:RT}
\end{align}
These are the oscillation, Nash, dispersion, and temporal errors,
respectively.
For $I=[t_-,t_+]\Subset(0,1)$ and
$0<\tau<\operatorname{dist}(I,\{0,1\})$, choose
$h\in C_c^\infty((0,1))$ such that
\begin{equation}\label{eq:time-cutoff-lemma}
   0\leq h\leq1,
   \qquad h=1\ \text{on }I,
   \qquad \supp h\subset(t_--\tau,t_++\tau),
   \qquad \norm{h'}_{L^\infty}\leq C\tau^{-1},
\end{equation}
where $C$ is universal.

The square-root amplitude is not finitely supported in frequency, whereas
the perturbation must be.  Truncation at scale $\mu^{1/2}$ is negligible
relative to the profile spacing $\mu$.

\begin{lemma}\label{lem:amplitude-cutoff}
Let $E\in C^1([0,1];C^\infty(\T^2))$ be real and have fixed finite spatial
Fourier support.  Suppose
$A\geq1+2\norm{E}_{C_tL^\infty}$ and put
\[
   a=(A-E)^{1/2},
   \qquad b_L=P_{\leq L}^{x,y}a,
   \qquad L\geq2.
\]
For every integer $N\geq0$,
\begin{equation}\label{eq:amplitude-smooth-bounds}
   \norm{a}_{\A^N_t}+\norm{\partial_ta}_{\A^N_t}
      \leq C_{E,A,N}.
\end{equation}
Moreover, $b_L$ is real, its spatial Fourier support is contained in
$\{\langle k\rangle\leq2L\}$, and
\begin{equation}\label{eq:amplitude-uniform-bounds}
   \norm{b_L}_{C_tL^\infty}
     +\norm{\partial_tb_L}_{C_tL^\infty}
       \leq C_{E,A}.
\end{equation}
If $0\leq j\leq N$, then
\begin{equation*}
\begin{split}
   \norm{a-b_L}_{\A^j_t}
      &\leq C_{E,A,N}L^{j-N},\\
   \norm{a^2-b_L^2}_{\A^j_t}
      &\leq C_{E,A,N,j}L^{j-N}.
\end{split}
\end{equation*}
In particular, for $L=\mu^{1/2}$, every integer $N\geq0$, and every
$S>0$,
\begin{equation}\label{eq:amplitude-negative-tail}
   \norm{\mathbb P_{x\neq0}(b_L^2-a^2)}
      _{\A^{-S}_{x\neq0,t}}
      \leq C_{E,A,N}\mu^{-N/2}.
\end{equation}
\end{lemma}

\begin{proof}
Smooth functional calculus and rapid Fourier decay give
\eqref{eq:amplitude-smooth-bounds}.  The cutoff gives the stated support,
and $b_L$ is real-valued.  The cutoff is contractive in nonnegative weighted
Wiener norms, and $\A^0\hookrightarrow L^\infty$ gives
\eqref{eq:amplitude-uniform-bounds}.

For $0\leq j\leq N$,
\[
 \norm{a-b_L}_{\A^j_t}
 \leq L^{j-N}\norm{a}_{\A^N_t}.
\]
The nonnegative weighted Wiener spaces are algebras, so
\[
 \norm{a^2-b_L^2}_{\A^j_t}
 \lesssim_j
   \norm{a-b_L}_{\A^j_t}
   \bigl(\norm{a}_{\A^j_t}+\norm{b_L}_{\A^j_t}\bigr).
\]
Taking $j=0$, using
$\langle k\rangle^{-S}\leq1$, and then setting $L=\mu^{1/2}$ gives
\eqref{eq:amplitude-negative-tail}.
\end{proof}

Fix
\begin{equation}\label{eq:onestep-fixed-parameters}
   0<\eps<1,
   \qquad 0<\beta<\frac12(1-\eps),
   \qquad S>d+1.
\end{equation}

Fix a pair $(u,E)$ with finite spatial Fourier support and let
\[
   K=\max\left(\{1\}\cup\left\{\langle(n,m)\rangle:
       (n,m)\in\supp\widehat u\cup\supp\widehat E\right\}\right).
\]
Choose $A$ so that
\begin{equation}\label{eq:Achoice}
    A\geq C_S\left(1+\norm{E}_{C_tL^\infty}
                         +\norm{E}_{\A^S_t}\right),
\end{equation}
where $C_S\geq1$ absorbs the weighted-Wiener algebra constant and is large enough
that the binomial series for $(A-E)^{1/2}$ converges absolutely.
Define the positive amplitude
\[
    a=(A-E)^{1/2}.
\]
Choose a closed interval $I\Subset(0,1)$ containing the time supports of
both $u$ and $E$.  With the profile from Lemma~\ref{lem:slab}, choose an
admissible $\la$ so large that
\begin{equation}\label{eq:periodic-frequency-separation}
    \mu\geq64,
    \qquad 16K\leq\mu^{1/2},
\end{equation}
and $\la^{-\beta}<\operatorname{dist}(I,\{0,1\})$.  Choose $h$ satisfying
\eqref{eq:time-cutoff-lemma} with $\tau=\la^{-\beta}$, and
set
\begin{equation}\label{eq:increment}
    b=P_{\leq\mu^{1/2}}^{x,y}a,
    \qquad
    w(t,x,y)=h(t)b(t,x,y)\rho_{\la,\eps}(x).
\end{equation}

Writing $k=(n,m)$, every $k\in\supp\widehat w$ has a decomposition
$k=(j,0)+q$, where
$j\in(\mu\Z\setminus\{0\})\cap[-2\nu,2\nu]$ belongs to
$\supp\widehat{\rho_{\la,\eps}}$ and
$\langle q\rangle\leq2\mu^{1/2}$.  Since $\mu\geq64$,
\begin{equation}\label{eq:wsupport}
 \abs n\geq\abs j-\abs{q_x}
      \geq\mu-2\mu^{1/2}\geq\frac12\mu,
 \qquad
 \abs n\leq2\nu+2\mu^{1/2}\leq3\nu,
 \qquad
 \abs m\leq2\mu^{1/2}
 \quad\text{on }\supp\widehat w.
\end{equation}
For fixed $u,E$, and $A$,
Lemmas~\ref{lem:amplitude-cutoff} and~\ref{lem:slab}, together with
$0\leq h\leq1$, give, for every $t$,
\[
  \norm{w(t)}_{L^p(\T^2)}
  \leq\norm{b(t)}_{L^\infty(\T^2)}
       \norm{\rho_{\la,\eps}}_{L^p(\T)}.
\]
Consequently,
\begin{equation}\label{eq:wLp}
    \norm{w}_{C_tL^p(\T^2)}
    \lesssim_{p,E,A,\eps}
    \la^{(1-\eps)(\frac12-\frac1p)}
    \longrightarrow0,
    \qquad 1\leq p<2.
\end{equation}

\subsection{Error estimates}

Throughout this subsection we assume
\eqref{eq:onestep-fixed-parameters}--\eqref{eq:increment}.  The pair $(u,E)$
and the constant $A$ are fixed before the admissible value of $\la$
is chosen.  The constants are uniform in $\la$ and may depend on
$u,E,A,S$, and $\eps$.

\subsubsection*{Dispersion error}

For $E_D$ in \eqref{eq:RD}, the Fourier multiplier satisfies
\begin{equation*}
 \abs{-(2\pi n)^{d-1}+\kappa\frac{m^2}{n^2}}
 \leq(2\pi\abs n)^{d-1}+\frac{m^2}{n^2}
 \lesssim_d\langle(n,m)\rangle^{d-1},
 \qquad n\neq0.
\end{equation*}
The symbol bound, Lemma~\ref{lem:L1wiener}, and \eqref{eq:wLp} with
$p=1$ give
\begin{equation}\label{eq:RDestimate}
    \norm{E_D}_{\A^{-S}_{x\neq0,t}}
    \lesssim_S\norm{w}_{C_tL^1}
    \lesssim_{E,A,S,\eps}\la^{-\frac12(1-\eps)},
    \qquad S>d+1.
\end{equation}
\subsubsection*{Nash error}

For $E_N=2\mathbb P_{x\neq0}(uw)$ in \eqref{eq:RN},
Lemma~\ref{lem:L1wiener} and \eqref{eq:wLp} with $p=1$ give
\begin{equation}\label{eq:RNestimate}
    \norm{E_N}_{\A^{-S}_{x\neq0,t}}
    \lesssim_S\norm{uw}_{C_tL^1}
    \leq\norm u_{C_tL^\infty}\norm w_{C_tL^1}
    \lesssim_{u,E,A,S,\eps}\la^{-\frac12(1-\eps)},
    \qquad S>2.
\end{equation}

\subsubsection*{Oscillation error}

For $E_O=\mathbb P_{x\neq0}(E+w^2)$ in \eqref{eq:RO}, the identities
$a^2=A-E$ and $hE=E$ give
\begin{equation*}
    E_O
    =h^2\mathbb P_{x\neq0}(b^2-a^2)
     +h^2\mathbb P_{x\neq0}
       \bigl(b^2(\rho_{\la,\eps}^2-1)\bigr).
\end{equation*}
The function $\rho_{\la,\eps}^2-1$ has zero $x$-mean and its nonzero
$x$-frequencies lie in $\mu\Z$.  Since $b^2$ has $x$-frequency at most
$4\mu^{1/2}$,
\begin{equation*}
 (n,m)\in\supp\widehat{b^2(\rho_{\la,\eps}^2-1)}
 \quad\Longrightarrow\quad
 \abs n\geq\mu-4\mu^{1/2}\geq\frac12\mu.
\end{equation*}
\[
   \norm{b^2(\rho_{\la,\eps}^2-1)}_{C_tL^1}
      \leq2\norm{b}_{C_tL^\infty}^2.
\]
Using $0\leq h\leq1$, the amplitude-tail estimate
\eqref{eq:amplitude-negative-tail} and the gap estimate
\eqref{eq:gapL1} therefore give, for every integer $N\geq0$,
\begin{equation}\label{eq:ROestimate}
    \norm{E_O}_{\A^{-S}_{x\neq0,t}}
    \leq C_{E,A,N}\mu^{-N/2}
       +C_{E,A,S}\mu^{2-S},
    \qquad S>2.
\end{equation}
\subsubsection*{Temporal error}

For $E_T=\partial_t\partial_x^{-1}w$ in \eqref{eq:RT}, differentiating
\eqref{eq:increment} gives
\[
   \partial_tw=h'b\rho_{\la,\eps}
      +h(\partial_tb)\rho_{\la,\eps}.
\]
Lemma~\ref{lem:amplitude-cutoff}, Lemma~\ref{lem:slab}, and
$\norm{h'}_{L^\infty}\lesssim\la^\beta$ therefore yield
\begin{equation}\label{eq:dwt}
\begin{split}
    \norm{\partial_tw}_{C_tL^1}
    &\leq
       \left(\norm{h'}_{L^\infty}\norm b_{C_tL^\infty}
          +\norm{\partial_tb}_{C_tL^\infty}\right)
          \norm{\rho_{\la,\eps}}_{L^1}\\
    &\lesssim_{E,A,\eps}
       (1+\la^\beta)\la^{-\frac12(1-\eps)}.
\end{split}
\end{equation}
Equations~\eqref{eq:wsupport} and \eqref{eq:dwt}, together with
Lemma~\ref{lem:L1wiener}, give, for $S>2$,
\begin{equation}\label{eq:RTestimate}
    \norm{E_T}_{\A^{-S}_{x\neq0,t}}
    \lesssim_{E,A,S,\eps}
       (1+\la^\beta)\la^{-\frac12(1-\eps)},
    \qquad S>2.
\end{equation}
By \eqref{eq:RO}--\eqref{eq:RT}, $E^+=E_O+E_N+E_D+E_T$.  The triangle
inequality, the choice $N=2$ in \eqref{eq:ROestimate}, and the bounds \eqref{eq:RDestimate},
\eqref{eq:RNestimate}, and \eqref{eq:RTestimate} give
\begin{equation}\label{eq:onestep-total-error}
   \norm{E^+}_{\A^{-S}_{x\neq0,t}}
   \leq C_{u,E,A,S,\eps}\bigl(
       \mu^{-1}+\mu^{2-S}
       +(1+\la^\beta)
          \la^{-\frac12(1-\eps)}\bigr).
\end{equation}
For fixed $u,E$, and $A$, the right side tends to zero as
$\la\to\infty$ through admissible values when $S>d+1$ and
$\beta<\frac12(1-\eps)$.

\subsection{Weighted Fourier estimate}

The bound
$\operatorname{AF}_S^x(w)\leq\norm{w}_{\A_t^0}^2$ would produce a term
of size $A$.  Expanding the two factors of $\rho_{\la,\eps}$ separates
the terms with zero total profile frequency.  Their part linear in $E$
contributes exactly $\norm{E}_{\A^{-S}_{x\neq0,t}}$, the remaining terms
are $O(A^{-1})$, and nonzero total profile frequency gives a factor
$\mu^{-S}$.

\begin{proposition}
\label{prop:periodic-absolute-summability}
Let $S>1$, and let $u,E\in C_tC^\infty(\T^2)$ be real functions with fixed
finite spatial Fourier support such that
\[
   \mathbb P_{x\neq0}u=u,
   \qquad \mathbb P_{x\neq0}E=E.
\]
Set
\[
   K=\max\left(\{1\}\cup\left\{\langle(n,m)\rangle:
       (n,m)\in\supp\widehat u\cup\supp\widehat E\right\}\right).
\]
Choose $A$ as in
\eqref{eq:Achoice}, with the constant there large enough that
\begin{equation}\label{eq:periodic-binomial-smallness}
   \frac{\norm{E}_{\A^0_t}}{A}\leq\frac12.
\end{equation}
Let $h=h(t)\in C([0,1])$ be a temporal cutoff satisfying
$0\leq h\leq1$ and $hE=E$, and let
\[
   b=P_{\leq\mu^{1/2}}^{x,y}(A-E)^{1/2},
   \qquad w=hb\rho_{\la,\eps},
\]
where $\rho_{\la,\eps}$ is given by Lemma~\ref{lem:slab}.  Assume
\eqref{eq:periodic-frequency-separation}.  Then
\begin{align}
    \operatorname{AF}_S^x(u,w)
      &\leq C_{u,E,S,\eps}A^{1/2}
        \la^{-\left(\frac{1-\eps}{2}+S\eps\right)},
      \label{eq:absolute-mixed}\\
    \operatorname{AF}_S^x(w)
      &\leq \norm{E}_{\A^{-S}_{x\neq0,t}}
       +C_{E,S}\bigl(A\mu^{-S}+A^{-1}\bigr).
      \label{eq:absolute-pure}
\end{align}
Consequently, for every $\eta>0$, one may first choose $A=A(\eta)$ and then
an admissible $\la=\la(\eta,A)$ sufficiently large so that
\begin{equation}\label{eq:periodic-AF-step}
   \operatorname{AF}_S^x(u+w)
      \leq\operatorname{AF}_S^x(u)
        +\norm{E}_{\A^{-S}_{x\neq0,t}}+\eta.
\end{equation}
\end{proposition}

\begin{proof}
Set $B=hb$ and expand $\operatorname{AF}_S^x(w)$ using the two Fourier
frequencies $j_1,j_2$ of $\rho_{\la,\eps}$.  When $j_1+j_2=0$, the total
frequency is the sum of the two frequencies from $B$.
Since $hE^j=E^j$ whenever $j\geq1$ and
\[
 P_{\leq\mu^{1/2}}^{x,y}1=1,
 \qquad P_{\leq\mu^{1/2}}^{x,y}E=E,
\]
the binomial expansion of $(A-E)^{1/2}$ gives
\begin{equation*}
\begin{split}
   B&=B_0+B_1+B_{\geq2},\\
   B_0&=A^{1/2}h,\\
   B_1&=-\frac{1}{2A^{1/2}}E,\\
   B_{\geq2}
      &=A^{1/2}\sum_{j\geq2}\binom{1/2}{j}(-1)^jA^{-j}
           P_{\leq\mu^{1/2}}^{x,y}(E^j).
\end{split}
\end{equation*}
The series converges absolutely in $\A^S_t$, and both the cutoff and
multiplication by $h$ are contractive there.
The Wiener algebra inequality, the fact that the
cutoff symbol lies between zero and one, and
\eqref{eq:periodic-binomial-smallness} imply
\begin{equation}\label{eq:periodic-amplitude-remainder}
 \norm{B_{\geq2}}_{\A^0_t}
 \leq A^{1/2}\sum_{j\geq2}
          \abs{\binom{1/2}{j}}
          \left(\frac{\norm E_{\A^0_t}}{A}\right)^j
 \leq C\norm E_{\A^0_t}^2A^{-3/2}.
\end{equation}
\begin{equation}\label{eq:periodic-amplitude-A0}
   \norm{B}_{\A^0_t}\leq C_EA^{1/2}.
\end{equation}

For $F,H\in\A^0_t$,
\begin{equation}\label{eq:AF-by-A0}
   \operatorname{AF}_S^x(F,H)
      \leq\norm{F}_{\A^0_t}\norm{H}_{\A^0_t}.
\end{equation}
For $B=B_0+B_1+B_{\geq2}$,
\begin{equation*}
   \operatorname{AF}_S^x(B)
      \leq\sum_{i,j\in\{0,1,\geq2\}}
          \operatorname{AF}_S^x(B_i,B_j).
\end{equation*}
The contribution linear in $E$ has the exact value
\begin{equation*}
   \operatorname{AF}_S^x(B_0,B_1)
       +\operatorname{AF}_S^x(B_1,B_0)
   =\sum_{\substack{k\in\Z^2\\k_x\neq0}}
       \langle k\rangle^{-S}
       \norm{h\widehat E(\cdot,k)}_{C_t}
     =\norm{E}_{\A^{-S}_{x\neq0,t}}.
\end{equation*}
All other nonzero contributions contain either two factors equal to $B_1$ or
at least one factor equal to $B_{\geq2}$.  By
\eqref{eq:AF-by-A0}, \eqref{eq:periodic-amplitude-remainder}, and $A\geq1$,
their sum is bounded by $C_EA^{-1}$.  It follows that
\begin{equation}\label{eq:periodic-amplitude-AF}
   \operatorname{AF}_S^x(B)
      \leq\norm{E}_{\A^{-S}_{x\neq0,t}}+C_{E,S}A^{-1}.
\end{equation}

Write
\begin{equation*}
   \rho_{\la,\eps}(x)=\sum_{j\in\Z}c_je^{2\pi ijx},
   \qquad c_j=\widehat\rho_{\la,\eps}(j).
\end{equation*}
Lemma~\ref{lem:slab} gives
\begin{equation}\label{eq:periodic-profile-coefficients}
\begin{gathered}
   c_j\geq0,\qquad c_{-j}=c_j,\qquad
   c_j=0\quad\text{unless}\quad
       j\in(\mu\Z\setminus\{0\})\cap[-2\nu,2\nu],\\
   \sum_{j\in\Z}c_j^2=1,\qquad
   c_j\leq\norm{\rho_{\la,\eps}}_{L^1(\T)}
       \lesssim_\eps\la^{-\frac12(1-\eps)}.
\end{gathered}
\end{equation}
For $k\in\Z^2$,
\begin{equation}\label{eq:periodic-w-Fourier}
   \widehat w(t,k)
      =\sum_{j\in\Z}c_j
          \widehat B(t,k-(j,0)).
\end{equation}
For $j\in\supp\widehat\rho_{\la,\eps}$, the sets
$(j,0)+\supp\widehat B$ in \eqref{eq:periodic-w-Fourier} are pairwise
disjoint.  Indeed, distinct nonzero frequencies in
$\supp\widehat\rho_{\la,\eps}$ differ by at least
$\mu$, whereas $\widehat B$ is supported where
$\langle\ell\rangle\leq2\mu^{1/2}$, and $4\mu^{1/2}<\mu$ for
$\mu\geq64$.  Thus at most one summand in
\eqref{eq:periodic-w-Fourier} is nonzero for each $k$, and
\begin{equation*}
\begin{split}
   \operatorname{AF}_S^x(w)
   ={}&\sum_{j_1,j_2\in\Z}\abs{c_{j_1}c_{j_2}}
      \sum_{\substack{\ell_1,\ell_2\in\Z^2\\
       j_1+j_2+\ell_{1,x}+\ell_{2,x}\neq0}}
       \langle(j_1+j_2,0)+\ell_1+\ell_2\rangle^{-S}\\
   &\hspace{35mm}\times
       \norm{\widehat B(\cdot,\ell_1)
              \widehat B(\cdot,\ell_2)}_{C_t}.
\end{split}
\end{equation*}

Consider first the terms with $j_1+j_2=0$.  Their inner sum is exactly
$\operatorname{AF}_S^x(B)$: the remaining frequency is
$\ell_1+\ell_2$, and the definition retains exactly the terms with
$\ell_{1,x}+\ell_{2,x}\neq0$; the transverse sum is unrestricted.  The
coefficient of this inner sum is
\begin{equation*}
   \sum_{j\in\Z}\abs{c_jc_{-j}}
      =\sum_{j\in\Z}c_j^2=1.
\end{equation*}
For the remaining terms, group them by $s=j_1+j_2$.  Cauchy--Schwarz and
\eqref{eq:periodic-profile-coefficients} give
\[
 \sum_{j\in\Z}\abs{c_jc_{s-j}}
 \leq\left(\sum_{j\in\Z}c_j^2\right)^{1/2}
      \left(\sum_{j\in\Z}c_{s-j}^2\right)^{1/2}=1.
\]
A nonzero $s$ lies in $\mu\Z$, and on the amplitude support
\[
   \abs{\ell_{1,x}+\ell_{2,x}}
      \leq4\mu^{1/2}\leq\frac12\abs s.
\]
Consequently,
\begin{equation*}
\begin{aligned}
 \sum_{\substack{j_1,j_2\in\Z\\j_1+j_2\neq0}}\!
   \abs{c_{j_1}c_{j_2}}
   \sum_{\ell_1,\ell_2\in\Z^2}
       \langle(j_1+j_2,0)+\ell_1+\ell_2\rangle^{-S}
       \norm{\widehat B(\cdot,\ell_1)
              \widehat B(\cdot,\ell_2)}_{C_t}
 &\leq C_S\norm{B}_{\A^0_t}^2
       \sum_{s\in\mu\Z\setminus\{0\}}\abs s^{-S}\\
 &= C_S\mu^{-S}\norm{B}_{\A^0_t}^2
       \sum_{j\in\Z\setminus\{0\}}\abs j^{-S}\\
 &\leq C_{E,S}A\mu^{-S}.
\end{aligned}
\end{equation*}
Combining the terms with $j_1+j_2=0$ and $j_1+j_2\neq0$ with
\eqref{eq:periodic-amplitude-AF} proves \eqref{eq:absolute-pure}.

For the mixed term, let $q\in\supp\widehat u$,
$\ell\in\supp\widehat B$, and
$j\in\supp\widehat\rho_{\la,\eps}\setminus\{0\}$.  From
$\abs{q_x}\leq K$, $\abs{\ell_x}\leq2\mu^{1/2}$,
$16K\leq\mu^{1/2}$, and $\mu\geq64$, we have
\begin{equation}\label{eq:mixed-frequency-gap}
   \abs{q_x+j+\ell_x}
      \geq\abs j-K-2\mu^{1/2}
      \geq\frac12\abs j.
\end{equation}
Substituting \eqref{eq:periodic-w-Fourier} into
Definition~\ref{def:Q} and using \eqref{eq:mixed-frequency-gap} gives
\begin{equation}\label{eq:mixed-expanded}
\begin{split}
   \operatorname{AF}_S^x(u,w)
      &\leq\sum_{q\in\Z^2}\sum_{j\in\Z\setminus\{0\}}
       \sum_{\ell\in\Z^2}
       \langle q+(j,0)+\ell\rangle^{-S}\abs{c_j}
       \norm{\widehat u(\cdot,q)
              \widehat B(\cdot,\ell)}_{C_t}\\
      &\leq C_S\norm{u}_{\A^0_t}\norm{B}_{\A^0_t}
          \sum_{j\in\Z\setminus\{0\}}\abs{c_j}\abs j^{-S}.
\end{split}
\end{equation}
By \eqref{eq:periodic-profile-coefficients}, every nonzero
$j\in\supp\widehat\rho_{\la,\eps}$ belongs to $\mu\Z$, and
\begin{equation}\label{eq:profile-weighted-l1}
 \sum_{j\in\Z\setminus\{0\}}\abs{c_j}\abs j^{-S}
 \leq\norm{\rho_{\la,\eps}}_{L^1}
      \sum_{p\in\Z\setminus\{0\}}\abs{\mu p}^{-S}
 \leq C_{S,\eps}\la^{-\frac12(1-\eps)}\mu^{-S}.
\end{equation}
Equations \eqref{eq:periodic-amplitude-A0}, \eqref{eq:mixed-expanded}, and
\eqref{eq:profile-weighted-l1} prove \eqref{eq:absolute-mixed}.

Finally, by Definition~\ref{def:Q},
\begin{equation}\label{eq:Qexpansion}
   \operatorname{AF}_S^x(u+w)
      \leq\operatorname{AF}_S^x(u)
       +2\operatorname{AF}_S^x(u,w)+\operatorname{AF}_S^x(w).
\end{equation}
Given $\eta>0$, first increase $A$ so that the binomial series converges and
$C_{E,S}A^{-1}<\eta/3$.  Keeping $A$ fixed, choose an admissible $\la$ so
large that \eqref{eq:periodic-frequency-separation} holds and
\[
   C_{E,S}A\mu^{-S}<\frac\eta3,
   \qquad
   2C_{u,E,S,\eps}A^{1/2}
       \la^{-\left(\frac{1-\eps}{2}+S\eps\right)}<\frac\eta3.
\]
Inserting these estimates into \eqref{eq:Qexpansion} proves
\eqref{eq:periodic-AF-step}.
\end{proof}

\subsection{Periodic quadratic iteration}

\begin{proposition}\label{prop:onestep}
Fix $d\in\{3,5\}$, $\kappa\in\{-1,1\}$, and
\[
   0<\eps<1,
   \qquad 0<\beta<\frac12(1-\eps),
   \qquad S>d+1.
\]
Let $(u,E)$ be a smooth real pair with finite spatial Fourier support that
satisfies \eqref{eq:relaxed-general} with $r=2$.  Assume that both $u$ and
$E$ have zero $x$-mean and are supported in a closed interval
$I\Subset(0,1)$.
Given finitely many $1\leq p_j<2$ and numbers $\eta,\delta_+>0$, there is
$A_0\geq1$ such that the following holds for every $A\geq A_0$ and every
sufficiently large admissible $\la$.  For
$h\in C_c^\infty((0,1))$ satisfying \eqref{eq:time-cutoff-lemma} with
$\tau=\la^{-\beta}$, define $w$ by \eqref{eq:increment} and
$(u^+,E^+)$ by \eqref{eq:periodic-update-general} with $r=2$.  Then
$(u^+,E^+)$ is a smooth real pair with finite spatial Fourier support,
satisfies the relaxed equation, and both $u^+$ and $E^+$ have zero
$x$-mean.  Moreover,
\begin{align}
    \norm{w}_{C_tL^{p_j}}&<\eta
       \quad\text{for every }j,\notag\\
    \norm{E^+}_{\A^{-S}_{x\neq0,t}}&<\delta_+,
       \label{eq:onesteperror}\\
    \operatorname{AF}_S^x(u^+)
       &\leq\operatorname{AF}_S^x(u)
          +\norm{E}_{\A^{-S}_{x\neq0,t}}+\eta.
       \label{eq:onestepQ}
\end{align}
The Fourier support of $w$ satisfies \eqref{eq:wsupport} and is disjoint
from $\supp\widehat u\cup\supp\widehat E$.  If
$I=[t_-,t_+]$, then both $u^+$ and $E^+$ are supported in
$(t_--\la^{-\beta},t_++\la^{-\beta})\Subset(0,1)$.
\end{proposition}

\begin{proof}
Set
\[
   K=\max\left(\{1\}\cup\left\{\langle(n,m)\rangle:
       (n,m)\in\supp\widehat u\cup\supp\widehat E\right\}\right).
\]
Choose $A_0$ large enough that every $A\geq A_0$ satisfies
\eqref{eq:Achoice}, \eqref{eq:periodic-binomial-smallness}, and
\begin{equation}\label{eq:onestep-A-choice}
   C_{E,S}A^{-1}<\frac\eta3.
\end{equation}
Fix any $A\geq A_0$ for the rest of the proof.  By \eqref{eq:Achoice},
$A-E\geq A/2$; hence the amplitude is smooth and real, and
Lemma~\ref{lem:amplitude-cutoff} applies.

Choose an admissible $\la$ sufficiently large that Lemma~\ref{lem:slab}
applies, \eqref{eq:periodic-frequency-separation} holds, and
$\la^{-\beta}<\operatorname{dist}(I,\{0,1\})$.  If $I=[t_-,t_+]$,
choose $h$ from \eqref{eq:time-cutoff-lemma} with
$\tau=\la^{-\beta}$.  Since $E$ is supported in $I$, $hE=E$.
Define $b,w,u^+$, and $E^+$ by \eqref{eq:increment} and
\eqref{eq:periodic-update-general}.

By \eqref{eq:wsupport} and \eqref{eq:periodic-frequency-separation},
$\supp\widehat w$ is disjoint from
$\supp\widehat u\cup\supp\widehat E$.

For all sufficiently large admissible $\la$, \eqref{eq:wLp} gives
$\norm w_{C_tL^{p_j}}<\eta$ for every $j$.

The combined error estimate \eqref{eq:onestep-total-error} tends to zero
as $\la\to\infty$ through admissible values, which proves
\eqref{eq:onesteperror}.

Finally, \eqref{eq:absolute-mixed}, \eqref{eq:absolute-pure}, and
\eqref{eq:Qexpansion} give
\begin{equation*}
 \operatorname{AF}_S^x(u^+)
 \leq\operatorname{AF}_S^x(u)
      +\norm E_{\A^{-S}_{x\neq0,t}}
      +C_{E,S}(A^{-1}+A\mu^{-S})
      +2C_{u,E,S,\eps}A^{1/2}
         \la^{-\left(\frac{1-\eps}{2}+S\eps\right)}.
\end{equation*}
Equation~\eqref{eq:onestep-A-choice} controls the $A^{-1}$ term.  With
$A$ fixed, the other two terms tend to zero as $\la\to\infty$ through
admissible values, proving \eqref{eq:onestepQ}.
\end{proof}

\section{\texorpdfstring{Whole-space localization and the $x$-frequency gap}
{Whole-space localization and the x-frequency gap}}

We use the Euclidean--periodic decoupling inequality~\cite[Lemma~2.11]{KSflex}.  In the form needed here, for
$1\leq p\leq\infty$, $a\in W^{3,p}(\R^2)$, and a periodically extended
$U\in L^p(\T^2)$, it gives
\begin{equation}\label{eq:Euclidean-periodic-decoupling}
    \norm{aU}_{L^p(\R^2)}
    \lesssim
    \norm{a}_{W^{3,p}(\R^2)}
    \norm{U}_{L^p(\T^2)}.
\end{equation}
The obstruction to a compactly supported amplitude is the transverse term
in the error:
\[
    \kappa\partial_x^{-2}\partial_y^2w,
    \qquad
    \widehat{\bigl(\kappa\partial_x^{-2}\partial_y^2w\bigr)}(\xi,\eta)
    =\kappa\frac{\eta^2}{\xi^2}\widehat w(\xi,\eta),
\]
which is singular on the entire hyperplane $\{\xi=0\}$.  If one takes
$w=a(x,y)e^{2\pi iKx}$ with $a\in C_c^\infty(\R^2)$, then
\[
    \widehat w(0,\eta)=\widehat a(-K,\eta).
\]
If $\widehat a(-K,\eta_0)\neq0$ for some $\eta_0\neq0$, continuity implies
that $(\eta^2/\xi^2)\widehat w(\xi,\eta)$ is not locally integrable near
$(0,\eta_0)$.

We therefore project the localized amplitude to low frequencies.  The
resulting Schwartz perturbation is supported away from $\{\xi=0\}$, where
$\partial_x^{-1}$ and $\partial_x^{-2}$ are well defined.

\subsection{Whole-space Wiener estimates and cutoff independence}

We measure errors in differentiated form:
\begin{equation}\label{eq:Euclidean-error-norm}
    \norm{\partial_xF}_{\A^{-S}_{\R^2,t}}
    =2\pi\int_{\R^2}\abs\xi\langle k\rangle^{-S}
      \norm{\widehat F(\cdot,k)}_{C_t}\,dk.
\end{equation}
For the transverse part, differentiation leaves the symbol
$\eta^2/\abs\xi$; the exact gap and the cone estimate
\eqref{eq:Euclidean-cone-ratio} control this weight.

If $u\in C_tL^1(\R^2)$ and
$\operatorname{AF}^{\partial_x}_{S,r}(u)<\infty$, then
$t\mapsto\partial_x[u^r](t)$ is continuous in $\A^{-S}_{\R^2}$.
Indeed, let $t_j\to t$.  Then
$\widehat u(t_j,k)\to\widehat u(t,k)$ for every
$k$.  The difference of the products in
\eqref{eq:Euclidean-nonlinearity} is bounded by
$2\prod_{j=1}^r\norm{\widehat u(\cdot,k_j)}_{C_t}$, the integrable
majorant in \eqref{eq:Euclidean-Q}.  Dominated convergence against
\eqref{eq:Euclidean-Q}, followed by the change of variables
$k=k_1+\cdots+k_r$, therefore gives
\[
 \norm{\partial_x[u^r](t_j)-\partial_x[u^r](t)}
       _{\A^{-S}_{\R^2}}\longrightarrow0.
\]
If $\widehat u(t,\xi,\eta)=0$ for $\abs\xi<c_0$, then, for every
$\phi\in\mathcal S(\R^2)$ and every fixed time,
\begin{equation}\label{eq:Euclidean-transverse-pairing}
 \abs{\ip{\partial_x^{-1}\partial_y^2u}{\phi}}
 \leq \frac{2\pi}{c_0}\norm{u}_{L^1}
 \int_{\R^2}\eta^2\abs{\widehat\phi(\xi,\eta)}\,d\xi\,d\eta.
\end{equation}
Thus the transverse term is a tempered distribution.  For fixed $c_0>0$,
the map $u\mapsto\partial_x^{-1}\partial_y^2u$ is continuous from
$\{u\in L^1:\widehat u=0\text{ for }\abs\xi<c_0\}$ to
$\mathcal S'(\R^2)$.

\begin{proposition}
\label{prop:Euclidean-cutoff}
Fix $r\in\{2,3\}$ and $S>0$.  Suppose $u\in C_tL^1(\R^2)$ satisfies
$\operatorname{AF}^{\partial_x}_{S,r}(u)<\infty$.  Let $P_N$ have a time-independent
multiplier $m_N$ with $\sup_N\norm{m_N}_{L^\infty}<\infty$ and
$m_N(k)\to1$ for almost every $k$.  In \eqref{eq:Euclidean-nonlinearity},
the integrand defining $\partial_x[(P_Nu)^r]$ contains the additional
factor $\prod_{j=1}^rm_N(k_j)$.  Then
\begin{equation}\label{eq:Euclidean-cutoff-independence}
 \partial_x[(P_Nu)^r]\longrightarrow\partial_x[u^r]
 \quad\text{in }\A^{-S}_{\R^2,t}.
\end{equation}
In particular, this includes every spatial mollifier
$m_\eps(k)=\widehat\rho(\eps k)$ with $\rho\in\mathcal S(\R^2)$ and
$\int\rho=1$.  If in addition $u\in C_tL^r(\R^2)$, then
$\partial_x[u^r]=\partial_x(u^r)$ in the ordinary distributional sense.
\end{proposition}

\begin{proof}
After the change of variables
$k=k_1+\cdots+k_r$, Minkowski's integral inequality bounds the norm of
the difference in \eqref{eq:Euclidean-cutoff-independence} by
\[
 2\pi\int_{(\R^2)^r}\abs{\xi_1+\cdots+\xi_r}
 \langle k_1+\cdots+k_r\rangle^{-S}
 \abs{\prod_{j=1}^rm_N(k_j)-1}
 \norm{\prod_{j=1}^r\widehat u(\cdot,k_j)}_{C_t}
 \prod_{j=1}^rdk_j.
\]
The integrand tends to zero almost everywhere and is bounded by
$\bigl((\sup_N\norm{m_N}_{L^\infty})^r+1\bigr)$ times the integrand in
\eqref{eq:Euclidean-Q}.  Dominated convergence proves
\eqref{eq:Euclidean-cutoff-independence}.  For spatial mollifiers,
$u*\rho_\eps\in C_t(L^1\cap L^\infty)$, so
$(u*\rho_\eps)^r\in C_tL^1$.  If $u\in C_tL^r$, take
$P_Nu=u*\rho_{\eps_N}$ with $\eps_N\downarrow0$.  Then $P_Nu\to u$ in
$C_tL^r$, and H\"older's inequality gives
\[
 \norm{(P_Nu)^r-u^r}_{C_tL^1}
 \leq\sum_{j=0}^{r-1}
  \norm{P_Nu}_{C_tL^r}^{r-1-j}\norm u_{C_tL^r}^j
  \norm{P_Nu-u}_{C_tL^r}\longrightarrow0.
\]
Thus the distributional limit is $\partial_x(u^r)$, while
\eqref{eq:Euclidean-cutoff-independence} gives $\partial_x[u^r]$.
\end{proof}

\subsection{A spatial cutoff with nonnegative Fourier transform}

\begin{lemma}
\label{lem:Euclidean-positive-approximation}
There exists a real, even function $\chi\in C_c^\infty(\R^2)$ such that
\begin{equation*}
   0\leq\chi\leq1,
   \qquad \chi(0)=1,
   \qquad \widehat\chi\geq0,
   \qquad \int_{\R^2}\widehat\chi(k)\,dk=1.
\end{equation*}
Put $\chi_L(z)=\chi(z/L)$.  For every $L\geq1$,
\begin{equation}\label{eq:Euclidean-spatial-cutoff-moment}
\begin{gathered}
   \widehat{\chi_L}\geq0,
   \qquad
   \int_{\R^2}\widehat{\chi_L}(k)\,dk=1,\\
   \int_{\R^2}\abs z
      \bigl(\widehat{\chi_L}*\widehat{\chi_L}\bigr)(z)\,dz
      \leq\frac{C_\chi}{L}.
\end{gathered}
\end{equation}
If $S\geq1$ and $F\in L^1(\R^2)$ is nonnegative, then
\begin{equation*}
 2\pi\int_{\R^2}\abs\xi\langle k\rangle^{-S}
       \bigl[(\widehat{\chi_L}*\widehat{\chi_L})*F\bigr](k)\,dk
 \leq
  2\pi\int_{\R^2}\abs\xi\langle k\rangle^{-S}F(k)\,dk
  +\frac{C_{\chi,S}}{L}\norm F_{L^1(\R^2)}.
\end{equation*}
For $A>0$,
\begin{equation}\label{eq:localized-constant-scaling}
    \norm{\partial_x(A\chi_L^2)}_{\A^{-S}_{\R^2}}
      \leq \frac{C_\chi A}{L}.
\end{equation}
\end{lemma}

\begin{proof}
Choose a nonzero, nonnegative, real, even function
$\psi\in C_c^\infty(\R^2)$ and set
\[
   \chi=\frac{\psi*\psi}{\norm\psi_{L^2}^2}.
\]
Then $\chi$ is real, even, nonnegative, and compactly supported.  Moreover,
\[
   \chi(0)=\frac{1}{\norm\psi_{L^2}^2}
       \int_{\R^2}\psi(y)^2\,dy=1,
\]
and Cauchy--Schwarz gives, for every $z\in\R^2$,
\[
   0\leq\chi(z)
    =\frac{1}{\norm\psi_{L^2}^2}
       \int_{\R^2}\psi(y)\psi(y-z)\,dy
    \leq1.
\]
The Fourier transform satisfies
\[
   \widehat\chi(k)
      =\frac{\abs{\widehat\psi(k)}^2}{\norm\psi_{L^2}^2}\geq0,
   \qquad
   \int_{\R^2}\widehat\chi(k)\,dk=\chi(0)=1.
\]
The Fourier scaling formula gives
\[
   \widehat{\chi_L}(k)=L^2\widehat\chi(Lk)\geq0,
   \qquad \int_{\R^2}\widehat{\chi_L}(k)\,dk=1.
\]
Furthermore,
\[
 \int_{\R^2}\abs z
      \bigl(\widehat{\chi_L}*\widehat{\chi_L}\bigr)(z)\,dz
 =\frac1L\int_{\R^2}\abs z
      \bigl(\widehat\chi*\widehat\chi\bigr)(z)\,dz
 \leq\frac{C_\chi}{L}.
\]
If $S\geq1$, the function
$k\mapsto2\pi\abs\xi\langle k\rangle^{-S}$ is bounded and globally
Lipschitz.  Away from $\xi=0$,
\[
   \abs{\nabla\bigl(2\pi\abs\xi\langle k\rangle^{-S}\bigr)}
   \leq C_S\left(\langle k\rangle^{-S}
       +\abs\xi\abs k\langle k\rangle^{-S-2}\right)
   \leq C_S.
\]
Integrating along a line segment, splitting once at $\xi=0$ if necessary,
gives
\begin{equation*}
   2\pi\abs{\xi+z_x}\langle k+z\rangle^{-S}
   \leq2\pi\abs\xi\langle k\rangle^{-S}+C_S\abs z
   \qquad(k,z\in\R^2).
\end{equation*}
Tonelli's theorem and \eqref{eq:Euclidean-spatial-cutoff-moment} give
\begin{align*}
 &2\pi\int_{\R^2}\abs\xi\langle k\rangle^{-S}\,
       \bigl[(\widehat{\chi_L}*\widehat{\chi_L})*F\bigr](k)\,dk\\
 &=\iint_{\R^2\times\R^2}
       2\pi\abs{(r+z)_x}\langle r+z\rangle^{-S}
       \bigl(\widehat{\chi_L}*\widehat{\chi_L}\bigr)(z)
       F(r)\,dz\,dr\\
 &\leq2\pi\int_{\R^2}\abs{r_x}\langle r\rangle^{-S}F(r)\,dr
   +C_S\iint_{\R^2\times\R^2}
       \abs z
       \bigl(\widehat{\chi_L}*\widehat{\chi_L}\bigr)(z)
       F(r)\,dz\,dr\\
 &\leq2\pi\int_{\R^2}\abs{r_x}\langle r\rangle^{-S}F(r)\,dr
   +\frac{C_{\chi,S}}{L}\norm F_{L^1}.
\end{align*}
Finally,
\[
 \norm{\partial_x(A\chi_L^2)}_{\A^{-S}_{\R^2}}
 =2\pi A\int_{\R^2}\abs\xi\langle k\rangle^{-S}
       \bigl(\widehat{\chi_L}*\widehat{\chi_L}\bigr)(k)\,dk
 \leq\frac{C_\chi A}{L}.
\]
\end{proof}

\subsection{The localized amplitude and frequency gap}

For $M\in\R$, set
\[
    \norm{F}_{\A^M_{\R^2,t}}
    :=\int_{\R^2}\langle k\rangle^M
      \norm{\widehat F(\cdot,k)}_{C_t}\,dk.
\]
Fix the function $\chi$ from
Lemma~\ref{lem:Euclidean-positive-approximation} and put
$\chi_L(z)=\chi(z/L)$.
Let $P_{\leq\ell}^{x,y}$ have a real, even multiplier
$m_\ell\in C_c^\infty(\R^2)$ with
\begin{equation*}
    0\leq m_\ell\leq1,
    \qquad m_\ell=1\ \text{when }\abs k\leq\ell,
    \qquad \supp m_\ell\subset\{k\in\R^2:\abs k\leq2\ell\}.
\end{equation*}
Fix $d\in\{3,5\}$ and $\kappa\in\{-1,1\}$.  Suppose a smooth real pair
$(u,E)$ satisfies \eqref{eq:relaxed-general} with $r=2$
on $[0,1]\times\R^2$.  We assume that $u$ has compact Fourier support away
from $\{\xi=0\}$ and that
\[
    E\in C([0,1];\mathcal S(\R^2)).
\]
Fix $M>S+6$ and choose
\begin{equation}\label{eq:Euclidean-A-choice}
    A\geq C_M\left(1+\norm{E}_{C_tL^\infty}
                    +\norm{E}_{\A^M_{\R^2,t}}\right),
\end{equation}
where $C_{\mathrm{alg},M}$ denotes a weighted Wiener algebra product
constant and $C_M$ is chosen so that
$C_{\mathrm{alg},M}\norm{E}_{\A^M_{\R^2,t}}/A\leq1/2$.
Then
$\sum_{j\geq1}\binom{1/2}{j}(-E/A)^j$ converges absolutely in
$\A^M_{\R^2,t}$.  Define
\begin{equation}\label{eq:Euclidean-amplitude}
    a_L=\chi_L(A-E)^{1/2},
    \qquad a_{L,\ell}=P_{\leq\ell}^{x,y}a_L.
\end{equation}

Keep $0<\eps<1$ fixed.  For an admissible $\la$ so large that
$\la^\eps\geq8\ell$, set
\[
    \mu=\la^\eps,
    \qquad \nu=\la^{1+\eps}=\la\mu,
\]
take the profile $\rho_{\la,\eps}$ from Lemma~\ref{lem:slab}, extended
periodically in $x$, and choose $0\leq h\leq1$ with $h=1$ on
$\supp_tE$.  Set
\begin{equation}\label{eq:Euclidean-increment}
    w(t,x,y)=h(t)a_{L,\ell}(t,x,y)\rho_{\la,\eps}(x).
\end{equation}
Since $\supp\widehat{a_{L,\ell}}
\subset\{k\in\R^2:\abs k\leq2\ell\}$,
\begin{equation}\label{eq:Euclidean-cone}
    \supp\widehat w(t)\subset
    \left\{(\xi,\eta):
       \mu-2\ell\leq\abs\xi\leq2\nu+2\ell,
       \quad \abs\eta\leq2\ell
    \right\}.
\end{equation}
Thus
\begin{equation}\label{eq:Euclidean-cone-ratio}
    \abs\xi\geq\frac\mu2,
    \qquad \abs{\frac\eta\xi}\leq\frac{4\ell}{\mu}.
\end{equation}

Lemma~\ref{lem:slab} and \eqref{eq:Euclidean-periodic-decoupling} give
\begin{equation}\label{eq:Euclidean-Lp}
    \norm{w}_{C_tL^p(\R^2)}
    \lesssim
    \norm{h a_{L,\ell}}_{C_tW^{3,p}(\R^2)}
    \la^{(1-\eps)(1/2-1/p)},
    \qquad 1\leq p<2.
\end{equation}
Since $h=1$ on $\supp_tE$,
\begin{equation}\label{eq:Euclidean-cancellation}
 E+w^2=h^2\Bigl((1-\chi_L^2)E+A\chi_L^2
 +(a_{L,\ell}^2-a_L^2)+a_{L,\ell}^2(\rho_{\la,\eps}^2-1)\Bigr).
\end{equation}

\section{Whole-space absolute Fourier estimates}

\subsection{The localized square-root amplitude}

Spatial localization replaces the exact amplitude identity by an
approximate one.  The two terms linear in $E$ have combined absolute
coefficient one; the remaining terms are controlled by the first moment
of the cutoff and the higher powers of $E/A$.

\begin{proposition}
\label{prop:Euclidean-amplitude-exact-coefficient}
Fix $S\geq1$ and $M>S+6$.  Let $E$ be real with
$\norm{E}_{\A^M_{\R^2,t}}<\infty$, and let $0\leq h\leq1$ be a smooth
function of time satisfying $hE=E$.  Choose $A$ as in
\eqref{eq:Euclidean-A-choice}, with the constant there large enough that
\begin{equation}\label{eq:Euclidean-amplitude-smallness}
   A\geq2\norm E_{C_tL^\infty},
   \qquad
   \frac{C_{\mathrm{alg},M}\norm E_{\A^M_{\R^2,t}}}{A}\leq\frac12.
\end{equation}
For $L,\ell\geq1$, define $a_L$ and $a_{L,\ell}$ by
\eqref{eq:Euclidean-amplitude}.  Then
\begin{equation}\label{eq:Euclidean-amplitude-exact-coefficient}
   \operatorname{AF}^{\partial_x}_S(h a_{L,\ell})
   \leq
   \norm{\partial_xE}_{\A^{-S}_{\R^2,t}}
   +C_{E,S,M,\chi}
      \left(\frac{A}{L}+\frac1L+\frac1A\right).
\end{equation}
\end{proposition}

\begin{proof}
The lower bound for $A$ gives $A-E\geq A/2$.  The binomial series
\begin{equation}\label{eq:Euclidean-sqrt-expansion}
   a_L
   =\sum_{j=0}^{\infty}\binom{1/2}{j}(-1)^j
       A^{1/2-j}\chi_LE^j
\end{equation}
converges absolutely in $\A^M_{\R^2,t}$, and
$h a_{L,\ell}\in C_tL^1(\R^2)\cap\A^0_{\R^2,t}$.  Since $S\geq1$,
\[
 \operatorname{AF}^{\partial_x}_S(h a_{L,\ell})
 \leq C_S\norm{h a_{L,\ell}}_{\A^0_{\R^2,t}}^2<\infty.
\]

To see that the convergence of the series is
uniform in $L\geq1$, $\widehat\chi\geq0$ and scaling give
\[
 \norm{\chi_L}_{\A^M_{\R^2}}
 =\int_{\R^2}\langle k\rangle^M L^2\widehat\chi(Lk)\,dk
 =\int_{\R^2}\langle z/L\rangle^M\widehat\chi(z)\,dz
 \leq\int_{\R^2}\langle z\rangle^M\widehat\chi(z)\,dz.
\]
Iterating the weighted Wiener algebra inequality therefore gives, for
every $j\geq1$,
\[
 \norm{\chi_LE^j}_{\A^M_{\R^2,t}}
 \leq C_{\chi,M}C_{\mathrm{alg},M}^j
       \norm E_{\A^M_{\R^2,t}}^j.
\]
The second condition in \eqref{eq:Euclidean-amplitude-smallness} gives
convergence in $\A^M_{\R^2,t}$ uniformly for $L\geq1$.
In particular, the series estimate and the contraction of
$P_{\leq\ell}^{x,y}$ give
\begin{equation}\label{eq:Euclidean-amplitude-Wiener-bounds}
 \norm{a_L}_{\A^M_{\R^2,t}}
 +\norm{a_{L,\ell}}_{\A^0_{\R^2,t}}
 +\norm{h a_{L,\ell}}_{\A^0_{\R^2,t}}
 \leq C_{E,M,\chi}A^{1/2}.
\end{equation}

Since $h a_{L,\ell}=P_{\leq\ell}^{x,y}(h a_L)$ and
$\abs{m_\ell}\leq1$,
\begin{equation}\label{eq:Euclidean-AF-cutoff-contraction}
   \operatorname{AF}^{\partial_x}_S(h a_{L,\ell})
      \leq\operatorname{AF}^{\partial_x}_S(h a_L).
\end{equation}

Because $hE=E$, one has $hE^j=E^j$ for every $j\geq1$.  Multiplying
\eqref{eq:Euclidean-sqrt-expansion} by $h$ therefore yields
\begin{equation}\label{eq:Euclidean-h-amplitude-series}
 h a_L=A^{1/2}h\chi_L-\frac{1}{2A^{1/2}}\chi_LE
 +\sum_{j\geq2}\binom{1/2}{j}(-1)^j
        A^{1/2-j}\chi_LE^j.
\end{equation}

The term of degree zero in $E$ satisfies
\begin{align*}
 2\pi A\iint_{\R^2\times\R^2}
   \abs{\xi_1+\xi_2}\langle k_1+k_2\rangle^{-S}
   \widehat{\chi_L}(k_1)\widehat{\chi_L}(k_2)
   \,dk_1\,dk_2
 &=2\pi A\int_{\R^2}\abs\xi\langle k\rangle^{-S}
      \bigl(\widehat{\chi_L}*\widehat{\chi_L}\bigr)(k)\,dk\\
 &\leq\frac{C_\chi A}{L}.
\end{align*}

The two terms linear in $E$ have combined absolute coefficient one.
By Minkowski's integral inequality and
$\widehat{\chi_L}\geq0$,
\begin{equation*}
   \norm{\widehat{\chi_LE}(\cdot,k)}_{C_t}
      \leq\int_{\R^2}\widehat{\chi_L}(z)
         \norm{\widehat E(\cdot,k-z)}_{C_t}\,dz.
\end{equation*}
The two cross terms have identical nonnegative majorants.  Set
$F(k)=\norm{\widehat E(\cdot,k)}_{C_t}$.  Tonelli's theorem and a change of
variables, followed by
Lemma~\ref{lem:Euclidean-positive-approximation}, bound the linear
contribution by
\[
 2\pi\int_{\R^2}\abs\xi\langle k\rangle^{-S}
   \bigl[(\widehat{\chi_L}*\widehat{\chi_L})*F\bigr](k)\,dk
 \leq \norm{\partial_xE}_{\A^{-S}_{\R^2,t}}
   +\frac{C_{\chi,S}}{L}\norm E_{\A^0_{\R^2,t}}.
\]
Set $R=\norm E_{\A^0_{\R^2,t}}$.  The Wiener algebra inequality gives
$\norm{E^j}_{\A^0_{\R^2,t}}\leq R^j$.  Since
$\widehat{\chi_L}\geq0$ and $\int\widehat{\chi_L}=1$, Minkowski's inequality
and Tonelli's theorem imply, for every $j\geq1$,
\begin{align*}
 \norm{\chi_LE^j}_{\A^0_{\R^2,t}}
 &\leq\iint_{\R^2\times\R^2}\widehat{\chi_L}(z)
    \norm{\widehat{E^j}(\cdot,k-z)}_{C_t}\,dz\,dk\\
 &=\norm{E^j}_{\A^0_{\R^2,t}}
 \leq R^j.
\end{align*}
For $S\geq1$, the multiplier
$2\pi\abs\xi\langle k\rangle^{-S}$ is bounded.  Hence the terms of total
degree at least two in $E$ are bounded by
\[
 C_SA\sum_{\substack{i,j\geq0\\i+j\geq2}}
   \abs{\binom{1/2}{i}}\abs{\binom{1/2}{j}}
   \left(\frac RA\right)^{i+j}
 \leq C_SA\left(\frac RA\right)^2
 \leq\frac{C_{E,S}}{A}.
\]

Combining the constant, linear, and higher-order terms in the expansion
with \eqref{eq:Euclidean-AF-cutoff-contraction} proves
\eqref{eq:Euclidean-amplitude-exact-coefficient}.
\end{proof}

\subsection{\texorpdfstring{The term $w^2$}{The term w squared}}

Write the periodic profile as
\begin{equation}\label{eq:Euclidean-localized-perturbation-Fourier}
    \rho_{\la,\eps}(x)
    =\sum_{n\in\mu\Z\setminus\{0\}}c_ne^{2\pi inx},
    \qquad
    c_{-n}=\overline{c_n},
    \qquad \sum_n\abs{c_n}^2=1,
    \qquad c_n=0\quad\text{for }\abs n>2\nu.
\end{equation}

\begin{proposition}

Let the hypotheses of
Proposition~\ref{prop:Euclidean-amplitude-exact-coefficient} hold with
$S>2$, let
$\rho_{\la,\eps}$ be given by Lemma~\ref{lem:slab}, and suppose
$\mu\geq8\ell$.  For
\[
   w=h a_{L,\ell}\rho_{\la,\eps},
\]
\begin{equation}\label{eq:Euclidean-pure-exact-coefficient}
 \operatorname{AF}^{\partial_x}_S(w)
 \leq\norm{\partial_xE}_{\A^{-S}_{\R^2,t}}
   +C_{E,S,M,\chi}\left(
       \frac{A}{L}+\frac1L+\frac1A+A\mu^{1-S}\right).
\end{equation}
\end{proposition}

\begin{proof}
The Fourier support of $h a_{L,\ell}$ is contained in
$\{k\in\R^2:\abs k\leq2\ell\}$, while
distinct frequencies $n$ in \eqref{eq:Euclidean-localized-perturbation-Fourier} differ by
at least $\mu\geq8\ell$.  Hence the translates of
$\supp\widehat{h a_{L,\ell}}$ by these frequencies are pairwise disjoint,
and every point of $\supp\widehat w$ belongs to a unique translate.
Consequently,
\begin{align*}
 \operatorname{AF}^{\partial_x}_S(w)
 ={}&2\pi\sum_{n_1,n_2}\abs{c_{n_1}c_{n_2}}
 \iint_{\R^2\times\R^2}
   \abs{n_1+n_2+\xi_1+\xi_2}\\
 &\quad\times
   \left\langle(n_1+n_2,0)+k_1+k_2\right\rangle^{-S}
   \norm{\widehat{h a_{L,\ell}}(\cdot,k_1)
          \widehat{h a_{L,\ell}}(\cdot,k_2)}_{C_t}
   \,dk_1\,dk_2.
\end{align*}

When $n_1+n_2=0$, the inner integral is
$\operatorname{AF}^{\partial_x}_S(h a_{L,\ell})$, with coefficient one by
\eqref{eq:Euclidean-localized-perturbation-Fourier}.
Proposition~\ref{prop:Euclidean-amplitude-exact-coefficient} therefore bounds
the part with $n_1+n_2=0$ by
\[
 \norm{\partial_xE}_{\A^{-S}_{\R^2,t}}
 +C_{E,S,M,\chi}\left(\frac AL+\frac1L+\frac1A\right).
\]

For the remaining frequency pairs, Cauchy--Schwarz gives
\[
 \sum_n\abs{c_nc_{s-n}}
 \leq\left(\sum_n\abs{c_n}^2\right)^{1/2}
      \left(\sum_n\abs{c_{s-n}}^2\right)^{1/2}=1.
\]
Consequently, if $S>2$,
\begin{equation}\label{eq:Euclidean-frequency-sum-bound}
   \sum_{s\in\mu\Z\setminus\{0\}}
      \left(\sum_n\abs{c_nc_{s-n}}\right)\abs{s}^{1-S}
   \leq\mu^{1-S}\sum_{j\in\Z\setminus\{0\}}\abs j^{1-S}
   \leq C_S\mu^{1-S}.
\end{equation}
For $s\neq0$ and $\abs{k_1},\abs{k_2}\leq2\ell$, one has
\[
   \abs{k_1+k_2}\leq4\ell\leq\frac\mu2\leq\frac{\abs s}{2}.
\]
It follows that
\[
 \abs{s+\xi_1+\xi_2}
 \left\langle(s,0)+k_1+k_2\right\rangle^{-S}
 \leq C_S\abs{s}^{1-S}.
\]
Inserting this bound in the exact Fourier expansion, grouping by
$s=n_1+n_2$, and applying \eqref{eq:Euclidean-frequency-sum-bound}, the
part with $n_1+n_2\neq0$ is at most
\[
 C_S\mu^{1-S}
    \norm{h a_{L,\ell}}_{\A^0_{\R^2,t}}^2.
\]
Finally, \eqref{eq:Euclidean-amplitude-Wiener-bounds} gives
$\norm{h a_{L,\ell}}_{\A^0_{\R^2,t}}\leq C_EA^{1/2}$.
Adding the contributions with $n_1+n_2=0$ and $n_1+n_2\neq0$ gives
\eqref{eq:Euclidean-pure-exact-coefficient}.
\end{proof}

\subsection{The mixed term}

In $uw$, the frequency $n\in\supp\widehat\rho_{\la,\eps}$ cannot cancel.
Hence the output $x$-frequency is comparable to $n$, and
$2\pi\abs\xi\langle k\rangle^{-S}\lesssim_S\abs n^{1-S}$.

\begin{proposition}
\label{prop:Euclidean-mixed-absolute}
Fix $0<\eps<1$ and $S>2$.  Let $u\in C_tL^1(\R^2)$ satisfy
\[
   \norm{u}_{\A^0_{\R^2,t}}<\infty,
   \qquad
   \operatorname{AF}^{\partial_x}_S(u)<\infty,
   \qquad
   \widehat u(t,\xi,\eta)=0\quad\text{whenever }\abs\xi>K
\]
for every $t\in[0,1]$ and some $K\geq1$.  Under the hypotheses of
Proposition~\ref{prop:Euclidean-amplitude-exact-coefficient}, let
\[
   w=h a_{L,\ell}\rho_{\la,\eps},
\]
where $\rho_{\la,\eps}$ is given by Lemma~\ref{lem:slab}, and suppose
\begin{equation}\label{eq:Euclidean-mixed-separation}
   \mu\geq8\ell,
   \qquad
   \mu>8K.
\end{equation}
Then the projections of $\supp\widehat u$ and $\supp\widehat w$ onto the
$\xi$-axis are disjoint, and
\begin{equation}\label{eq:Euclidean-mixed-absolute}
\begin{split}
 \operatorname{AF}^{\partial_x}_S(u,w)
 &\leq C_{S,\eps}
   \norm{u}_{\A^0_{\R^2,t}}
   \norm{h a_{L,\ell}}_{\A^0_{\R^2,t}}
   \la^{-\frac12(1-\eps)}\mu^{1-S}\\
 &\leq C_{u,E,S,M,\eps,\chi}A^{1/2}
   \la^{-\frac12(1-\eps)-\eps(S-1)}.
\end{split}
\end{equation}
Consequently,
\begin{equation}\label{eq:Euclidean-combined-absolute}
\begin{split}
 \operatorname{AF}^{\partial_x}_S(u+w)
 \leq{}&\operatorname{AF}^{\partial_x}_S(u)
   +\norm{\partial_xE}_{\A^{-S}_{\R^2,t}}\\
 &+C_{E,S,M,\chi}\left(
       \frac A L+\frac1L+\frac1A+A\mu^{1-S}\right)\\
 &+C_{u,E,S,M,\eps,\chi}A^{1/2}
   \la^{-\frac12(1-\eps)-\eps(S-1)}.
\end{split}
\end{equation}
\end{proposition}

\begin{proof}
The Fourier expansion of the perturbation is
\begin{equation}\label{eq:Euclidean-mixed-w-Fourier}
 \widehat w(t,k)
 =\sum_{n\in\mu\Z\setminus\{0\}}c_n
    \widehat{h a_{L,\ell}}\bigl(t,k-(n,0)\bigr).
\end{equation}
Let $k=(\xi,\eta)$ belong to the Fourier support of $u$, let
$z=(z_x,z_y)$ belong to the Fourier support of $h a_{L,\ell}$, and let
$n$ be a nonzero Fourier frequency in
\eqref{eq:Euclidean-mixed-w-Fourier}.  Then
\[
   \abs\xi\leq K,
   \qquad
   \abs{z_x}\leq2\ell,
   \qquad
   \abs n\geq\mu.
\]
The assumption $\norm u_{\A^0_{\R^2,t}}<\infty$ and Fourier inversion give
$u\in C_tL^\infty$; hence $uw\in C_tL^1$ and
$\widehat{uw}=\widehat u*\widehat w$.
By \eqref{eq:Euclidean-mixed-separation},
\begin{equation}\label{eq:Euclidean-mixed-frequency-gap}
 \frac58\abs n\leq\abs{\xi+n+z_x}\leq\frac{11}{8}\abs n.
\end{equation}
The same separation makes the projections of $\supp\widehat u$ and
$\supp\widehat w$ onto the $\xi$-axis disjoint.  Moreover,
\eqref{eq:Euclidean-mixed-frequency-gap} gives
\begin{equation}\label{eq:Euclidean-mixed-weight}
 2\pi\abs{\xi+n+z_x}
   \left\langle k+(n,0)+z\right\rangle^{-S}
 \leq C_S\abs n^{1-S}.
\end{equation}

Insert the finite expansion \eqref{eq:Euclidean-mixed-w-Fourier} into
\eqref{eq:Euclidean-Q}, apply the triangle inequality inside
the $C_t$ norm, and change variables in the second frequency integral.
Tonelli's theorem and \eqref{eq:Euclidean-mixed-weight} give
\begin{align*}
 \operatorname{AF}^{\partial_x}_S(u,w)
 &\leq2\pi\sum_{n\in\mu\Z\setminus\{0\}}\abs{c_n}
   \iint_{\R^2\times\R^2}
   \frac{\abs{\xi+n+z_x}}
        {\left\langle k+(n,0)+z\right\rangle^S}
   \norm{\widehat u(\cdot,k)
     \widehat{h a_{L,\ell}}(\cdot,z)}_{C_t}\,dk\,dz\\
 &\leq C_S\sum_{n\in\mu\Z\setminus\{0\}}
      \abs{c_n}\abs n^{1-S}
   \iint_{\R^2\times\R^2}
     \norm{\widehat u(\cdot,k)}_{C_t}
     \norm{\widehat{h a_{L,\ell}}(\cdot,z)}_{C_t}
       \,dk\,dz\\
 &=C_S\norm{u}_{\A^0_{\R^2,t}}
      \norm{h a_{L,\ell}}_{\A^0_{\R^2,t}}
      \sum_{n\in\mu\Z\setminus\{0\}}
         \abs{c_n}\abs n^{1-S}.
\end{align*}
By \eqref{eq:slab-coefficient-bound}, every Fourier coefficient satisfies
\[
   \abs{c_n}\leq C_\eps\la^{-\frac12(1-\eps)}.
\]
Extending the finite sum to the whole nonzero lattice, we obtain
\[
 \sum_{n\in\mu\Z\setminus\{0\}}\abs{c_n}\abs n^{1-S}
 \leq C_\eps\la^{-\frac12(1-\eps)}
   \sum_{j\in\Z\setminus\{0\}}\abs{\mu j}^{1-S}
 \leq C_{S,\eps}\la^{-\frac12(1-\eps)}\mu^{1-S}.
\]
By \eqref{eq:Euclidean-amplitude-Wiener-bounds},
$\norm{h a_{L,\ell}}_{\A^0_{\R^2,t}}\leq C_{E,M,\chi}A^{1/2}$.
Since $\mu=\la^\eps$, these estimates prove
\eqref{eq:Euclidean-mixed-absolute}.

The quadratic expansion of $\operatorname{AF}^{\partial_x}_S(u+w)$, together
with \eqref{eq:Euclidean-mixed-absolute} and
\eqref{eq:Euclidean-pure-exact-coefficient}, gives
\eqref{eq:Euclidean-combined-absolute}.
\end{proof}

\section{Whole-space quadratic iteration}

\begin{proposition}\label{prop:Euclidean-onestep}
Fix $d\in\{3,5\}$, $\kappa\in\{-1,1\}$,
$0<\eps<1$, $0<\beta<\frac12(1-\eps)$, $S>d+1$, and $M>S+6$.
Let $u,E\in C^\infty([0,1];\mathcal S(\R^2))$ be a real pair satisfying
\eqref{eq:relaxed-general} with $r=2$.  Assume that $u$ and $E$ have compact spatial
Fourier support and that
\[
  \widehat u(t,\xi,\eta)=0\quad\text{for }\abs\xi<c_0
\]
for some $c_0>0$, and that
$\supp_tu\cup\supp_tE\subset I$ for a closed interval $I\Subset(0,1)$.
Let
\[
 K=\max\left(\{1,c_0\}\cup
 \{\abs\xi:(\xi,\eta)\in\supp\widehat u\}\right).
\]
Given finitely many exponents $1\leq p_j<2$ and numbers
$\delta_w,\delta_+>0$, one can choose parameters in the order
\[
 A\ \longrightarrow\ L\ \longrightarrow\ \ell\ \longrightarrow\ \la
\]
and a cutoff $h$ satisfying \eqref{eq:time-cutoff-lemma} with
$\tau=\la^{-\beta}$, for which $w=h a_{L,\ell}\rho_{\la,\eps}$ and
$u^+=u+w$, with $E^+$ defined by
\eqref{eq:Euclidean-update-general} for $r=2$, satisfy
\begin{align}
 \norm w_{C_tL^{p_j}(\R^2)}&<\delta_w,
 \label{eq:Euclidean-onestep-Lp}\\
 \norm{\partial_xE^+}_{\A^{-S}_{\R^2,t}}&<\delta_+,
 \notag\\
 \operatorname{AF}^{\partial_x}_S(u^+)
 &\leq\operatorname{AF}^{\partial_x}_S(u)
       +\norm{\partial_xE}_{\A^{-S}_{\R^2,t}}+\delta_w.
 \label{eq:Euclidean-onestep-absolute}
\end{align}
The pair $(u^+,E^+)$ is smooth and real, is Schwartz in space, has compact
spatial Fourier support, satisfies \eqref{eq:relaxed-general} with
$r=2$, and is
supported in the open $\la^{-\beta}$-neighborhood of $I$.  Furthermore,
\begin{equation}\label{eq:Euclidean-onestep-support}
 \supp\widehat w\subset
 \left\{(\xi,\eta):
   \mu-2\ell\leq\abs\xi\leq2\nu+2\ell,
   \ \abs\eta\leq2\ell\right\},
\end{equation}
where
\begin{equation}\label{eq:Euclidean-onestep-separation}
 \mu\geq8\ell,
 \qquad \mu>8K.
\end{equation}
Moreover,
\[
 \widehat{u^+}(t,\xi,\eta)=0\qquad(\abs\xi<c_0),
 \qquad
 \supp\widehat u\cap\supp\widehat w=\varnothing.
\]

Once $A,L,$ and $\ell$ are fixed, every sufficiently large admissible
$\la$ works.
\end{proposition}

\begin{proof}
Fix $(u,E)$, the finite list of exponents, and the two tolerances.
Choose $A\geq1$ so large that \eqref{eq:Euclidean-A-choice},
\eqref{eq:Euclidean-amplitude-smallness}, and
\begin{equation*}
 \frac{C_{E,S,M,\chi}}{A}<\frac{\delta_w}8
\end{equation*}
  hold.  Keeping this $A$ fixed, choose $L\geq1$ sufficiently large that
\begin{equation*}
 \frac{C_{E,\chi}+C_\chi A}{L}<\frac{\delta_+}{8},
 \qquad
 C_{E,S,M,\chi}\left(\frac{A}{L}+\frac1L\right)<\frac{\delta_w}8.
\end{equation*}
After $A$ and $L$ have been fixed, choose $\ell\geq1$ sufficiently large
that $C_{E,M,\chi}A\ell^{-M}<\delta_+/8$.

Choose an admissible $\la$ large enough that
\eqref{eq:Euclidean-onestep-separation} holds and
$\la^{-\beta}<\operatorname{dist}(I,\{0,1\})$.  Choose $h$ satisfying
\eqref{eq:time-cutoff-lemma} with $\tau=\la^{-\beta}$.  Then
\begin{equation*}
 0\leq h\leq1,\qquad h=1\ \text{on }I,\qquad
 \norm{h'}_{L^\infty}\leq C\la^\beta.
\end{equation*}
Since $E$ is supported in $I$, one has $hE=E$.  Define $w,u^+,$ and $E^+$
by \eqref{eq:Euclidean-increment} and
\eqref{eq:Euclidean-update-general}.

\subsubsection*{Dispersion error}

The dispersion error in $\partial_xE^+$ is
\[
 \partial_x\left((-1)^{(d+1)/2}\partial_x^{d-1}w
 +\kappa\partial_x^{-2}\partial_y^2w\right).
\]
By \eqref{eq:Euclidean-mixed-w-Fourier}, the Fourier support of $w$ is
contained in the disks
\[
 \{(\xi,\eta):\abs{(\xi-n,\eta)}\leq2\ell\},
 \qquad n\in\mu\Z\setminus\{0\},\quad \abs n\leq2\nu.
\]
On each disk, $\abs\xi\simeq\abs n$, $\abs\eta\leq2\ell$, and
$\langle(\xi,\eta)\rangle^{-S}\lesssim_S\abs n^{-S}$.
Since $\abs{\widehat w(t,\xi,\eta)}\leq\norm{w(t)}_{L^1}$, integration
over the disks and summation over $n\in\mu\Z\setminus\{0\}$ give
\begin{align*}
 \norm{\partial_x^dw}_{\A^{-S}_{\R^2,t}}
 &\leq C_S\norm w_{C_tL^1}
 \sum_{n\in\mu\Z\setminus\{0\}}
 \int_{\abs{(\xi-n,\eta)}\leq2\ell}
   \abs\xi^d\langle(\xi,\eta)\rangle^{-S}\,d\xi\,d\eta\\
 &\leq C_S\ell^2\norm w_{C_tL^1}
   \sum_{n\in\mu\Z\setminus\{0\}}\abs n^{d-S}
 \leq C_S\ell^2\mu^{d-S}\norm w_{C_tL^1},\\
 \norm{\partial_x(\partial_x^{-2}\partial_y^2w)}
      _{\A^{-S}_{\R^2,t}}
 &\leq C_S\norm w_{C_tL^1}
 \sum_{n\in\mu\Z\setminus\{0\}}
 \int_{\abs{(\xi-n,\eta)}\leq2\ell}
   \frac{\eta^2}{\abs\xi}\langle(\xi,\eta)\rangle^{-S}
   \,d\xi\,d\eta\\
 &\leq C_S\ell^4\norm w_{C_tL^1}
   \sum_{n\in\mu\Z\setminus\{0\}}\abs n^{-S-1}
 \leq C_S\ell^4\mu^{-S-1}\norm w_{C_tL^1}.
\end{align*}
Consequently,
\begin{equation}\label{eq:Euclidean-onestep-dispersion}
 \norm{\partial_x\bigl((-1)^{(d+1)/2}\partial_x^{d-1}w
       +\kappa\partial_x^{-2}\partial_y^2w)}_{\A^{-S}_{\R^2,t}}
 \leq C_S\left(\ell^2\mu^{d-S}
             +\ell^4\mu^{-S-1}\right)\norm w_{C_tL^1}.
\end{equation}
By \eqref{eq:Euclidean-Lp} with $p=1$, this is at most
\[
 C_{E,A,L,\ell,S,\eps}\left(
 \la^{\eps(d-S)-\frac12(1-\eps)}
 +\la^{-\eps(S+1)-\frac12(1-\eps)}\right),
\]
which tends to zero because $S>d+1$.

\subsubsection*{Nash error}

The Nash error is $\partial_x(2uw)$.  Minkowski's inequality and
Proposition~\ref{prop:Euclidean-mixed-absolute}, applied using
\eqref{eq:Euclidean-onestep-separation}, give
\begin{equation}\label{eq:Euclidean-onestep-Nash}
 \norm{\partial_x(2uw)}_{\A^{-S}_{\R^2,t}}
 \leq C_{u,E,S,M,\chi,\eps}A^{1/2}
 \la^{-\frac12(1-\eps)-\eps(S-1)}.
\end{equation}

\subsubsection*{Oscillation error}

The oscillation error is
\[
 \partial_x(E+w^2)
 =\partial_x\left[h^2\left((1-\chi_L^2)E+A\chi_L^2
 +(a_{L,\ell}^2-a_L^2)
 +a_{L,\ell}^2(\rho_{\la,\eps}^2-1)\right)\right].
\]
Put
$K_L=\widehat{\chi_L}*\widehat{\chi_L}$.  Since
$K_L\geq0$ and $\int_{\R^2}K_L=1$,
\[
 \widehat{(1-\chi_L^2)E}(t,k)
 =\int_{\R^2}K_L(z)
   \bigl(\widehat E(t,k)-\widehat E(t,k-z)\bigr)\,dz.
\]
Minkowski's inequality and the fundamental theorem of calculus give
\begin{align*}
 \int_{\R^2}
 \norm{\widehat{(1-\chi_L^2)E}(\cdot,k)}_{C_t}\,dk
 &\leq\int_{\R^2}\abs z
 (\widehat{\chi_L}*\widehat{\chi_L})(z)\,dz
 \int_0^1\int_{\R^2}
 \norm{\nabla\widehat E(\cdot,k-\sigma z)}_{C_t}\,dk\,d\sigma\\
 &\leq\frac{C_\chi}{L}
 \int_{\R^2}\norm{\nabla\widehat E(\cdot,k)}_{C_t}\,dk.
\end{align*}
Here the last step is \eqref{eq:Euclidean-spatial-cutoff-moment}.
Therefore,
\begin{equation*}
 \norm{\partial_x\bigl(h^2(1-\chi_L^2)E\bigr)}
      _{\A^{-S}_{\R^2,t}}
 \leq\frac{C_{E,\chi}}{L}.
\end{equation*}
The localized constant term satisfies
\begin{equation*}
 \norm{\partial_x(h^2A\chi_L^2)}_{\A^{-S}_{\R^2,t}}
 \leq\frac{C_\chi A}{L}
\end{equation*}
by \eqref{eq:localized-constant-scaling}.  Since $m_\ell=1$ on
$\{\abs k\leq\ell\}$, \eqref{eq:Euclidean-amplitude-Wiener-bounds}
gives
\[
 \norm{a_L-a_{L,\ell}}_{\A^0_{\R^2,t}}
 \leq\int_{\abs k>\ell}\norm{\widehat a_L(\cdot,k)}_{C_t}\,dk
 \leq\ell^{-M}\norm{a_L}_{\A^M_{\R^2,t}}
 \leq C_{E,M,\chi}A^{1/2}\ell^{-M}.
\]
The Wiener algebra inequality therefore yields
\begin{equation*}
 \norm{\partial_x\bigl(h^2(a_{L,\ell}^2-a_L^2)\bigr)}
      _{\A^{-S}_{\R^2,t}}
 \leq C_S\norm{a_{L,\ell}-a_L}_{\A^0_{\R^2,t}}
 \left(\norm{a_{L,\ell}}_{\A^0_{\R^2,t}}
       +\norm{a_L}_{\A^0_{\R^2,t}}\right)
 \leq C_{E,M,\chi}A\ell^{-M}.
\end{equation*}

For the remaining term, write
\[
 \rho_{\la,\eps}^2-1
 =\sum_{s\in\mu\Z\setminus\{0\}}
 \widehat{\rho_{\la,\eps}^2}(s)e^{2\pi isx}.
\]
The normalization gives $\widehat{\rho_{\la,\eps}^2}(0)=1$.  Since
\[
 \abs{\widehat{\rho_{\la,\eps}^2}(s)}
 \leq\sum_n\abs{c_nc_{s-n}},
\]
\eqref{eq:Euclidean-frequency-sum-bound} yields
\begin{equation}\label{eq:Euclidean-onestep-profile-convolution}
 \sum_{s\in\mu\Z\setminus\{0\}}
 \abs{\widehat{\rho_{\la,\eps}^2}(s)}\abs s^{1-S}
 \leq C_S\mu^{1-S}.
\end{equation}
Since $\supp\widehat{a_{L,\ell}^2}\subset\{\abs k\leq4\ell\}$ and
$\mu\geq8\ell$, for $s\neq0$ and $\abs k\leq4\ell$ one has
\[
 2\pi\abs{s+k_x}\left\langle(s,0)+k\right\rangle^{-S}
 \leq C_S\abs s^{1-S}.
\]
Therefore, \eqref{eq:Euclidean-onestep-profile-convolution} and
\eqref{eq:Euclidean-amplitude-Wiener-bounds} give
\begin{align*}
 \norm{\partial_x\left[
 h^2a_{L,\ell}^2(\rho_{\la,\eps}^2-1)\right]}
 _{\A^{-S}_{\R^2,t}}
 &\leq C_S\sum_{s\in\mu\Z\setminus\{0\}}
 \abs{\widehat{\rho_{\la,\eps}^2}(s)}\abs s^{1-S}
 \norm{a_{L,\ell}^2}_{\A^0_{\R^2,t}}\\
 &\leq C_{E,S,M,\chi}A\mu^{1-S}.
\end{align*}
These estimates and \eqref{eq:Euclidean-cancellation} yield
\begin{equation}\label{eq:Euclidean-onestep-oscillation}
 \norm{\partial_x(E+w^2)}_{\A^{-S}_{\R^2,t}}
 \leq\frac{C_{E,\chi}}{L}+\frac{C_\chi A}{L}
 +C_{E,M,\chi}A\ell^{-M}+C_{E,S,M,\chi}A\mu^{1-S}.
\end{equation}

\subsubsection*{Temporal error}

The temporal error in $\partial_xE^+$ is $\partial_tw$, where
\[
 \partial_tw=h'a_{L,\ell}\rho_{\la,\eps}
 +h(\partial_ta_{L,\ell})\rho_{\la,\eps}.
\]
Since $A-E\geq A/2$,
\[
 \partial_ta_L=-\frac{\chi_L\partial_tE}{2(A-E)^{1/2}},
 \qquad
 \partial_ta_{L,\ell}=P_{\leq\ell}^{x,y}\partial_ta_L,
\]
and hence
\[
 \norm{a_{L,\ell}}_{C_tW^{3,1}}
 +\norm{\partial_ta_{L,\ell}}_{C_tW^{3,1}}
 \leq C_{E,A,L,\ell}.
\]
Applying \eqref{eq:Euclidean-periodic-decoupling} with $p=1$ gives
\begin{align*}
 \norm{\partial_tw}_{C_tL^1}
 &\leq C\left(\norm{h'}_{L^\infty}
       \norm{a_{L,\ell}}_{C_tW^{3,1}}
       +\norm{\partial_ta_{L,\ell}}_{C_tW^{3,1}}\right)
       \norm{\rho_{\la,\eps}}_{L^1(\T^2)}\\
 &\leq C_{E,A,L,\ell,\eps}
       (1+\la^\beta)\la^{-\frac12(1-\eps)}.
\end{align*}
The gap in \eqref{eq:Euclidean-cone} defines $\partial_x^{-1}w$ and gives
$\partial_x\partial_t\partial_x^{-1}w=\partial_tw$.  Since $S>2$,
\begin{equation}\label{eq:Euclidean-onestep-temporal}
 \norm{\partial_x(\partial_t\partial_x^{-1}w)}
      _{\A^{-S}_{\R^2,t}}
 \leq\left(\int_{\R^2}\langle k\rangle^{-S}\,dk\right)
       \norm{\partial_tw}_{C_tL^1}
 \leq C_{E,A,L,\ell,S,\eps}
       (1+\la^\beta)\la^{-\frac12(1-\eps)}
 \longrightarrow0.
\end{equation}

By \eqref{eq:Euclidean-update-general},
\begin{align*}
 \norm{\partial_xE^+}_{\A^{-S}_{\R^2,t}}
 \leq{}&\norm{\partial_x(E+w^2)}_{\A^{-S}_{\R^2,t}}
 +\norm{\partial_x(2uw)}_{\A^{-S}_{\R^2,t}}\\
 &+\norm{\partial_x\bigl((-1)^{(d+1)/2}\partial_x^{d-1}w+
       \kappa\partial_x^{-2}\partial_y^2w)}_{\A^{-S}_{\R^2,t}}
 +\norm{\partial_x(\partial_t\partial_x^{-1}w)}
       _{\A^{-S}_{\R^2,t}}.
\end{align*}
For the absolute Fourier bound, Proposition~\ref{prop:Euclidean-mixed-absolute}
gives
\begin{align*}
 \operatorname{AF}^{\partial_x}_S(u^+)
 \leq{}&\operatorname{AF}^{\partial_x}_S(u)
 +\norm{\partial_xE}_{\A^{-S}_{\R^2,t}}\\
 &+C_{E,S,M,\chi}\left(
   \frac{A}{L}+\frac1L+\frac1A+A\mu^{1-S}\right)\\
 &+C_{u,E,S,M,\chi,\eps}A^{1/2}
 \la^{-\frac12(1-\eps)-\eps(S-1)}.
\end{align*}
The choices of $A,L,$ and $\ell$ control all $\la$-independent remainder
terms.  It remains to impose the separation and time-support conditions,
the finitely many bounds \eqref{eq:Euclidean-onestep-Lp}, the
$\la$-dependent terms in
\eqref{eq:Euclidean-onestep-dispersion}--
\eqref{eq:Euclidean-onestep-temporal}, and the two terms containing
$\mu$ or $\la$ in the last display.  The support conditions hold for all
sufficiently large admissible $\la$, and every listed norm or error term
tends to zero.  One sufficiently large choice of $\la$ therefore proves
\eqref{eq:Euclidean-onestep-Lp}--
\eqref{eq:Euclidean-onestep-absolute}.

\end{proof}

\section{Cubic estimates for the modified KP equations}

We construct one profile with frequency scale $\la$ in both variables and
another with $x$-frequency scale $\la$ and $y$-frequency scale $\la^2$.
The stronger $L^2$ decay of the latter yields the $x$-Sobolev exponent
$1/2$.

\begin{lemma}[Two-dimensional cubic profiles]\label{lem:cubic-profile}
Fix $0<\eps<1$.  For every sufficiently large admissible $\la$, set
\begin{equation*}
 \mu=\la^\eps,\qquad
 J=\left\lfloor\frac{\la}{\mu}\right\rfloor.
\end{equation*}
There are real trigonometric polynomials
$\rho_{3,\la,\eps}$ and
$\widetilde{\rho}_{3,\la,\eps}$ such that
\begin{align}
 &\int_{\T^2}(\rho_{3,\la,\eps})^3\,dx\,dy
 =\int_{\T^2}(\widetilde{\rho}_{3,\la,\eps})^3\,dx\,dy=1,
 \label{eq:cubic-profile-moment}\\
 &\widehat\rho_{3,\la,\eps}(n,m)\geq0,
 \qquad
 \widehat{\widetilde{\rho}}_{3,\la,\eps}(n,m)\geq0,\notag\\
 &\supp\widehat\rho_{3,\la,\eps}
 \cup\supp\widehat{\widetilde{\rho}}_{3,\la,\eps}
 \subset(\mu\Z)^2,
 \label{eq:cubic-profile-positivity}\\
 &7\mu J\leq\abs n\leq17\mu J
 \quad\text{on }
 \supp\widehat\rho_{3,\la,\eps}
 \cup\supp\widehat{\widetilde{\rho}}_{3,\la,\eps}.
 \label{eq:cubic-profile-support}
\end{align}
For large $\la$,
\begin{equation}\label{eq:cubic-profile-support-bounds}
 \begin{aligned}
 \supp\widehat\rho_{3,\la,\eps}
 &\subset\{(n,m):3\la\leq\abs n\leq18\la,\ \abs m\leq2\la\},\\
 \supp\widehat{\widetilde{\rho}}_{3,\la,\eps}
 &\subset\{(n,m):3\la\leq\abs n\leq18\la,\ \abs m\leq2\la^2\}.
 \end{aligned}
\end{equation}
For $1\leq p\leq\infty$,
\begin{align}
 \norm{\rho_{3,\la,\eps}}_{L^p(\T^2)}
 &\leq C_{p,\eps}\la^{2(1-\eps)(1/3-1/p)},\notag\\
 \norm{\widetilde{\rho}_{3,\la,\eps}}_{L^p(\T^2)}
 &\leq C_{p,\eps}\la^{3(1-\eps)(1/3-1/p)}.
 \label{eq:cubic-profile-Lp}
\end{align}
Their Fourier coefficients obey
\begin{align}
 \widehat\rho_{3,\la,\eps}(k)
 &\leq C_\eps\la^{-4(1-\eps)/3},\notag\\
 \widehat{\widetilde{\rho}}_{3,\la,\eps}(k)
 &\leq C_\eps\la^{-2(1-\eps)},
 \label{eq:cubic-profile-coefficient-bound}\\
 \norm{\rho_{3,\la,\eps}}_{L^2(\T^2)}^2
 &\leq C_\eps\la^{-2(1-\eps)/3},\notag\\
 \norm{\widetilde{\rho}_{3,\la,\eps}}_{L^2(\T^2)}^2
 &\leq C_\eps\la^{-(1-\eps)},\notag\\
 \norm{\rho_{3,\la,\eps}}_{L^3(\T^2)}^3
 +\norm{\widetilde{\rho}_{3,\la,\eps}}_{L^3(\T^2)}^3
 &\leq C_\eps.
 \label{eq:cubic-profile-endpoints}
\end{align}
Moreover,
\begin{align}
 \sum_{k\in\Z^2}\widehat\rho_{3,\la,\eps}(k)
 &\leq C_\eps\la^{2(1-\eps)/3},\notag\\
 \sum_{k\in\Z^2}\widehat{\widetilde{\rho}}_{3,\la,\eps}(k)
 &\leq C_\eps\la^{1-\eps}.
 \label{eq:cubic-profile-Fourier-mass}
\end{align}
\begin{equation}
\begin{aligned}
 0\leq\widehat{(\rho_{3,\la,\eps})^3}(s)
 &\leq C_\eps,\\
 0\leq\widehat{(\widetilde{\rho}_{3,\la,\eps})^3}(s)
 &\leq C_\eps
 \qquad(s\in\Z^2).
\end{aligned}
\label{eq:cubic-profile-cubic-coefficient-bound}
\end{equation}
\end{lemma}

\begin{proof}
Define the Fej\'er kernel and trigonometric polynomial by
\[
 F_J(z)=\sum_{\abs a<J}\left(1-\frac{\abs a}{J}\right)e^{2\pi iaz},
 \qquad
 Q_J(z)=\sum_{q\in\{\pm8J,\pm16J\}}e^{2\pi iqz}.
\]
The unnormalized profiles are
\begin{equation*}
 \begin{aligned}
 R(x,y)
 &=Q_J(\mu x)F_J(\mu x)F_J(\mu y),\\
 \widetilde R(x,y)
 &=Q_J(\mu x)F_J(\mu x)F_{J^2}(\mu y),
 \end{aligned}
\end{equation*}
Normalize them by
\[
 \rho_{3,\la,\eps}
 =\frac{R}
 {\left(\int_{\T^2}R^3\,dx\,dy\right)^{1/3}},
 \qquad
 \widetilde{\rho}_{3,\la,\eps}
 =\frac{\widetilde R}
 {\left(\int_{\T^2}(\widetilde R)^3\,dx\,dy\right)^{1/3}}.
\]
The Fourier coefficients are
\begin{equation}\label{eq:cubic-profile-raw-Fourier}
 \begin{aligned}
 \widehat R(\mu a,\mu b)
 &=\sum_{q\in\{\pm8J,\pm16J\}}
 \left(1-\frac{\abs{a-q}}J\right)_+
 \left(1-\frac{\abs b}J\right)_+,\\
 \widehat{\widetilde R}(\mu a,\mu b)
 &=\sum_{q\in\{\pm8J,\pm16J\}}
 \left(1-\frac{\abs{a-q}}J\right)_+
 \left(1-\frac{\abs b}{J^2}\right)_+.
 \end{aligned}
\end{equation}
Both transforms vanish off $(\mu\Z)^2$.  The four translates in
\eqref{eq:cubic-profile-raw-Fourier} are disjoint, so their coefficients
lie in $[0,1]$, and
\[
 7J\leq\abs a\leq17J
\]
on either support; the transverse bounds are $\abs b\leq J$ and
$\abs b\leq J^2$, respectively.

The Fej\'er estimates
\begin{equation*}
 \norm{F_J}_{L^p(\T)}\leq C_pJ^{1-1/p},
 \qquad1\leq p\leq\infty,
\end{equation*}
give
\begin{align}
 \norm{R}_{L^p(\T^2)}
 &\leq C_pJ^{2(1-1/p)},\notag\\
 \norm{\widetilde R}_{L^p(\T^2)}
 &\leq C_pJ^{3(1-1/p)}.
 \label{eq:cubic-profile-raw-Lp}
\end{align}
All Fourier coefficients are nonnegative.  Using the zero-sum carrier triple
$(8J,8J,-16J)$,
\[
 \int_{\T^2}R^3\,dx\,dy\geq cJ^4,
 \qquad
 \int_{\T^2}(\widetilde R)^3\,dx\,dy\geq cJ^6.
\]
For this carrier triple, choose the first two $x$-frequency deviations with
magnitude at most $J/8$ and set the third equal to minus their sum.  This
gives at least $cJ^2$ choices in the $x$-variable.  In the $y$-variable,
choose the first two frequencies with magnitude at most $J/8$ for $R$ and
at most $J^2/8$ for $\widetilde R$, and set the third equal to minus their
sum.  Every Fej\'er coefficient selected in this way is at least $3/4$.
This gives at least $cJ^2$ and $cJ^4$ choices, respectively.  The
$p=3$ case of
\eqref{eq:cubic-profile-raw-Lp} gives the reverse bounds, and hence
\begin{equation}\label{eq:cubic-profile-normalizer}
 \begin{aligned}
 cJ^{4/3}
 &\leq\left(\int_{\T^2}R^3\,dx\,dy\right)^{1/3}
 \leq CJ^{4/3},\\
 cJ^2
 &\leq\left(\int_{\T^2}(\widetilde R)^3\,dx\,dy\right)^{1/3}
 \leq CJ^2.
 \end{aligned}
\end{equation}
Equations \eqref{eq:cubic-profile-raw-Fourier},
\eqref{eq:cubic-profile-raw-Lp}, and
\eqref{eq:cubic-profile-normalizer} give
\eqref{eq:cubic-profile-moment}--\eqref{eq:cubic-profile-endpoints}, since
$\frac12\la^{1-\eps}\leq J\leq\la^{1-\eps}$ for large $\la$.
Indeed, normalization bounds the coefficients by $CJ^{-4/3}$ and $CJ^{-2}$,
while the $p=2$ case of \eqref{eq:cubic-profile-raw-Lp} gives $L^2$ norms
bounded by $CJ^{-1/3}$ and $CJ^{-1/2}$, respectively; the $p=3$ bounds
become uniform.
Nonnegativity of the Fourier coefficients also gives
\[
 \sum_k\widehat\rho_{3,\la,\eps}(k)
 =\rho_{3,\la,\eps}(0,0)\lesssim J^{2/3},
 \qquad
 \sum_k\widehat{\widetilde{\rho}}_{3,\la,\eps}(k)
 =\widetilde{\rho}_{3,\la,\eps}(0,0)\lesssim J,
\]
which proves \eqref{eq:cubic-profile-Fourier-mass}.  The Fourier
coefficients of each cube are nonnegative because those of the profile are
nonnegative.  Moreover,
\[
 \abs{\widehat{(\rho_{3,\la,\eps})^3}(s)}
 \leq\norm{\rho_{3,\la,\eps}}_{L^3}^3,
 \qquad
 \abs{\widehat{(\widetilde{\rho}_{3,\la,\eps})^3}(s)}
 \leq\norm{\widetilde{\rho}_{3,\la,\eps}}_{L^3}^3.
\]
Thus \eqref{eq:cubic-profile-endpoints} gives
\eqref{eq:cubic-profile-cubic-coefficient-bound}.
\end{proof}

\subsection{Periodic cubic estimates}

The cubic expansion contains both $u^2w$ and $uw^2$.  Frequency separation
controls $u^2w$; in $uw^2$, the two perturbation frequencies can cancel,
so the decay comes from the small $L^2$ mass of the profile.  At zero total
profile frequency, the terms in $w^3$ that are linear in $E$ must
contribute $-E$.

\begin{proposition}
\label{prop:cubic-periodic-estimate}
Fix $0<\eps<1$ and $S>2$.  Let $\rho$ be either
$\rho_{3,\la,\eps}$ or
$\widetilde{\rho}_{3,\la,\eps}$ from
Lemma~\ref{lem:cubic-profile}.  Let $u$ and $E$ be real functions in
$C_tC^\infty(\T^2)$ with finite spatial Fourier support and zero
$x$-mean.  Choose
$A\geq1+2\norm E_{C_tL^\infty}$ so large that
$\norm E_{\A^0_t}/A\leq1/2$.  Let $0\leq h\leq1$ depend only on time
and satisfy $hE=E$.  For a sufficiently large admissible $\la$,
put
\begin{equation*}
 b=P_{\leq\mu^{1/2}}^{x,y}(A-E)^{1/3},
 \qquad B=hb,
 \qquad w=B\rho.
\end{equation*}
Assume that
\[
 \mu^{1/2}>16\max\bigl(\{1\}\cup
 \{\abs k:k\in\supp\widehat u\cup\supp\widehat E\}\bigr),
 \qquad \mu\geq256.
\]
Then
\begin{align*}
 \operatorname{AF}_{S,3}^x(w)
 &\leq\norm E_{\A^{-S}_{x\neq0,t}}
   +C_{E,S}\left(A^{-1}+A\mu^{-S}\right),\\
 \operatorname{AF}_{S,3}^x(u,u,w)
 &\leq C_{u,E,S,\eps}A^{1/3}
   \la^{-4(1-\eps)/3}\mu^{-S},\\
 \operatorname{AF}_{S,3}^x(u,w,w)
 &\leq C_{u,E,S,\eps}A^{2/3}
   \la^{-2(1-\eps)/3}.
\end{align*}
\end{proposition}

\begin{proof}
The cube-root series in $\A^0_t$ gives
\begin{equation}\label{eq:cubic-periodic-amplitude-expansion}
 \begin{split}
 B={}&B_0+B_1+B_{\geq2},\\
 B_0={}&A^{1/3}h,\\
 B_1={}&-\frac{1}{3A^{2/3}}E,\\
 B_{\geq2}={}&A^{1/3}\sum_{j\geq2}
  \binom{1/3}{j}(-1)^jA^{-j}
  P_{\leq\mu^{1/2}}^{x,y}(E^j).
 \end{split}
\end{equation}
Since $\supp\widehat E\subset\{\langle k\rangle<\mu^{1/2}\}$,
$P_{\leq\mu^{1/2}}^{x,y}E=E$; also $hE^j=E^j$ for $j\geq1$.
The Wiener algebra inequality gives
\begin{equation}\label{eq:cubic-periodic-amplitude-remainder}
 \norm{B_{\geq2}}_{\A^0_t}
 \leq C_EA^{-5/3},
 \qquad
 \norm B_{\A^0_t}\leq C_EA^{1/3}.
\end{equation}

Expand $\operatorname{AF}_{S,3}^x(B)$ using
\eqref{eq:cubic-periodic-amplitude-expansion}.  The three terms linear in
$E$ have total absolute coefficient one and contribute
\[
 \sum_{k_x\neq0}\langle k\rangle^{-S}
  \norm{h^2\widehat E(\cdot,k)}_{C_t}
=\norm E_{\A^{-S}_{x\neq0,t}}.
\]
Every remaining term in $\operatorname{AF}_{S,3}^x(B)$ contains
either two copies of $B_1$ or one copy of $B_{\geq2}$.  Using
$\langle k\rangle^{-S}\leq1$ and
\eqref{eq:cubic-periodic-amplitude-remainder} gives
\begin{equation}\label{eq:cubic-periodic-amplitude-AF}
 \operatorname{AF}_{S,3}^x(B)
 \leq\norm E_{\A^{-S}_{x\neq0,t}}+C_EA^{-1}.
\end{equation}

Cauchy--Schwarz and Lemma~\ref{lem:cubic-profile} give
\begin{equation*}
 \widehat{\rho^3}(0)=1,
 \qquad
 0\leq\widehat{\rho^2}(s)\leq\norm{\rho}_{L^2(\T^2)}^2
 \leq C_\eps\la^{-2(1-\eps)/3}
 \quad(s\in\Z^2).
\end{equation*}

The Fourier translates of $B$ by the frequencies in
$\supp\widehat\rho$ are disjoint.  In the exact expansion of
$\operatorname{AF}_{S,3}^x(w)$, the terms with
$q_1+q_2+q_3=0$ therefore contribute
$\operatorname{AF}_{S,3}^x(B)$, multiplied by
$\widehat{\rho^3}(0)=1$.  If
$s=q_1+q_2+q_3$ is nonzero, the sum of the three frequencies from $B$
has magnitude at most $6\mu^{1/2}$, and hence at most $\abs s/2$ since
$\mu\geq256$.
It follows from \eqref{eq:cubic-profile-cubic-coefficient-bound} that the
remaining terms are at most
\begin{align*}
 C_S\norm B_{\A^0_t}^3
 \sum_{s\in(\mu\Z)^2\setminus\{0\}}\abs s^{-S}
 &\leq C_{E,S}A\mu^{-S}
 \sum_{j\in\Z^2\setminus\{0\}}\abs j^{-S}\\
 &\leq C_{E,S}A\mu^{-S}.
\end{align*}
For $q\in\supp\widehat\rho$, the sum of $q$,
two frequencies from $\supp\widehat u$, and one frequency from $B$ has
magnitude at least $\abs q/2$.  Hence
\begin{align*}
 \operatorname{AF}_{S,3}^x(u,u,w)
 &\leq C_S\norm u_{\A^0_t}^2\norm B_{\A^0_t}
   \sum_{q\neq0}\widehat\rho(q)\abs q^{-S}
 \\
 &\leq C_{u,E,S,\eps}A^{1/3}
   \la^{-4(1-\eps)/3}\mu^{-S},
\end{align*}
where we used \eqref{eq:cubic-profile-coefficient-bound} and extended the finite
sum over $q\in\supp\widehat\rho$ to
$(\mu\Z)^2\setminus\{0\}$.

To estimate $\operatorname{AF}_{S,3}^x(u,w,w)$, split according to
$s=q_1+q_2$.  For every $s$,
$\widehat{\rho^2}(s)\leq\norm{\rho}_{L^2}^2$.  If $s=0$, then
$\langle k\rangle^{-S}\leq1$.  If $s\neq0$, the support assumption and
$\mu\geq256$ imply that
the sum of the frequency from $u$ and the two frequencies from $B$ has
magnitude at most $\abs s/2$; hence
$\langle k\rangle^{-S}\lesssim_S\abs s^{-S}$.  Therefore,
\begin{align*}
 \operatorname{AF}_{S,3}^x(u,w,w)
 &\leq C_S\norm u_{\A^0_t}\norm B_{\A^0_t}^2
  \la^{-2(1-\eps)/3}
  \left(1+\sum_{s\in(\mu\Z)^2\setminus\{0\}}\abs s^{-S}\right)
 \\
 &\leq C_{u,E,S,\eps}A^{2/3}
  \la^{-2(1-\eps)/3}.
\end{align*}
\end{proof}

\subsection{Whole-space cubic estimates}

The convolution $\widehat{\chi_L}*\widehat{\chi_L}*\widehat{\chi_L}$ is
nonnegative, has integral one, and
\begin{equation}\label{eq:cubic-spatial-cutoff-moment}
 \int_{\R^2}\abs z
 (\widehat{\chi_L}*\widehat{\chi_L}*\widehat{\chi_L})(z)\,dz
 \leq\frac{C_\chi}{L}.
\end{equation}
The proof of Lemma~\ref{lem:Euclidean-positive-approximation} also gives,
for every nonnegative $F\in L^1(\R^2)$,
\begin{equation}\label{eq:cubic-positive-approximation}
 \int_{\R^2}2\pi\abs\xi\langle k\rangle^{-S}
 \bigl[(\widehat{\chi_L}*\widehat{\chi_L}*
 \widehat{\chi_L})*F\bigr](k)\,dk
 \leq\int_{\R^2}2\pi\abs\xi\langle k\rangle^{-S}F(k)\,dk
   +\frac{C_{\chi,S}}L\norm F_{L^1}.
\end{equation}

\begin{proposition}
\label{prop:cubic-Euclidean-estimate}
Fix $0<\eps<1$, $S>3$, and $M>S+6$.  Let $\rho$ be either
$\rho_{3,\la,\eps}$ or
$\widetilde{\rho}_{3,\la,\eps}$ from
Lemma~\ref{lem:cubic-profile}.  Let
$E\in C_t\mathcal S(\R^2)$ be real, and choose $A$ so large that
\begin{equation}\label{eq:cubic-Euclidean-A}
 A\geq2\norm E_{C_tL^\infty},
 \qquad
 \frac{C_{\mathrm{alg},M}\norm E_{\A^M_{\R^2,t}}}{A}\leq\frac12.
\end{equation}
For $L,\ell\geq1$, define
\begin{equation}\label{eq:cubic-Euclidean-amplitude}
 a_L=\chi_L(A-E)^{1/3},
 \qquad a_{L,\ell}=P_{\leq\ell}^{x,y}a_L.
\end{equation}
Let $0\leq h\leq1$ be a smooth function of time satisfying $hE=E$,
and set
\[
 w=h a_{L,\ell}\rho.
\]
Suppose
\begin{equation*}
 \mu\geq16\ell.
\end{equation*}
Then
\begin{equation*}
 \operatorname{AF}^{\partial_x}_{S,3}(w)
 \leq\norm{\partial_xE}_{\A^{-S}_{\R^2,t}}
  +C_{E,S,M,\chi}\left(
    \frac AL+\frac1L+\frac1A+A\mu^{1-S}\right).
\end{equation*}
Suppose also that $u\in C_tL^1(\R^2)$,
$\norm u_{\A^0_{\R^2,t}}<\infty$,
$\widehat u(t,k)=0$ for $\abs k>R$, where $R\geq1$, and
$\mu>8R$.  Then
\begin{align*}
 \operatorname{AF}^{\partial_x}_{S,3}(u,u,w)
 &\leq C_{u,E,S,M,\chi,\eps}A^{1/3}
   \la^{-4(1-\eps)/3}\mu^{1-S},\\
 \operatorname{AF}^{\partial_x}_{S,3}(u,w,w)
 &\leq C_{u,E,S,M,\chi,\eps}A^{2/3}
   \la^{-2(1-\eps)/3}.
\end{align*}
\end{proposition}

\begin{proof}
The binomial series converges in $\A^M_{\R^2,t}$, uniformly for $L\geq1$.
Since $hE^j=E^j$ for every $j\geq1$, it yields
\begin{equation}\label{eq:cubic-Euclidean-amplitude-series}
 h a_L=A^{1/3}h\chi_L-\frac1{3A^{2/3}}\chi_LE
 +\sum_{j\geq2}\binom{1/3}{j}(-1)^j
       A^{1/3-j}\chi_LE^j.
\end{equation}
The multiplier of $P_{\leq\ell}^{x,y}$ has absolute value at most one, so
\[
 \operatorname{AF}^{\partial_x}_{S,3}(h a_{L,\ell})
 \leq \operatorname{AF}^{\partial_x}_{S,3}(h a_L).
\]

The term with no factor of $E$ in the expansion of
$\operatorname{AF}^{\partial_x}_{S,3}(h a_L)$ is
\[
 A\int_{\R^2}2\pi\abs\xi\langle k\rangle^{-S}
 (\widehat{\chi_L}*\widehat{\chi_L}*\widehat{\chi_L})(k)\,dk
 \leq\frac{C_\chi A}{L}
\]
by \eqref{eq:cubic-spatial-cutoff-moment}.  The three terms containing
exactly one copy of $E$ have total absolute coefficient one.
$\widehat\chi_L\geq0$, Minkowski's inequality, and Tonelli's
theorem bound their sum by the convolution of
$\norm{\widehat E(\cdot,k)}_{C_t}$ with
$\widehat{\chi_L}*\widehat{\chi_L}*\widehat{\chi_L}$.  Thus
\eqref{eq:cubic-positive-approximation} bounds these three terms by
\[
 \norm{\partial_xE}_{\A^{-S}_{\R^2,t}}
 +\frac{C_{\chi,S}}L\norm E_{\A^0_{\R^2,t}}.
\]
Every remaining term has degree at least two in $E$.  Since
$2\pi\abs\xi\langle k\rangle^{-S}$ is bounded for $S\geq1$, their
sum is at most
\[
 C_SA\sum_{i+j+k\geq2}
 \abs{\binom{1/3}{i}\binom{1/3}{j}\binom{1/3}{k}}
 \left(\frac{\norm E_{\A^0_{\R^2,t}}}{A}\right)^{i+j+k}
 \leq\frac{C_{E,S}}A.
\]
Combining these estimates gives
\begin{equation}\label{eq:cubic-Euclidean-amplitude-AF}
 \operatorname{AF}^{\partial_x}_{S,3}(h a_{L,\ell})
 \leq\norm{\partial_xE}_{\A^{-S}_{\R^2,t}}
 +C_{E,S,M,\chi}\left(\frac AL+\frac1L+\frac1A\right).
\end{equation}

The shifted copies of $\widehat{h a_{L,\ell}}$ are disjoint, so the
expansion over $q_1,q_2,q_3\in\supp\widehat\rho$ is exact.
The terms with $q_1+q_2+q_3=0$ equal
\eqref{eq:cubic-Euclidean-amplitude-AF} with coefficient
\[
 \widehat{\rho^3}(0)=1.
\]
For $s\in(\mu\Z)^2\setminus\{0\}$, the coefficient
$\widehat{\rho^3}(s)$ is bounded by
\eqref{eq:cubic-profile-cubic-coefficient-bound}.  If
$\zeta_i\in\supp\widehat{h a_{L,\ell}}$, then
$\abs{\zeta_1+\zeta_2+\zeta_3}\leq6\ell\leq3\abs s/8$.
Hence the terms with $q_1+q_2+q_3\neq0$ are at most
\begin{align*}
 C_{E,S}A\sum_{s\in(\mu\Z)^2\setminus\{0\}}\abs s^{1-S}
 &\leq C_{E,S}A\mu^{1-S}
 \sum_{j\in\Z^2\setminus\{0\}}\abs j^{1-S}\\
 &\leq C_{E,S}A\mu^{1-S}.
\end{align*}

For $\operatorname{AF}^{\partial_x}_{S,3}(u,u,w)$, write a frequency in
$\supp\widehat w$ as $q+\zeta$, where
$q\in\supp\widehat\rho$ and
$\abs\zeta\leq2\ell$.  The sum of $\zeta$ and two frequencies from
$\supp\widehat u$ has magnitude at most $3\abs q/8$, so the factor
$2\pi\abs\xi\langle k\rangle^{-S}$ at the output frequency is at most
$C_S\abs q^{1-S}$.  Consequently,
\begin{align*}
 \operatorname{AF}^{\partial_x}_{S,3}(u,u,w)
 &\leq C_S\norm u_{\A^0_{\R^2,t}}^2
       \norm{h a_{L,\ell}}_{\A^0_{\R^2,t}}
       \sum_{q\neq0}\widehat\rho(q)\abs q^{1-S}\\
 &\leq C_{u,E,S,M,\chi,\eps}A^{1/3}
       \la^{-4(1-\eps)/3}\mu^{1-S},
\end{align*}
where the last line follows from
\eqref{eq:cubic-profile-coefficient-bound} and $S>3$.

For $\operatorname{AF}^{\partial_x}_{S,3}(u,w,w)$, group the terms by
$s=q_1+q_2$.  Cauchy--Schwarz and
\eqref{eq:cubic-profile-endpoints} give, uniformly in $s$,
\begin{equation*}
 0\leq\widehat{\rho^2}(s)
 \leq\norm{\rho}_{L^2(\T^2)}^2
 \leq C_\eps\la^{-2(1-\eps)/3}.
\end{equation*}
For $s=0$, $2\pi\abs\xi\langle k\rangle^{-S}\leq C_S$.  If $s\neq0$,
the frequency from $\supp\widehat u$ and the two frequencies in
$\supp\widehat{h a_{L,\ell}}$ have total magnitude at most
$3\abs s/8$.  Thus
$2\pi\abs\xi\langle k\rangle^{-S}\leq C_S\abs s^{1-S}$.
Summing over $(\mu\Z)^2\setminus\{0\}$ gives
\begin{align*}
 \operatorname{AF}^{\partial_x}_{S,3}(u,w,w)
 &\leq C_S\norm u_{\A^0_{\R^2,t}}
       \norm{h a_{L,\ell}}_{\A^0_{\R^2,t}}^2
       \la^{-2(1-\eps)/3}(1+C_S\mu^{1-S})\\
 &\leq C_{u,E,S,M,\chi,\eps}A^{2/3}
       \la^{-2(1-\eps)/3}.
\end{align*}
\end{proof}

\subsection{Cubic iteration steps}

\begin{proposition}
\label{prop:cubic-periodic-onestep}
Fix $d\in\{3,5\}$, $\kappa\in\{-1,1\}$,
$0<\eps<1$, and
\[
 0<\beta<\frac43(1-\eps).
\]
Let $\rho$ be either $\rho_{3,\la,\eps}$ or
$\widetilde{\rho}_{3,\la,\eps}$ from
Lemma~\ref{lem:cubic-profile}.
Let $S>d+1$.  Suppose that $u,E\in C^\infty([0,1]\times\T^2)$
are real, have finite Fourier support and zero $x$-mean, and solve
\eqref{eq:relaxed-general} with $r=3$.
Let $I\Subset(0,1)$ be a closed interval containing their time supports.
Given $\eta,\delta_+>0$, finitely many $1\leq p_j<3$, and finitely
many $\alpha_j\geq0$ satisfying
\begin{equation}\label{eq:cubic-periodic-alpha-condition}
 \begin{aligned}
 \alpha_j&<\frac13(1-\eps)
 &&\text{for the isotropic profile},\\
 \alpha_j&<\frac12(1-\eps)
 &&\text{for the parabolic profile},
 \end{aligned}
\end{equation}
there is $A$ such that every sufficiently large admissible $\la$ admits
a cutoff $h$ equal to one on $I$.  Set
\[
 b=P_{\leq\mu^{1/2}}^{x,y}(A-E)^{1/3},
 \qquad w=hb\rho,
 \qquad u^+=u+w.
\]
Define $E^+$ by
\eqref{eq:periodic-update-general} with $r=3$.  Then
\begin{equation}\label{eq:cubic-periodic-onestep-regularity}
 \begin{aligned}
 \norm w_{C_tL^{p_j}}+\norm w_{C_tH^{\alpha_j}}&<\eta
 &&\text{for the isotropic profile},\\
 \norm w_{C_tL^{p_j}}+\norm w_{C_tH^{\alpha_j,0}}&<\eta
 &&\text{for the parabolic profile}.
 \end{aligned}
\end{equation}
\begin{align}
 \norm{E^+}_{\A^{-S}_{x\neq0,t}}&<\delta_+,
 \label{eq:cubic-periodic-onestep-error}\\
 \operatorname{AF}_{S,3}^x(u^+)
 &\leq\operatorname{AF}_{S,3}^x(u)
   +\norm E_{\A^{-S}_{x\neq0,t}}+\eta.
 \label{eq:cubic-periodic-onestep-AF}
\end{align}
The functions $u^+$ and $E^+$ are smooth and real, have finite spatial
Fourier support and zero $x$-mean, and satisfy
\eqref{eq:relaxed-general}.
$\supp\widehat w$ is disjoint from $\supp\widehat u$ and is contained in
\begin{equation}\label{eq:cubic-periodic-support}
 \begin{aligned}
 3\la\leq\abs n\leq18\la,\quad \abs m\leq2\la
 &\qquad\text{for the isotropic profile},\\
 3\la\leq\abs n\leq18\la,\quad \abs m\leq2\la^2
 &\qquad\text{for the parabolic profile}.
 \end{aligned}
\end{equation}
The time supports of
$u^+$ and $E^+$ lie in the open $\la^{-\beta}$-neighborhood of $I$.
\end{proposition}

\begin{proof}
Fix an integer $N\geq1$.  Choose
$A\geq1+2\norm E_{C_tL^\infty}$ so that the binomial series converges
absolutely in $\A_t^N$ and $C_{E,S}A^{-1}<\eta/4$.  For a sufficiently
large admissible $\la$, choose $h$ satisfying
\eqref{eq:time-cutoff-lemma} with $\tau=\la^{-\beta}$.

\subsubsection*{Oscillation error}

The oscillation error is $\mathbb P_{x\neq0}(E+w^3)$.  Let
$a=(A-E)^{1/3}$.  Since $h^3E=E$ and $a^3=A-E$,
\begin{equation*}
 \mathbb P_{x\neq0}(E+w^3)
 =\mathbb P_{x\neq0}\left(
 h^3(b^3-a^3)+h^3b^3(\rho^3-1)\right).
\end{equation*}
Rapid Fourier decay of $a$ and the Wiener algebra inequality give
\begin{align*}
 \norm{b-a}_{\A^0_t}
 &\leq \mu^{-N/2}\norm a_{\A^N_t},\\
 \norm{b^3-a^3}_{\A^0_t}
 &\leq C\norm{b-a}_{\A^0_t}
   \bigl(\norm a_{\A^0_t}+\norm b_{\A^0_t}\bigr)^2
 \leq C_{E,A,N}\mu^{-N/2}.
\end{align*}
The nonzero coefficients of $\rho^3-1$ lie in
$(\mu\Z)^2\setminus\{0\}$ and are bounded by
\eqref{eq:cubic-profile-cubic-coefficient-bound}.  Since the Fourier support of
$b^3$ lies in $\{\abs n,\abs m\leq6\mu^{1/2}\}$,
\begin{equation*}
 \norm{\mathbb P_{x\neq0}(E+w^3)}_{\A^{-S}_{x\neq0,t}}
 \leq C_{E,A,N}\mu^{-N/2}+C_{E,A,S}\mu^{-S}.
\end{equation*}

\subsubsection*{Nash error}

The Nash error is $\mathbb P_{x\neq0}(3u^2w+3uw^2)$.
Minkowski's inequality and \eqref{eq:periodic-AF}
give
\begin{align*}
 \norm{\mathbb P_{x\neq0}(3u^2w)}_{\A^{-S}_{x\neq0,t}}
 &\leq3\operatorname{AF}_{S,3}^x(u,u,w),\\
 \norm{\mathbb P_{x\neq0}(3uw^2)}_{\A^{-S}_{x\neq0,t}}
 &\leq3\operatorname{AF}_{S,3}^x(u,w,w).
\end{align*}

\subsubsection*{Dispersion error}

The dispersion error is
\[
 (-1)^{(d+1)/2}\partial_x^{d-1}w
 +\kappa\partial_x^{-2}\partial_y^2w.
\]
Its multiplier is bounded by
$C\langle(n,m)\rangle^{d-1}$.  Therefore
\begin{align*}
 \norm{(-1)^{(d+1)/2}\partial_x^{d-1}w+
  \kappa\partial_x^{-2}\partial_y^2w}
       _{\A^{-S}_{x\neq0,t}}
 &\leq C\norm w_{C_tL^1}
 \sum_{\substack{n\neq0\\m\in\Z}}
   \langle(n,m)\rangle^{d-1-S}\\
 &\leq C_{S}\norm w_{C_tL^1}.
\end{align*}
Equation~\eqref{eq:cubic-profile-Lp} gives
\[
 \norm w_{C_tL^1}
 \leq C_{E,A,\eps}
 \la^{-4(1-\eps)/3}.
\]
\subsubsection*{Temporal error}

The temporal error is $\partial_t\partial_x^{-1}w$.  Using
\eqref{eq:multiplierL1}, \eqref{eq:time-cutoff-lemma}, and
$\partial_tw=h'b\rho+h\partial_tb\,\rho$ gives
\begin{equation*}
 \norm{\partial_t\partial_x^{-1}w}_{\A^{-S}_{x\neq0,t}}
 \leq C_{E,A,S,\eps}(1+\la^\beta)
 \la^{-4(1-\eps)/3}\longrightarrow0.
\end{equation*}

For every sufficiently large admissible $\la$, choose the corresponding
cutoff $h$.  The four estimates and \eqref{eq:periodic-update-general}
then prove \eqref{eq:cubic-periodic-onestep-error}.  Expanding
$\operatorname{AF}_{S,3}^x(u+w)$ and applying the triangle inequality
together with Proposition~\ref{prop:cubic-periodic-estimate} gives
\eqref{eq:cubic-periodic-onestep-AF}.

The profile estimate \eqref{eq:cubic-profile-Lp} gives
\begin{equation}\label{eq:cubic-periodic-onestep-Lp}
 \norm w_{C_tL^p}
 \leq C_{E,A,p,\eps}
 \la^{2(1-\eps)(1/3-1/p)},
 \qquad1\leq p<3.
\end{equation}
Furthermore, \eqref{eq:cubic-profile-support}, the amplitude cutoff,
Plancherel's theorem, and \eqref{eq:cubic-profile-endpoints} give
\begin{equation}\label{eq:cubic-periodic-onestep-Halpha}
 \begin{aligned}
 \norm w_{C_tH^\alpha}
 &\leq C_{E,A,\alpha,\eps}
 \la^{\alpha-(1-\eps)/3}
 &&\text{for the isotropic profile},\\
 \norm w_{C_tH^{\alpha,0}}
 &\leq C_{E,A,\alpha,\eps}
 \la^{\alpha-(1-\eps)/2}
 &&\text{for the parabolic profile}.
 \end{aligned}
\end{equation}
Equations \eqref{eq:cubic-periodic-onestep-Halpha} and
\eqref{eq:cubic-periodic-onestep-Lp}, together with
\eqref{eq:cubic-periodic-alpha-condition}, prove
\eqref{eq:cubic-periodic-onestep-regularity}.

Equation~\eqref{eq:cubic-profile-support} and $\supp\widehat b\subset
\{\abs n,\abs m\leq2\mu^{1/2}\}$ give
\eqref{eq:cubic-periodic-support}.  For sufficiently large $\la$, this
support is disjoint from $\supp\widehat u$.
\end{proof}

\begin{proposition}
\label{prop:cubic-Euclidean-onestep}
Fix $d\in\{3,5\}$, $\kappa\in\{-1,1\}$,
$0<\eps<1$,
\[
 0<\beta<\frac43(1-\eps),
\]
and $S>d+1$.  Fix $M>S+6$.
Let $\rho$ be either $\rho_{3,\la,\eps}$ or
$\widetilde{\rho}_{3,\la,\eps}$ from
Lemma~\ref{lem:cubic-profile}.
Suppose $u,E\in C^\infty([0,1];\mathcal S(\R^2))$ are real, have
compact spatial Fourier support, and satisfy
\eqref{eq:relaxed-general} with $r=3$.
Assume, for some $c_0>0$, that
\[
 \widehat u(t,\xi,\eta)=0\quad\text{for }\abs\xi<c_0
\]
and that the time supports of $u$ and $E$ are contained in a closed
interval $I\Subset(0,1)$.  Given $\delta_w,\delta_+>0$, finitely many
$1\leq p_j<3$, and finitely many $\alpha_j\geq0$ satisfying
\[
 \begin{aligned}
 \alpha_j&<\frac13(1-\eps)
 &&\text{for the isotropic profile},\\
 \alpha_j&<\frac12(1-\eps)
 &&\text{for the parabolic profile},
 \end{aligned}
\]
one can choose $A,L,\ell$ so that every sufficiently large admissible
$\la$ admits a cutoff $h=1$ on $I$, with the parameters selected in the
order
\[
 A\longrightarrow L\longrightarrow\ell\longrightarrow\la.
\]
The parameter $A$ satisfies
\eqref{eq:cubic-Euclidean-A}, and $a_L,a_{L,\ell}$ are then defined by
\eqref{eq:cubic-Euclidean-amplitude}.  With
\[
 w=h a_{L,\ell}\rho,\qquad u^+=u+w,
\]
define $E^+$ by \eqref{eq:Euclidean-update-general} with $r=3$.  Then
\begin{equation}\label{eq:cubic-Euclidean-onestep-regularity}
 \begin{aligned}
 \norm w_{C_tL^{p_j}}+\norm w_{C_tH^{\alpha_j}}&<\delta_w
 &&\text{for the isotropic profile},\\
 \norm w_{C_tL^{p_j}}+\norm w_{C_tH^{\alpha_j,0}}&<\delta_w
 &&\text{for the parabolic profile}.
 \end{aligned}
\end{equation}
\begin{align}
 \norm{\partial_xE^+}_{\A^{-S}_{\R^2,t}}&<\delta_+,
 \notag\\
 \operatorname{AF}^{\partial_x}_{S,3}(u^+)
 &\leq\operatorname{AF}^{\partial_x}_{S,3}(u)
 +\norm{\partial_xE}_{\A^{-S}_{\R^2,t}}+\delta_w.
 \label{eq:cubic-Euclidean-onestep-AF}
\end{align}
Both $u^+$ and $E^+$ are smooth, real, and Schwartz in space, have compact
spatial Fourier support, and satisfy \eqref{eq:relaxed-general}.
Their time supports lie in the open $\la^{-\beta}$-neighborhood of $I$.

The parameters may be chosen so that
\begin{equation}\label{eq:cubic-Euclidean-onestep-separation}
 \mu\geq16\ell,
 \qquad
 \frac\mu2>
 4\max\left(\{1,c_0\}\cup
 \{\abs{(\xi,\eta)}:(\xi,\eta)\in\supp\widehat u\}\right),
\end{equation}
and, on $\supp\widehat w$,
\begin{equation}\label{eq:cubic-Euclidean-onestep-support}
 \begin{aligned}
 3\la\leq\abs\xi\leq18\la,\quad \abs\eta\leq2\la
 &\qquad\text{for the isotropic profile},\\
 3\la\leq\abs\xi\leq18\la,\quad \abs\eta\leq2\la^2
 &\qquad\text{for the parabolic profile}.
 \end{aligned}
\end{equation}
Thus $\supp\widehat u\cap\supp\widehat w=\varnothing$ and
\[
 \widehat{u^+}(t,\xi,\eta)=0\qquad(\abs\xi<c_0).
\]
\end{proposition}

\begin{proof}
Choose $A\geq1$ satisfying \eqref{eq:cubic-Euclidean-A} and
$C_{E,S,M,\chi}A^{-1}<\delta_w/8$.  With $A$ fixed, choose $L$
so large that
\[
 C_{E,S,M,\chi}\left(\frac AL+\frac1L\right)<\frac{\delta_w}8,
 \qquad \frac{C_{E,\chi}+C_\chi A}{L}<\frac{\delta_+}{8}.
\]
After $A$ and $L$ have been fixed, choose $\ell$ sufficiently large
that $C_{E,A,M,\chi}\ell^{-M}<\delta_+/8$.

Choose an admissible $\la$ satisfying
\eqref{eq:cubic-Euclidean-onestep-separation} and such that the
$\la^{-\beta}$-neighborhood of $I$ lies in
$(0,1)$.  Choose $h$ satisfying \eqref{eq:time-cutoff-lemma} with
$\tau=\la^{-\beta}$.
\subsubsection*{Oscillation error}

The oscillation error is $\partial_x(E+w^3)$.  Since $hE=E$,
\begin{equation*}
 E+w^3=h^3\left(
 (1-\chi_L^3)E+A\chi_L^3
 +(a_{L,\ell}^3-a_L^3)
 +a_{L,\ell}^3(\rho^3-1)\right).
\end{equation*}
Since the triple convolution of $\widehat{\chi_L}$ has integral one,
\begin{align*}
 \widehat{(1-\chi_L^3)E}(k)
 =\int_{\R^2}
 (\widehat{\chi_L}*\widehat{\chi_L}*\widehat{\chi_L})(z)
 \bigl(\widehat E(k)-\widehat E(k-z)\bigr)\,dz.
\end{align*}
The Schwartz regularity of $E$ gives
\[
 \int_{\R^2}2\pi\abs\xi\langle k\rangle^{-S}
 \norm{\widehat E(\cdot,k)-\widehat E(\cdot,k-z)}_{C_t}\,dk
 \leq C_E\abs z.
\]
Together with \eqref{eq:cubic-spatial-cutoff-moment}, this gives
\[
 \norm{\partial_x((1-\chi_L^3)E)}_{\A^{-S}_{\R^2,t}}
 \leq\frac{C_{E,\chi}}L,
 \qquad
 \norm{\partial_x(A\chi_L^3)}_{\A^{-S}_{\R^2,t}}
 \leq\frac{C_\chi A}L.
\]
The binomial series and \eqref{eq:cubic-Euclidean-A} give
$\sup_{L\geq1}\norm{a_L}_{\A^M_{\R^2,t}}<\infty$.  Since $1-m_\ell$
is supported where $\abs k\gtrsim\ell$,
\[
 \norm{a_L-a_{L,\ell}}_{\A^0_{\R^2,t}}
 \leq C_{E,A,M,\chi}\ell^{-M}.
\]
The Wiener algebra inequality consequently gives
\[
 \norm{\partial_x(a_{L,\ell}^3-a_L^3)}
       _{\A^{-S}_{\R^2,t}}
 \leq C_{E,A,M,\chi}\ell^{-M}.
\]
The nonzero coefficients of $\rho^3-1$ are bounded by
\eqref{eq:cubic-profile-cubic-coefficient-bound}.  Since
$\supp\widehat{a_{L,\ell}^3}\subset\{\abs k\leq6\ell\}$,
\begin{align*}
 \norm{\partial_x\left[
 a_{L,\ell}^3(\rho^3-1)\right]}
 _{\A^{-S}_{\R^2,t}}
 &\leq C_{E,A,S,\chi}
 \sum_{s\in(\mu\Z)^2\setminus\{0\}}\abs s^{1-S}\\
 &\leq C_{E,A,S,\chi}\mu^{1-S}.
\end{align*}
Thus
\begin{equation*}
 \norm{\partial_x(E+w^3)}_{\A^{-S}_{\R^2,t}}
 \leq\frac{C_{E,\chi}+C_\chi A}{L}
 +C_{E,A,M,\chi}\ell^{-M}
 +C_{E,A,S,\chi}\mu^{1-S}.
\end{equation*}

\subsubsection*{Nash error}

The Nash error is $\partial_x(3u^2w+3uw^2)$.  The trilinear definition
and Minkowski's inequality yield
\begin{align*}
 \norm{\partial_x(3u^2w)}_{\A^{-S}_{\R^2,t}}
 &\leq3\operatorname{AF}^{\partial_x}_{S,3}(u,u,w),\\
 \norm{\partial_x(3uw^2)}_{\A^{-S}_{\R^2,t}}
 &\leq3\operatorname{AF}^{\partial_x}_{S,3}(u,w,w).
\end{align*}

\subsubsection*{Dispersion error}

The dispersion error is
$\partial_x\bigl((-1)^{(d+1)/2}\partial_x^{d-1}w
+\kappa\partial_x^{-2}\partial_y^2w\bigr)$.
For $q\in\supp\widehat\rho$ and $\abs\zeta\leq2\ell$,
$\abs{q_x+\zeta_x}\simeq\la$, and
\[
 \abs{q_x+\zeta_x}^d\langle q+\zeta\rangle^{-S}
 \lesssim_{d,S}\la^{d-S},
 \qquad
 \frac{\abs{q_y+\zeta_y}^2}{\abs{q_x+\zeta_x}}
 \langle q+\zeta\rangle^{-S}
 \lesssim_S\la^{1-S}.
\]
Indeed, for $S>2$,
\[
 \sup_{y\in\R}\frac{y^2}{\la}(\la^2+y^2)^{-S/2}
 =\la^{1-S}\sup_{z\in\R}z^2(1+z^2)^{-S/2}
 \lesssim_S\la^{1-S}.
\]
Together with
\eqref{eq:cubic-profile-Fourier-mass}, these estimates give
\begin{equation*}
 \begin{aligned}
 &\norm{\partial_x\bigl((-1)^{(d+1)/2}\partial_x^{d-1}w+
 \kappa\partial_x^{-2}\partial_y^2w\bigr)}_{\A^{-S}_{\R^2,t}}\\
 &\qquad\leq C_{E,A,L,\ell,d,S,\eps}
 \la^{d-S}\sum_q\widehat\rho(q)
 \leq C_{E,A,L,\ell,d,S,\eps}\la^{d-S+1-\eps}
 \longrightarrow0.
 \end{aligned}
\end{equation*}

\subsubsection*{Temporal error}

The temporal error is $\partial_tw$.  Differentiating the amplitude gives
\[
 \partial_ta_L=-\frac{\chi_L\partial_tE}{3(A-E)^{2/3}}.
\]
Since
\[
 \partial_tw=h'a_{L,\ell}\rho
 +h(\partial_ta_{L,\ell})\rho,
\]
\eqref{eq:Euclidean-periodic-decoupling} with $p=1$ gives
\begin{equation*}
 \norm{\partial_tw}_{\A^{-S}_{\R^2,t}}
 \leq C_{E,A,L,\ell,S,\eps}
 (1+\la^\beta)
 \la^{-4(1-\eps)/3}\longrightarrow0.
\end{equation*}

With $\ell$ fixed, these four estimates and
\eqref{eq:Euclidean-update-general} give
$\norm{\partial_xE^+}_{\A^{-S}_{\R^2,t}}<\delta_+$ for every sufficiently
large admissible $\la$.
Expanding $\operatorname{AF}^{\partial_x}_{S,3}(u+w)$ and applying the
triangle inequality together with
Proposition~\ref{prop:cubic-Euclidean-estimate} gives
\eqref{eq:cubic-Euclidean-onestep-AF}.  Equations
\eqref{eq:Euclidean-periodic-decoupling} and
\eqref{eq:cubic-profile-Lp} give
\begin{equation}\label{eq:cubic-Euclidean-onestep-Lp}
 \norm w_{C_tL^p(\R^2)}
 \leq C_{E,A,L,\ell,p,\eps}
 \la^{2(1-\eps)(1/3-1/p)},
 \qquad1\leq p<3.
\end{equation}
The support of the profile, the amplitude cutoff, and
\eqref{eq:cubic-profile-endpoints} give
\begin{equation*}
 \begin{aligned}
 \norm w_{C_tH^\alpha}
 &\leq C_{E,A,L,\ell,\alpha,\eps}
 \la^{\alpha-(1-\eps)/3}
 &&\text{for the isotropic profile},\\
 \norm w_{C_tH^{\alpha,0}}
 &\leq C_{E,A,L,\ell,\alpha,\eps}
 \la^{\alpha-(1-\eps)/2}
 &&\text{for the parabolic profile}.
 \end{aligned}
\end{equation*}
The Sobolev estimates and \eqref{eq:cubic-Euclidean-onestep-Lp} prove
\eqref{eq:cubic-Euclidean-onestep-regularity}.

Equation~\eqref{eq:cubic-profile-support} and the amplitude cutoff give
\eqref{eq:cubic-Euclidean-onestep-support}.  Condition
\eqref{eq:cubic-Euclidean-onestep-separation} gives
$\supp\widehat u\cap\supp\widehat w=\varnothing$ and
$\widehat{u^+}(t,\xi,\eta)=0$ for $\abs\xi<c_0$.
\end{proof}

\section{Iteration}

Fix $r\in\{2,3\}$, $d\in\{3,5\}$, and
$\kappa\in\{-1,1\}$, and assume throughout this section that
$S>d+1$.  For $r=2$, take the profile in Lemma~\ref{lem:slab}.  For
$r=3$, carry out the induction with $\rho_{3,\la,\eps}$ and with
$\widetilde{\rho}_{3,\la,\eps}$ from
Lemma~\ref{lem:cubic-profile}.

Set
\[
 p_0=1,\qquad p_j=r-2^{-j}\quad(j\geq1).
\]
Set
\[
 \alpha_j=\frac13(1-2^{-j}),\qquad
 s_j=\frac12(1-2^{-j})\quad(j\geq0),
\]
for the isotropic and parabolic profiles, respectively.
For $q\geq1$ let
\[
 \eps_q=2^{-q-4},\qquad \beta=\frac13,\qquad
 \delta_q=2^{-q-20},\qquad
 \mu_q=\la_q^{\eps_q},\quad \nu_q=\la_q^{1+\eps_q},
\]
and put $\delta_0=2^{-20}$.  When $0\leq j\leq q$,
\[
 \beta<\frac12(1-\eps_q)<\frac43(1-\eps_q),\qquad
 \alpha_j<\frac13(1-\eps_q),\qquad
 s_j<\frac12(1-\eps_q).
\]
For every sufficiently large integer $\mu_q$, the choice
$\la_q=\mu_q^{2^{q+4}}$ is admissible.

Let
\[
 I_0=\left[\frac{11}{32},\frac{21}{32}\right],
 \qquad I_*=\left[\frac38,\frac58\right],
\]
and define
\[
 I_q=\left[
  \frac{11}{32}-\sum_{j=1}^q\la_j^{-\beta},
  \frac{21}{32}+\sum_{j=1}^q\la_j^{-\beta}
 \right].
\]
The bound $\la_q^{-\beta}\leq2^{-q-8}$ makes the time-support
enlargements summable and keeps
$I_*\subset I_q\Subset(5/16,11/16)$ for every $q$.  Fix
$\theta\in C_c^\infty(\operatorname{int}I_0)$ equal to one on $I_*$.

\subsection{Periodic induction}

Fix $m_*\in\Z\setminus\{0\}$ and $\gamma>0$, and let $n_*\geq1$ be an
integer to be chosen.  Put
\begin{equation*}
 u_0(t,x,y)=\gamma\theta(t)
 \cos(2\pi n_*x)\cos(2\pi m_*y)
\end{equation*}
and
\begin{equation}\label{eq:periodic-induction-initial-pair}
 E_0=\mathbb P_{x\neq0}\left(
 u_0^r+(-1)^{(d+1)/2}\partial_x^{d-1}u_0
 +\kappa\partial_x^{-2}\partial_y^2u_0
 +\partial_t\partial_x^{-1}u_0\right).
\end{equation}
The pair satisfies the relaxed form of \eqref{eq:KP-family}, and
\begin{equation}\label{eq:periodic-induction-initial-error}
 \norm{E_0}_{\A^{-S}_{x\neq0,t}}
 \leq C_{\theta,m_*,S,r}\left(
 \gamma n_*^{d-1-S}+\gamma n_*^{-S-1}
 +\gamma n_*^{-S-2}+\gamma^r n_*^{-S}\right).
\end{equation}
Choose $n_*$ so large that
$\norm{E_0}_{\A^{-S}_{x\neq0,t}}<\delta_0$.  Also,
\begin{equation}\label{eq:periodic-induction-initial-coefficient}
 \widehat u_0(t,n_*,m_*)=\frac\gamma4\theta(t).
\end{equation}

\begin{proposition}\label{prop:inductive}
There are increasing admissible frequencies $(\la_q)_{q\geq1}$ and smooth real pairs
$(u_q,E_q)_{q\geq0}$ with finite spatial Fourier support.  Put
\[
 K_q=\max\left(\{1\}\cup
 \left\{\langle(n,m)\rangle:
 (n,m)\in\supp\widehat u_q\cup\supp\widehat E_q\right\}\right).
\]
For every $q$, the following hold.
\begin{enumerate}[label=\textup{(\roman*)}]
 \item Both $u_q$ and $E_q$ have zero $x$-mean and
 \begin{equation}\label{eq:periodic-induction-indexed-relaxed}
  \partial_tu_q+(-1)^{(d+1)/2}\partial_x^du_q
  +\kappa\partial_x^{-1}\partial_y^2u_q+\partial_x(u_q^r)
  =\partial_xE_q.
 \end{equation}

 \item The time supports of $u_q$ and $E_q$ are contained in $I_q$.

 \item For $q\geq1$, let $w_q=u_q-u_{q-1}$.  On
 $\supp\widehat w_q$,
 \begin{equation*}
 \begin{aligned}
  r=2:&\quad \frac12\mu_q\leq\abs n\leq3\nu_q,
       \qquad \abs m\leq2\mu_q^{1/2},\\
  \text{isotropic profile}:&\quad
       3\la_q\leq\abs n\leq18\la_q,
       \qquad \abs m\leq2\la_q,\\
  \text{parabolic profile}:&\quad
       3\la_q\leq\abs n\leq18\la_q,
       \qquad \abs m\leq2\la_q^2
 \end{aligned}
 \end{equation*}
 The parameters satisfy
 \begin{equation*}
  \mu_q\geq256,\qquad16K_{q-1}\leq\mu_q^{1/2}.
 \end{equation*}
 Consequently the sets $\supp\widehat w_q$ are pairwise disjoint,
 $K_q>K_{q-1}$, and, for some $C_*>0$ independent of $q$,
 \begin{equation}\label{eq:periodic-induction-support}
  \supp\widehat u_q\subset
  \begin{cases}
   \{(n,m):\abs m\leq C_*\abs n\},
      &r=2\text{ or for the isotropic profile},\\
   \{(n,m):\abs m\leq C_*\abs n^2\},
      &\text{for the parabolic profile}.
  \end{cases}
 \end{equation}

 \item For $q\geq1$ and $0\leq j\leq q$,
 \begin{equation*}
  \norm{w_q}_{C_tL^{p_j}(\T^2)}<\delta_{q-1}.
 \end{equation*}
 For the isotropic profile,
 \begin{equation}\label{eq:periodic-induction-Sobolev-diagonal}
  \norm{w_q}_{C_tH^{\alpha_j}(\T^2)}
  <\delta_{q-1}
  \qquad(0\leq j\leq q).
 \end{equation}
 For the parabolic profile,
 \begin{equation}\label{eq:periodic-induction-parabolic-Sobolev-diagonal}
  \norm{w_q}_{C_tH^{s_j,0}(\T^2)}
  <\delta_{q-1}
  \qquad(0\leq j\leq q).
 \end{equation}

 \item The coefficient at the fixed mode $(n_*,m_*)$ is preserved:
 \begin{equation}\label{eq:periodic-induction-preserved-coefficient}
  \widehat u_q(t,n_*,m_*)=\frac\gamma4\theta(t).
 \end{equation}

 \item The errors satisfy
 \begin{equation}\label{eq:periodic-induction-error}
  \norm{E_q}_{\A^{-S}_{x\neq0,t}}<\delta_q.
 \end{equation}

 \item
 \begin{equation}\label{eq:Qbound}
  \operatorname{AF}_{S,r}^x(u_q)
  \leq\operatorname{AF}_{S,r}^x(u_0)
   +2\sum_{j=0}^{q-1}\delta_j.
 \end{equation}
\end{enumerate}
\end{proposition}

\begin{proof}
Equations \eqref{eq:periodic-induction-initial-pair}--
\eqref{eq:periodic-induction-initial-coefficient} establish the base case.

Assume the induction hypotheses hold at index $q$.  If $r=2$, apply
Proposition~\ref{prop:onestep} to $(u_q,E_q)$, $I_q$, and
$p_0,\ldots,p_{q+1}$.  For the isotropic profile, apply
Proposition~\ref{prop:cubic-periodic-onestep} to the same pair and
interval with $\rho_{3,\la,\eps}$ and
$\alpha_0,\ldots,\alpha_{q+1}$.  For the parabolic profile, use
$\widetilde{\rho}_{3,\la,\eps}$ and
$s_0,\ldots,s_{q+1}$.  In every case take
\[
 \eps=\eps_{q+1},\qquad \eta=\delta_q,\qquad
 \delta_+=\delta_{q+1}.
\]
Choose $A$ and then an admissible
$\la_{q+1}$ so large that the required increment, error, and support
estimates hold, with $\la_{q+1}>\la_q$ when $q\geq1$, and
\begin{equation}\label{eq:periodic-induction-parameter-choice}
 \mu_{q+1}\geq256,\qquad
 16K_q\leq\mu_{q+1}^{1/2},\qquad
 \la_{q+1}^{-\beta}\leq2^{-q-9}.
\end{equation}
Let $w_{q+1}$ and $E_{q+1}$ be the perturbation and updated error given
by the one-step proposition, and set
$u_{q+1}=u_q+w_{q+1}$.
The one-step support bound and
\eqref{eq:periodic-induction-parameter-choice} make
$\supp\widehat w_{q+1}$ disjoint from
$\supp\widehat u_q$.  The support bounds for $w_{q+1}$, together with
the finite support of $u_0$, give
\eqref{eq:periodic-induction-support}.  The amplitude is nonzero on $I_q$, so
$w_{q+1}$ is nonzero and
\[
 K_{q+1}\geq\frac12\mu_{q+1}
 \geq128K_q^2>K_q.
\]
Since $(n_*,m_*)\notin\supp\widehat w_{q+1}$,
\eqref{eq:periodic-induction-preserved-coefficient} is preserved.  The one-step
absolute Fourier estimate and \eqref{eq:periodic-induction-error} give
\begin{align*}
 \operatorname{AF}_{S,r}^x(u_{q+1})
 \leq\operatorname{AF}_{S,r}^x(u_q)+2\delta_q
 \leq\operatorname{AF}_{S,r}^x(u_0)
   +2\sum_{j=0}^{q}\delta_j.
\end{align*}
\end{proof}

\subsection{Whole-space induction}

Choose $N_*,M_*>0$ and
$0<\varrho_*<\frac14\min\{N_*,M_*\}$.  Let
$f\in\mathcal S(\R^2)$ be real and nonzero, with smooth compactly
supported Fourier transform contained in the four balls of radius
$\varrho_*$ centered at $(\pm N_*,\pm M_*)$, and nonzero in each ball.
Put $c_0=N_*-\varrho_*>0$, $u_0=\gamma\theta f$, and
\begin{equation}\label{eq:Euclidean-induction-initial-pair}
 E_0=u_0^r+(-1)^{(d+1)/2}\partial_x^{d-1}u_0
 +\kappa\partial_x^{-2}\partial_y^2u_0
 +\partial_t\partial_x^{-1}u_0.
\end{equation}
Then
\begin{equation}\label{eq:Euclidean-induction-initial-error}
 \norm{\partial_xE_0}_{\A^{-S}_{\R^2,t}}
 \leq C_{f,\theta,S,r}(\gamma+\gamma^r).
\end{equation}
Fix a nonzero $\gamma$ so small that
$\norm{\partial_xE_0}_{\A^{-S}_{\R^2,t}}<\delta_0$, and fix an integer
$M>S+6$.

\begin{proposition}\label{prop:Euclidean-induction}
There exist a sequence $(\ell_q)_{q\geq1}$ with $\ell_q\geq1$,
increasing admissible frequencies $(\la_q)_{q\geq1}$, and smooth real
pairs $(u_q,E_q)_{q\geq0}$, Schwartz in space and with
compact spatial Fourier support.  Put
\[
 K_q^x=\max\left(\{1\}\cup
 \{\abs\xi:(\xi,\eta)\in\supp\widehat u_q\}\right).
\]
\begin{enumerate}[label=\textup{(\roman*)}]
 \item Each pair satisfies
 \begin{equation}\label{eq:Euclidean-induction-indexed-relaxed}
  \partial_tu_q+(-1)^{(d+1)/2}\partial_x^du_q
  +\kappa\partial_x^{-1}\partial_y^2u_q+\partial_x(u_q^r)
  =\partial_xE_q.
 \end{equation}

 \item The time supports of $u_q$ and $E_q$ are contained in $I_q$.

 \item For $q\geq1$, let $w_q=u_q-u_{q-1}$.  On
 $\supp\widehat w_q$,
 \begin{equation*}
 \begin{aligned}
  r=2:&\quad \mu_q-2\ell_q\leq\abs\xi
       \leq2\nu_q+2\ell_q,
       \qquad \abs\eta\leq2\ell_q,\\
 \text{isotropic profile}:&\quad
       3\la_q\leq\abs\xi\leq18\la_q,
       \qquad \abs\eta\leq2\la_q,\\
 \text{parabolic profile}:&\quad
       3\la_q\leq\abs\xi\leq18\la_q,
       \qquad \abs\eta\leq2\la_q^2
 \end{aligned}
 \end{equation*}
 The parameters satisfy
 \begin{equation*}
  \mu_q\geq16\ell_q,\qquad
  \frac12\mu_q>
  8\max\left(\{1,c_0\}\cup
  \{\abs k:k\in\supp\widehat u_{q-1}\}\right).
 \end{equation*}
 Consequently the projections of $\supp\widehat w_q$ onto the $\xi$-axis
 are pairwise disjoint, and $K_q^x>K_{q-1}^x$.

 \item For every $q$,
 \begin{equation*}
  \widehat u_q(t,\xi,\eta)=0\qquad(\abs\xi<c_0),
 \end{equation*}
 and, for some $C_*>0$ independent of $q$,
 \begin{equation}\label{eq:Euclidean-induction-cone}
  \supp\widehat u_q\subset
  \begin{cases}
   \{(\xi,\eta):\abs\eta\leq C_*\abs\xi\},
      &r=2\text{ or for the isotropic profile},\\
   \{(\xi,\eta):\abs\eta\leq C_*\abs\xi^2\},
      &\text{for the parabolic profile}.
  \end{cases}
 \end{equation}

 \item For $q\geq1$ and $0\leq j\leq q$,
 \begin{equation*}
  \norm{w_q}_{C_tL^{p_j}(\R^2)}<\delta_{q-1}.
 \end{equation*}
 For the isotropic profile,
 \begin{equation}\label{eq:Euclidean-induction-Sobolev-diagonal}
  \norm{w_q}_{C_tH^{\alpha_j}(\R^2)}
  <\delta_{q-1}
  \qquad(0\leq j\leq q).
 \end{equation}
 For the parabolic profile,
 \begin{equation}\label{eq:Euclidean-induction-parabolic-Sobolev-diagonal}
  \norm{w_q}_{C_tH^{s_j,0}(\R^2)}
  <\delta_{q-1}
  \qquad(0\leq j\leq q).
 \end{equation}

 \item For every $q$,
 \begin{equation}\label{eq:Euclidean-induction-error}
  \norm{\partial_xE_q}_{\A^{-S}_{\R^2,t}}<\delta_q.
 \end{equation}

 \item
 \begin{equation}\label{eq:Euclidean-induction-AF-bound}
  \operatorname{AF}^{\partial_x}_{S,r}(u_q)
  \leq\operatorname{AF}^{\partial_x}_{S,r}(u_0)
   +2\sum_{j=0}^{q-1}\delta_j.
 \end{equation}

 \item On the initial Fourier support,
 \begin{equation}\label{eq:Euclidean-induction-preserved-support}
  \widehat u_q(t,\xi,\eta)=\gamma\theta(t)\widehat f(\xi,\eta)
  \qquad((\xi,\eta)\in\supp\widehat f).
 \end{equation}
\end{enumerate}
\end{proposition}

\begin{proof}
Equations \eqref{eq:Euclidean-induction-initial-pair} and
\eqref{eq:Euclidean-induction-initial-error} establish the base case; the
Fourier transform of $u_0$ vanishes for $\abs\xi<c_0$, and
\[
 \abs\eta\leq\frac{M_*+\varrho_*}{N_*-\varrho_*}\abs\xi,
 \qquad
 \abs\eta\leq
 \frac{M_*+\varrho_*}{(N_*-\varrho_*)^2}\abs\xi^2.
\]

Assume the induction hypotheses hold at index $q$.  If $r=2$, apply
Proposition~\ref{prop:Euclidean-onestep}.  For the isotropic profile,
apply Proposition~\ref{prop:cubic-Euclidean-onestep} with
$\rho_{3,\la,\eps}$ and
$\alpha_0,\ldots,\alpha_{q+1}$.  For the parabolic profile, use
$\widetilde{\rho}_{3,\la,\eps}$ and
$s_0,\ldots,s_{q+1}$.  In every case take $I=I_q$,
$\delta_w=\delta_q$, the gap parameter $c_0$, the exponents
$p_0,\ldots,p_{q+1}$, and
\[
 \eps=\eps_{q+1},\qquad \delta_+=\delta_{q+1}.
\]
Choose the parameters in the order
\[
 A\longrightarrow L\longrightarrow
 \ell_{q+1}\longrightarrow\la_{q+1}.
\]
After $A,L,\ell_{q+1}$ are fixed, take an admissible $\la_{q+1}$
sufficiently large that the one-step conclusions hold, with
$\la_{q+1}>\la_q$ when $q\geq1$, and
\begin{equation}\label{eq:Euclidean-induction-parameter-choice}
 \mu_{q+1}\geq16\ell_{q+1},\qquad
 \frac12\mu_{q+1}>
 8\max\left(\{1,c_0\}\cup
 \{\abs k:k\in\supp\widehat u_q\}\right),
 \qquad \la_{q+1}^{-\beta}\leq2^{-q-9}.
\end{equation}
Let $w_{q+1}$ and $E_{q+1}$ be the perturbation and updated error given
by the one-step proposition, and set
$u_{q+1}=u_q+w_{q+1}$.
The support estimate and
\eqref{eq:Euclidean-induction-parameter-choice} give
\[
 \abs\xi\geq\frac78\mu_{q+1}>14K_q^x
 \qquad\text{on }\supp\widehat w_{q+1}.
\]
Thus $\supp\widehat w_{q+1}$ is disjoint from
$\supp\widehat u_q$.  Hence $\widehat u_{q+1}=0$ for $\abs\xi<c_0$, and
\eqref{eq:Euclidean-induction-preserved-support} is preserved.  The
support bound for $w_{q+1}$ gives
\eqref{eq:Euclidean-induction-cone}.  The amplitude is
nonzero on $I_q$, so $w_{q+1}$ is nonzero and
$K_{q+1}^x>14K_q^x$.  The one-step absolute
Fourier estimate and \eqref{eq:Euclidean-induction-error} give
\[
 \operatorname{AF}^{\partial_x}_{S,r}(u_{q+1})
 \leq\operatorname{AF}^{\partial_x}_{S,r}(u_q)+2\delta_q.
\]
\end{proof}

\begin{proof}[Proof of Theorems~\ref{thm:periodic} and
\ref{thm:Euclidean}]
On $\T^2$, Proposition~\ref{prop:inductive} gives convergence for the
exponents $p_j$ in the iteration, and interpolation gives
\[
 u_q\longrightarrow u\quad\text{in }C_tL^p(\T^2),
 \qquad 1\leq p<r.
\]
For the isotropic cubic profile,
\eqref{eq:periodic-induction-Sobolev-diagonal} and interpolation give convergence in
$C_tH^\alpha$ for $0\leq\alpha<1/3$.  This controls $H^{\alpha,0}$
for $\alpha\geq0$, while the $L^2$ convergence controls both Sobolev
norms for $\alpha<0$.  Thus the convergence holds in
$C_t(H^{\alpha,0}\cap H^\alpha)$ for every $\alpha<1/3$.  For the
parabolic profile,
\eqref{eq:periodic-induction-parabolic-Sobolev-diagonal} and interpolation give convergence
in $C_tH^{s,0}$ for every $s<1/2$; on the support in
\eqref{eq:periodic-induction-support},
\[
 \norm f_{H^\alpha}\lesssim\norm f_{H^{2\alpha,0}}
 \quad(0\leq\alpha<1/4),
 \qquad
 \norm f_{H^\alpha}\leq\norm f_{L^2}
 \quad(\alpha<0),
\]
which gives $C_tH^\alpha$ for every $\alpha<1/4$.

For the quadratic profile, Hausdorff--Young and H\"older give, for every
$\alpha<0$ and some $p<2$,
\begin{equation}\label{eq:periodic-induction-negative-Sobolev}
 \norm f_{H^{\alpha,0}(\T^2)}
 \lesssim_{\alpha,p,C_*}\norm f_{L^p(\T^2)},
 \qquad
 \supp\widehat f\subset
 \{(n,m):\abs m\leq C_*\abs n\}.
\end{equation}
Applied to $u-u_q$, this gives convergence in $C_tH^{\alpha,0}$; the
bound $\abs m\leq C_*\abs n$ implies
$\langle(n,m)\rangle^\alpha\simeq\langle n\rangle^\alpha$, so the same
convergence holds in $C_tH^\alpha$.

The limit is real, has zero $x$-mean, satisfies
\eqref{eq:periodic-induction-support}, and has time support in
$(5/16,11/16)$.  Equation
\eqref{eq:periodic-induction-preserved-coefficient} gives
\[
 \widehat u(t,n_*,m_*)=\frac\gamma4\theta(t),
\]
so $\partial_yu\not\equiv0$.  Convergence in $C_tL^1$ gives uniform
convergence of every Fourier coefficient.  Fatou's lemma and
\eqref{eq:Qbound} then give
$\operatorname{AF}_{S,r}^x(u)<\infty$.

Let $P_{\leq K_q}^{x,y}$ be a smooth radial cutoff equal to one on the
ball of radius $K_q$ and zero outside the ball of radius $2K_q$.
The frequency inequalities in Proposition~\ref{prop:inductive} give
$\abs k>2K_q$ on $\supp\widehat w_j$ for every $j>q$.  Hence
\[
 P_{\leq K_q}^{x,y}u=u_q.
\]
Proposition~\ref{prop:cutoff} therefore gives
\[
 \mathbb P_{x\neq0}(u_q^r)
 \longrightarrow\mathbb P_{x\neq0}(u^r)
 \quad\text{in }\A^{-S}_{x\neq0,t}.
\]
Testing \eqref{eq:periodic-induction-indexed-relaxed} and using
\eqref{eq:periodic-induction-error} proves \eqref{eq:weakform} with
$u_{\mathrm{in}}=0$.

On $\R^2$, Proposition~\ref{prop:Euclidean-induction} gives convergence
for the exponents $p_j$, and interpolation between consecutive exponents
gives
\[
 u_q\longrightarrow u\quad\text{in }C_tL^p(\R^2),
 \qquad 1\leq p<r.
\]
For the isotropic cubic profile,
\eqref{eq:Euclidean-induction-Sobolev-diagonal} and interpolation give convergence in
$C_tH^\alpha$ for $0\leq\alpha<1/3$.  This controls $H^{\alpha,0}$
for $\alpha\geq0$, while the $L^2$ convergence controls both Sobolev
norms for $\alpha<0$.  Thus the convergence holds in
$C_t(H^{\alpha,0}\cap H^\alpha)$ for every $\alpha<1/3$.  For the
parabolic profile,
\eqref{eq:Euclidean-induction-parabolic-Sobolev-diagonal} and
interpolation give convergence in $C_tH^{s,0}$ for every $s<1/2$;
\eqref{eq:Euclidean-induction-cone} then gives convergence in
$C_tH^\alpha$ for every $\alpha<1/4$.  For $r=2$,
Hausdorff--Young and H\"older give, for every $\alpha<0$ and a suitable
$p<2$,
\[
 \norm g_{H^{\alpha,0}(\R^2)}
 \lesssim_{\alpha,p,C_*,c_0}\norm g_{L^p(\R^2)},
 \qquad
 \supp\widehat g\subset
 \{(\xi,\eta):\abs\xi\geq c_0,\ \abs\eta\leq C_*\abs\xi\}.
\]
Applied to $u-u_q$, this estimate and
\eqref{eq:Euclidean-induction-cone} give convergence in
$C_t(H^{\alpha,0}\cap H^\alpha)$ for every $\alpha<0$.

The limit is real, vanishes in Fourier space for $\abs\xi<c_0$, has time
support in $(5/16,11/16)$, and satisfies
\eqref{eq:Euclidean-induction-cone} and
\eqref{eq:Euclidean-induction-preserved-support}.
Equation~\eqref{eq:Euclidean-induction-preserved-support} gives
$\partial_yu\not\equiv0$.  Convergence in $C_tL^1$ gives uniform
convergence of the Fourier transforms.  Fatou's lemma and
\eqref{eq:Euclidean-induction-AF-bound} give
$\operatorname{AF}^{\partial_x}_{S,r}(u)<\infty$.

Let $P_{\leq K_q^x}^x$ be a smooth cutoff equal to one on
$[-K_q^x,K_q^x]$ and zero outside $[-2K_q^x,2K_q^x]$.
The frequency inequalities in Proposition~\ref{prop:Euclidean-induction}
give $\abs\xi>2K_q^x$ on $\supp\widehat w_j$ whenever $j>q$.  Hence
\[
 P_{\leq K_q^x}^xu=u_q.
\]
Proposition~\ref{prop:Euclidean-cutoff} therefore yields
\[
 \partial_x(u_q^r)=\partial_x[(P_{\leq K_q^x}^xu)^r]
 \longrightarrow\partial_x[u^r]
 \quad\text{in }\A^{-S}_{\R^2,t}.
\]
Testing \eqref{eq:Euclidean-induction-indexed-relaxed}, using
\eqref{eq:Euclidean-transverse-pairing} for the transverse term and
\eqref{eq:Euclidean-induction-error} for the error, proves
\eqref{eq:Euclidean-weak-form}.

Extend $u$ by zero in time and set
\[
 u_{\Lambda,t_0}(t,x,y)
 =\Lambda^{(d-1)/(r-1)}u\bigl(
 \Lambda^d(t-t_0),\Lambda x,\Lambda^{(d+1)/2}y\bigr).
\]
On $\T^2$, take $\Lambda\in\mathbb N$; the
nonzero Fourier coefficients satisfy
\[
 \widehat{u_{\Lambda,t_0}}
 \bigl(t,\Lambda n,\Lambda^{(d+1)/2}m\bigr)
 =\Lambda^{(d-1)/(r-1)}
 \widehat u\bigl(\Lambda^d(t-t_0),n,m\bigr).
\]
On $\R^2$,
\[
 \widehat{u_{\Lambda,t_0}}(t,\xi,\eta)
 =\Lambda^{(d-1)/(r-1)-(d+3)/2}
 \widehat u\left(\Lambda^d(t-t_0),
       \frac\xi\Lambda,\frac\eta{\Lambda^{(d+1)/2}}\right).
\]
These identities show that the Fourier convolutions commute with the
scaling.  Since
$r(d-1)/(r-1)+1=(d-1)/(r-1)+d$, a change of variables in the weak
formulation shows that $u_{\Lambda,t_0}$ solves the same equation.  The
Fourier formulas also give
\[
 \operatorname{AF}_{S,r}^x(u_{\Lambda,t_0})
 \leq\Lambda^{r(d-1)/(r-1)}\operatorname{AF}_{S,r}^x(u),
 \qquad
 \operatorname{AF}^{\partial_x}_{S,r}(u_{\Lambda,t_0})
 \leq\Lambda^{r(d-1)/(r-1)+1}
       \operatorname{AF}^{\partial_x}_{S,r}(u).
\]
Thus the stated Lebesgue and Sobolev regularity and the support inclusions
are preserved.  The
time support of $u_{\Lambda,t_0}$ is dilated by $\Lambda^{-d}$; on
$\R^2$, its Fourier transform vanishes for $\abs\xi<\Lambda c_0$.
Taking $\Lambda$ sufficiently large and then varying $t_0$ produces
infinitely many solutions with the prescribed time support and, on
$\R^2$, the prescribed Fourier gap.

To obtain smallness in a fixed finite collection of the stated norms, fix
$\Lambda$.  Each scaled norm is bounded by a finite constant depending on
$\Lambda$ times the corresponding unscaled norm, so it suffices to impose
the resulting smaller bounds in the iteration.  In the cubic case, choose
$Q$ sufficiently large for all the prescribed Sobolev exponents and set
\[
 \eps_q=2^{-q-Q-4},\qquad \la_q=\mu_q^{2^{q+Q+4}},
\]
and impose all the prescribed nonnegative Sobolev and Lebesgue bounds at
every step.  Include the auxiliary $L^2$ bound, which controls every
prescribed negative cubic $H^\alpha$ and $H^{\alpha,0}$ norm.
For increments from the parabolic profile, whose Fourier support satisfies
$\abs{k_y}\lesssim\abs{k_x}^2$, use
\[
 \norm f_{H^\alpha}\lesssim\norm f_{H^{2\alpha,0}}
 \quad(\alpha\geq0),
 \qquad
 \norm f_{H^\alpha}\leq\norm f_{L^2}
 \quad(\alpha<0).
\]
For each prescribed negative quadratic Sobolev exponent, choose $p<2$
close enough to $2$ that \eqref{eq:periodic-induction-negative-Sobolev}
or the whole-space Hausdorff--Young estimate applies, and impose the
corresponding $L^p$ bound at every step.  Set
$\delta_q=c2^{-q-20}$ for $q\geq0$, with $c>0$ sufficiently small.  On
$\T^2$, choose $n_*$ so that
\eqref{eq:periodic-induction-initial-error} is smaller than $\delta_0$
for every $0<\gamma\leq1$; then choose $\gamma$ so that $u_0$ satisfies
the prescribed bounds.  On $\R^2$, choose $\gamma$ so that $u_0$ satisfies
the prescribed bounds and \eqref{eq:Euclidean-induction-initial-error} is
smaller than $\delta_0$.  The initial norms and the sums of the increment
norms then satisfy the prescribed bounds.

\end{proof}

\appendix

\section{Periodic stationary solutions}

We adapt the stationary iteration of~\cite{Pathak}.

For a time-independent distribution $f$ on $\T^2$, set
\[
 \norm f_{\A^s_{x\neq0}}
 :=\sum_{\substack{n\neq0\\m\in\Z}}
   \langle(n,m)\rangle^s\abs{\widehat f(n,m)}.
\]
\begin{proof}[Proof of Theorem~\ref{thm:stationary-flexibility}]
For time-independent functions and $h\equiv1$,
\eqref{eq:periodic-update-general} gives
\begin{equation}\label{eq:stationary-error-update}
 E^+=\mathbb P_{x\neq0}\left(
 E+w^2+2uw+(-1)^{(d+1)/2}\partial_x^{d-1}w
   +\kappa\partial_x^{-2}\partial_y^2w\right).
\end{equation}
Estimates \eqref{eq:ROestimate}, \eqref{eq:RNestimate}, and
\eqref{eq:RDestimate}, together with
Proposition~\ref{prop:periodic-absolute-summability}, give the estimates in
Proposition~\ref{prop:onestep} with the temporal error omitted.  At each
step, $\supp\widehat w_q$ is disjoint from
$\supp\widehat u_{q-1}$, and the frequencies may be chosen as in
\eqref{eq:periodic-induction-parameter-choice} so that the sets
$\supp\widehat w_q$ are pairwise disjoint.

Fix $m_*\in\Z\setminus\{0\}$ and $1\leq\gamma\leq2$.  For an integer
$n_*$, set
\[
 u_0(x,y)=\gamma\cos(2\pi n_*x)\cos(2\pi m_*y)
\]
and
\[
 E_0=\mathbb P_{x\neq0}\left(
   (-1)^{(d+1)/2}\partial_x^{d-1}u_0
   +\kappa\partial_x^{-2}\partial_y^2u_0+u_0^2
 \right).
\]
Direct calculation gives
\[
 \norm{E_0}_{\A^{-S}_{x\neq0}}
 \leq C_{d,m_*,S}\left(
    \gamma n_*^{d-1-S}+\gamma^2n_*^{-S}\right),
 \qquad
 \widehat u_0(n_*,m_*)=\frac\gamma4.
\]
Because $1\leq\gamma\leq2$, a single sufficiently large $n_*$ gives
$\norm{E_0}_{\A^{-S}_{x\neq0}}<\delta_0$ for every such $\gamma$.
For each $q$, set
\[
 K_q=\max\left(\{1\}\cup
 \{\langle k\rangle:k\in\supp\widehat u_q
       \cup\supp\widehat E_q\}\right).
\]

For $q\geq1$, apply the stationary one-step estimates with
$p_0,\ldots,p_q$, $\eps=\eps_q$, $\eta=\delta_{q-1}$, and
$\delta_+=\delta_q$.  Choose $A_q$ as required by those estimates, in
particular $A_q\geq1+2\norm{E_{q-1}}_{L^\infty}$, and then choose an
admissible $\la_q$ so that the one-step estimates hold and
\[
 \mu_q\geq256,\qquad 16K_{q-1}\leq\mu_q^{1/2}.
\]
Set
\[
 b_q=P_{\leq\mu_q^{1/2}}^{x,y}(A_q-E_{q-1})^{1/2},
 \qquad w_q=b_q\rho_{\la_q,\eps_q},
 \qquad u_q=u_{q-1}+w_q,
\]
with $E_q$ given by \eqref{eq:stationary-error-update}.
The nonzero frequencies of $\rho_{\la_q,\eps_q}^2$ have absolute
$x$-frequency at least $\mu_q$, whereas
$\supp\widehat{b_q^2}\subset\{\abs k\leq4\mu_q^{1/2}\}$.
Thus frequency separation and
$\widehat{\rho_{\la_q,\eps_q}^2}(0)=1$ give
$\int b_q^2\rho_{\la_q,\eps_q}^2=\int b_q^2$.
Moreover, the cutoff preserves the zero Fourier coefficient, so
\[
 \widehat b_q(0,0)=\int_{\T^2}(A_q-E_{q-1})^{1/2}
 \geq\left(\frac{A_q}{2}\right)^{1/2}.
\]
Consequently,
\[
 \norm{w_q}_{L^2(\T^2)}^2
 =\int_{\T^2}b_q^2
 \geq\abs{\widehat b_q(0,0)}^2
 \geq\frac{A_q}{2}\geq\frac12,
\]
and the sets $\supp\widehat w_q$ are pairwise disjoint.
The bounds for $w_q$ give convergence
of $u_q$ to a limit $u$ in every $L^p(\T^2)$, $1\leq p<2$, while
$E_q\to0$ in $\A^{-S}_{x\neq0}$.  The supports of the quadratic
perturbations lie in a fixed cone, so
\eqref{eq:periodic-induction-negative-Sobolev}, applied to $u-u_q$, gives
convergence in $H^{-\sigma,0}(\T^2)$ and
$H^{-\sigma}(\T^2)$ for every $\sigma>0$.  At each step,
Proposition~\ref{prop:periodic-absolute-summability} gives
\[
 \operatorname{AF}_S^x(u_{q+1})
 \leq\operatorname{AF}_S^x(u_q)+2\delta_q.
\]
Since $\sum_q\delta_q<\infty$, Fatou's lemma gives
$\operatorname{AF}_S^x(u)<\infty$.  Since
$(n_*,m_*)\notin\supp\widehat w_q$ for every $q$,
$\widehat u(n_*,m_*)=\gamma/4$.  Let $P_{\leq K_q}^{x,y}$ be the dilate
of one fixed smooth multiplier which is one on the ball of radius $K_q$
and zero outside the ball of radius $2K_q$.  For $j>q$,
\[
 \abs n\geq\frac12\mu_j\geq128K_{j-1}^2>2K_q
 \qquad\text{on }\supp\widehat w_j,
\]
so
\[
 P_{\leq K_q}^{x,y}u=u_q.
\]
Proposition~\ref{prop:cutoff} gives
\[
 \mathbb P_{x\neq0}(u_q^2)
 \longrightarrow\mathbb P_{x\neq0}(u^2)
 \quad\text{in }\A^{-S}_{x\neq0}.
\]
The definition of $E_0$ and \eqref{eq:stationary-error-update} give, for
every $q$,
\[
 E_q=\mathbb P_{x\neq0}\left(
 u_q^2+(-1)^{(d+1)/2}\partial_x^{d-1}u_q
 +\kappa\partial_x^{-2}\partial_y^2u_q\right).
\]
Letting $q\to\infty$ in this identity, using
$\mathbb P_{x\neq0}(u_q^2)\to\mathbb P_{x\neq0}(u^2)$,
$u_q\to u$ in distributions, and $E_q\to0$, gives the stationary
equation.  Varying
$\gamma\in[1,2]$ gives infinitely many distinct solutions.
Pairwise disjointness of the sets $\supp\widehat w_q$ gives
\[
 \norm{u_q}_2^2
 =\norm{u_0}_2^2+\sum_{j=1}^q\norm{w_j}_2^2
 \geq\frac q2.
\]
Since $P_{\leq K_q}^{x,y}u=u_q$, these smooth Fourier truncations have
unbounded $L^2$ norm.  If $u\in L^2(\T^2)$, Plancherel's theorem would
give $\norm{P_{\leq K_q}^{x,y}u}_2\leq C\norm u_2$ uniformly in
$q$.  Hence $L^2$ is the sharp threshold between the singular solutions
of Theorem~\ref{thm:stationary-flexibility} and the smooth solutions of
Theorem~\ref{thm:stationary-rigidity}.
\end{proof}

\begin{proof}[Proof of Theorem~\ref{thm:stationary-rigidity}]
The stationary equation gives, for every $n\neq0$ and $m\in\Z$,
\begin{equation}\label{eq:stationary-KPI-Fourier}
 \left((2\pi n)^{d-1}+\frac{m^2}{n^2}\right)\widehat u(n,m)
 =\widehat{u^2}(n,m),
 \qquad \widehat u(0,m)=0.
\end{equation}
For every $q>3/2$,
\begin{equation}\label{eq:stationary-reciprocal-sum}
 \sum_{\substack{n\neq0\\m\in\Z}}
 \left((2\pi n)^{d-1}+\frac{m^2}{n^2}\right)^{-q}<\infty.
\end{equation}
Indeed, for every $n\neq0$,
\begin{align*}
 \sum_{m\in\Z}
 \left((2\pi n)^{d-1}+\frac{m^2}{n^2}\right)^{-q}
 &=|n|^{2q}\sum_{m\in\Z}
       \bigl((2\pi)^{d-1}|n|^{d+1}+m^2\bigr)^{-q}\\
 &\leq C_q|n|^{2q}
       \left(|n|^{-(d+1)q}
       +\int_{\R}(|n|^{d+1}+s^2)^{-q}\,ds\right)
 \\
 &\leq C_{d,q}|n|^{(d+1)/2-(d-1)q}.
\end{align*}
Since $d\in\{3,5\}$, the last expression is bounded by
$C_{d,q}|n|^{2-2q}$, which proves
\eqref{eq:stationary-reciprocal-sum}.

Since $u^2\in L^1(\T^2)$, its Fourier coefficients are bounded.
Equations \eqref{eq:stationary-KPI-Fourier} and
\eqref{eq:stationary-reciprocal-sum} with $q=8/5$ imply
$\widehat u\in\ell^{8/5}(\Z^2)$.  The inverse Hausdorff--Young inequality gives
$u\in L^{8/3}(\T^2)$.  It follows that
$u^2\in L^{4/3}$ and
$\widehat{u^2}\in\ell^4$.  Using
\eqref{eq:stationary-reciprocal-sum} with $q=2$ and H\"older's inequality
in \eqref{eq:stationary-KPI-Fourier}, we obtain
\[
 \norm{\widehat u}_{\ell^{4/3}}
 \leq
 \left(\sum_{\substack{n\neq0\\m\in\Z}}
 \left((2\pi n)^{d-1}+\frac{m^2}{n^2}\right)^{-2}\right)^{1/2}
 \norm{\widehat{u^2}}_{\ell^4}<\infty.
\]
The inverse Hausdorff--Young inequality gives $u\in L^4(\T^2)$, and hence
$u^2\in L^2(\T^2)$.  Cauchy--Schwarz and
\eqref{eq:stationary-reciprocal-sum} now give
\[
 \norm{\widehat u}_{\ell^1}
 \leq
 \left(\sum_{\substack{n\neq0\\m\in\Z}}
 \left((2\pi n)^{d-1}+\frac{m^2}{n^2}\right)^{-2}\right)^{1/2}
 \norm{\widehat{u^2}}_{\ell^2}<\infty.
\]
Thus $u\in L^\infty(\T^2)$.

For $n\neq0$,
\[
 (2\pi n)^{d-1}+\frac{m^2}{n^2}
 \geq(2\pi n)^2+\frac{m^2}{n^2}
 \geq c\langle(n,m)\rangle.
\]
Plancherel and
\eqref{eq:stationary-KPI-Fourier} therefore give $u\in H^1(\T^2)$.
Since $u\in H^1\cap L^\infty$, one has $u^2\in H^1$, and the same Fourier
identity yields $u\in H^2$.

For every integer $j\geq2$, $H^j(\T^2)$ is an algebra.  Thus
\[
 u\in H^j
 \quad\Longrightarrow\quad u^2\in H^j
 \quad\Longrightarrow\quad u\in H^{j+1},
\]
where the final implication again follows from
\eqref{eq:stationary-KPI-Fourier}.  Induction gives $u\in H^j$ for every
$j$, and hence $u\in C^\infty(\T^2)$.
\end{proof}

\section*{Acknowledgements}

AR was partially supported by a grant of the Ministry of Research,
Innovation and Digitization, CCCDI - UEFISCDI, project number
ROSUA-2024-0001, within PNCDI IV. The author wishes to thank Nick Gismondi for discussions related to the
presentation of the paper as well as providing an early version of the
decoupling lemma from \cite{KSflex}.

\end{document}